\documentclass[nameyear,11pt]{imsart}
\usepackage{graphicx}
\usepackage{amsmath}
\usepackage{rotating}

\usepackage{multirow}
\usepackage{bbm}
\RequirePackage[OT1]{fontenc}
\RequirePackage{amsthm,amsmath}
\RequirePackage[authoryear,round]{natbib}
\RequirePackage[colorlinks,citecolor=blue,urlcolor=blue]{hyperref}

\usepackage{amsmath}
\usepackage{amsthm}
\usepackage{amssymb}

\startlocaldefs
\theoremstyle{plain}
\newtheorem{theorem}{Theorem}

\newtheorem{condition}{Condition}

\newtheorem{lemma}{Lemma} 
\newtheorem{remark}{Remark}

\newcommand{\bu}{\mbox{\bf u}}
\newcommand{\bv}{\mbox{\bf v}}
\newcommand{\bw}{\mathbf{w}}
\newcommand{\bx}{\mbox{\bf x}}

\newcommand{\bA}{\mbox{\bf A}}

\newcommand{\bD}{\mbox{\bf D}}

\newcommand{\bG}{\mbox{\bf G}}

\newcommand{\bM}{\mbox{\bf M}}

\newcommand{\bR}{\mbox{\bf R}}

\newcommand{\bT}{\mbox{\bf T}}

\newcommand{\bW}{\mbox{\bf W}}
\newcommand{\bX}{\mbox{\bf X}}

\newcommand{\bY}{\mbox{\bf Y}}

\newcommand{\bvarepsilon}{\mbox{\boldmath $\varepsilon$}}

\newcommand{\bbeta}{\mbox{\boldmath $\beta$}}

\newcommand{\bSigma}{\mbox{\boldmath $\Sigma$}}

\endlocaldefs
\usepackage{rotating}

\begin{document}

\title{ \sc  \Large randomized   \Large maximum-contrast selection: \\ subagging for  large-scale regression}
\runtitle{Variable selection for Large Scale regression}

\author{\fnms{Jelena} \snm{Bradic}
\ead[label=e1]{jbradic@ucsd.edu}}

\address{Department of Mathematics\\
University of California San Diego\\
\printead{e1}}

\begin{abstract}

We introduce a very general method for sparse and large-scale variable selection.
The large-scale regression settings is such that   both the number of parameters and the number of samples are  extremely large.
The proposed method is  based on
careful combination of  penalized estimators, each applied to a  random projection of the sample space into a low-dimensional space. 
 In one special case
that we study in detail, the random projections are divided into non-overlapping blocks; each consisting of only a small portion of the original data.
Within each block we select the projection yielding the smallest out-of-sample error. 
Our random   ensemble estimator  then aggregates the results  
according to    
new maximal-contrast voting scheme
to determine the final selected set. Our theoretical results illuminate the effect on performance
of increasing the number of non-overlapping blocks.  
Moreover, we demonstrate that statistical optimality is   retained along with 
  the computational speedup.   The proposed method achieves minimax rates for approximate recovery  over all estimators using the full set of   samples.  Furthermore, our theoretical results allow the number of subsamples  to grow  with the  subsample size and do not require irrepresentable condition.
The estimator is also compared empirically
with several other popular high-dimensional estimators via an extensive simulation
study, which reveals its excellent finite-sample performance.\\

%
%
\end{abstract}

\maketitle



%


\section{Introduction}

 In recent years, statistical analysis of massive data sets has become a subject of increased interest. Modern data sets, often characterized by both high-dimensionality and massive sample sizes,   introduce a range of unique computational challenges, including scalability and storage bottlenecks \citep{FHL13}. Due to the associated storage and computational constraints, the number of data points from the original sample that   can be  processed will be severely limited \citep{FH14}.   In such settings,  it becomes  natural to employ  techniques  based on random subsets of the original data that are of a very small size.

Classical ideas of bootstrap \citep{E79}, aggregation \citep{B96}  and subsampling \citep{P01} are the first that come to mind.   However, the key limitation of the traditional bootstrap is   that resamples  typically have  the same order of magnitude and size as the original data. In large-scale settings,  $n$ and $p$ are often   simultaneously large;    orders of magnitude of $10^6$   or more are not uncommon.
 Hence, even computation of a simple point estimate on the full data set can be an issue.  Moreover, repeated evaluations of the resamples are likely to face the same set of computational and storage challenges as the original problem.

  Subsampling  provides a viable alternative, as it only requires computations   of samples potentially much smaller  than the original data.  However, subsampling is quite sensitive to the choice of the subsample size \citep{S03}:  the smaller the subsample size is,  the larger the variance of the estimates. \cite{BLB12} propose an alternative approach called bag of little bootstraps, which combines bootstrap with smaller order subsampling and demonstrate that for the purpose of estimating continuous parameters, such   method retains bootstrap-like convergence rates. In this work, we consider all of the above mentioned   methods in the context of the support recovery and large-scale regression. 
    
  We propose a  method that operates with a large number of subsamples, each of a very small size, that  can approximately  retrieve the support set  of the regression parameters. We employ an  approach of subsample bootstrap aggregation (subagging) and illustrate  how to aggregate estimators in order to produce stable   approximations of the support set $S$, without requiring the widely used Irrepresentable condition \citep{vGB09}. The key is to compute  estimates, at non-overlapping subsamples using randomized bootstrap and then exploit the randomization to enlarge support sets of those estimates. In this way, we ensure   uniformity of the results, obtained across subsamples and avoid  the potentially harmful effects of smaller order subsampling.

  We consider the simple linear model 
\begin{equation}\label{eq:model}
\mathbf Y =\mathbf X \boldsymbol{\boldsymbol\beta}^* + \boldsymbol\varepsilon,  
\end{equation}
where  $\mathbf Y \in \mathbb{R}^n$, $\mathbf X  \in \mathbb{R}^{n \times p}$ and both $n$ and $p$ are very large numbers, possibly of the order of hundreds of thousands or millions. The noise vector $ \boldsymbol\varepsilon$   is assumed to be i.i.d. and independent of $\mathbf X$.
We assume that the vector $\boldsymbol{\boldsymbol\beta}^*$  is such that its support set $\mbox{supp}(\boldsymbol{\boldsymbol\beta}^*)=S \subseteq \{1,\cdots,p \}$ 
 describes a particular subset of important variables for the
linear model \eqref{eq:model}, for which we then assume $|S|\ll p$. 
  The method presented here  is  quite  intuitive: we partition the dataset of size $n$ randomly into $d$ equal   subsets of size $N \ll \min\{n,p\}$ and
  compute the  penalized  regression estimate  for each of the $i = 1,\dots,d$ subsets independently, using a  careful choice of the regularization parameter $\lambda_N$. The estimates are subsequently averaged using a new multiple voting scheme that shrinks high variability in the estimates. 
  
As computations are done independently of each other, the proposed method  is 
 especially suited  for implementation across   distributed and parallel computing platforms, often used for processing of large-scale data.
  Distributed approaches based on bootstrap aggregation  \citep{B96} have been studied by several authors, including  \cite{BY02} for least-squares algorithms,  \cite{B08} for bootstrapped Lasso algorithm, \cite{MHM10}  for perceptron-based algorithms, \cite{BLB12} for distributed versions of a bootstrap, \cite{ZDW12} for   convex optimization  and  \cite{ZDW13}  for kernel ridge regression algorithm. 
 However,    support recovery guarantees for  subsampled approaches of a smaller scale, have not been analyzed much in the existing literature.  Recent proposals include Bayesian median of subsets  \cite{M14} and related median of subsets 
\cite{D14}. However, both approaches work with subsamples  with sizes proportional to $n$.
 
 The fundamental observations that underpins our proposal is the fact
that naive aggregation of the estimators computed on the small subsamples,
 is not sensible, since most of the subsamples
 will typically destroy the  structure in the data.
Nevertheless, such  estimators have great flexibility in   a choice of  the regularization parameter.
Our theoretical analysis   demonstrates that, even though each   subsampled estimator is computed  only on a very small fraction   of the samples, it is   essential to regularize  such estimators  as if they had 
 all $n$ samples. As a result of that,  each sub-problem is under-regularized, which allows for  the small   bias in   estimation,  but   causes an adverse blow-up in the variance. 
Hence, a simple majority vote of  the  retained
sets can be highly suboptimal; instead, we argue that the voting  should
be chosen  to minimize the worst case risk.
We show that such voting scheme, named maximal-contrast  voting, sufficiently  reduces the variance of under-regularization.
 Our main theorem shows that, while achieving  computational speedup the proposed method retains statistical optimality;  it achieves minimax rates, over all estimators using the set of $n$ samples.

 Our theoretical results are divided into three parts. In the first, we consider   one randomized subsample estimator
and quantify
the difference between the 
selected set of such one subsample estimator 
 and
the Lasso estimator computed on the data of the full size $n$. We then consider
minimax-contrast aggregation of such randomized estimators computed on non-overlapping blocks.
Under a condition implied by the widely-used restricted eigenvalue assumption \citep{BRT09}, we can then control the
difference between the selected set of the aggregated estimator  and the true support set $S$, as a
function of terms that depend on the number of non-overlapping blocks, the pairwise block distance, as well as the size of the each block and terms that diminish when the number of blocks grow. Furthermore, we establish minimax optimality of the proposed method.
 The final part of our theory gives risk bounds on the
 naive bagging of Lasso estimators named Bootstrapped Lasso in \cite{B08}, namely its inability to retrieve support set when the number of blocks grows. 

We briefly introduce the notation used throughout the paper.  Let $I$ and $J$ be two subsets of $\{1,\cdots,n\}$ and $\{1,\cdots,p\}$ respectively. For any set $S$ we use  $|S|$ to denote its cardinality. 
 For any vector $\mathbf x$, let $\mathbf x_{J} \in \mathbb{R}^{|J|}$ denote  its sub vector with coordinates belonging to $J$.   We use notation $\|\bx\|$, $\| \bx \|_1$ and $\| \bx\|_{\infty} $ to denote $L_2$, $L_1$  and $L_\infty$ norms  respectively, of a vector $\bx$.
  For any matrix $\mathbf A$,  let  $\mathbf a^j$ denote its $j$-th column vector, and  let $\mathbf A_{I,J}$ denote a matrix formed by the rows and columns of $\mathbf A$, which belong to the set $I$ and $J$, respectively. Shorthand notation   $\mathbf A _I$ stands for $\mathbf A_{I,\{1,\cdots,p\}}$. We also use $\phi_{\min}(\mathbf A)=\lambda_{\min}(\mathbf A^T \mathbf A) $ and $\phi_{\max}(\mathbf A)=\lambda_{\max}(\mathbf A^T \mathbf A) $ with $\lambda_{\min}$ and $\lambda_{\max}$ denoting minimal and maximal eigenvalue, respectively.   We denote with $\kappa(\mathbf A^T \mathbf A)$ the conditioning number of matrix $\mathbf A$, and define it as  $\phi_{\max}(\mathbf A)/\phi_{\min}(\mathbf A)$. We use $\| \bA\|$ to denote the operator norm of a matrix  $\bA$. Note that   $\lambda_N$ and $\lambda_n$ denote tuning parameters of   Lasso  problems, computed over subsample of size $N$ and full sample of size $n$, respectively.

The rest of the paper is organized as follows.  Section \ref{sec:methodology} describes the proposed weighted small sample bagging of Lasso estimators. Theorems \ref{thm:21} and \ref{thm:31} of  Section 3,  discuss  the difference between the  estimated sparsity sets of the sub-Lasso and the Lasso estimators. Subsection \ref{sec:support} discusses theoretical findings on support recovery properties of the weighted subagging, with the  main result summarized in Theorem \ref{thm:final}.   Theorem \ref{thm:optimal} develops novel minimax rates for approximate support recovery. Subsection \ref{sec:subbagging} outlines the inefficiency of  classical subagging of Lasso estimators ,with  the main results summarized in Theorems \ref{prop:1} and \ref{prop:bootstrap}.  Finally, in Section   \ref{sec:examples}  we provide details and results of the implementation of our method on the simulated    data.

 \section{Randomized Maximum-Contrast Selection}\label{sec:methodology}
 
 We start by describing our sample partitioning  and defining the relevant notation. 
Let $I=\{1, \cdots, n\}$  denote the sample index set,  which we  divide into a group of disjoint subsets each of size $N \ll n$. Although we are mostly interested in subsets of a much smaller order than $n$,  the exposition of the method does not depend on the choice of $N$.  
In this sense, the complete dataset $\{(Y_1,\bX_1),(Y_2,\bX_2)\cdots, (Y_n,\bX_n)\}$  is split evenly and uniformly at random  into many small, disjoint subsets. We  consider $d$  of such subsets, i.e.  
\begin{equation}\label{eq:split}
 (\mathbf Y_{ I_i } ,\mathbf X_{ I_{i} } ) = \bigl\{(Y_{i_1},\bX_{i_1}), \cdots, (Y_{i_N},\bX_{i_N})\bigl\},
 \end{equation}
  for $i=1,\cdots, d$ and allow $d$ to grow with $n$.

Let the weighted sub-Lasso estimator ${\hat {\boldsymbol \beta}}_{i:k} (\lambda_N)$ be defined as  
\begin{equation}\label{eq:smallLasso}
{\hat {\boldsymbol \beta}}_{i:k} (\lambda_N ) =  \arg \min_{\boldsymbol \beta} \biggl\{  \frac{1}{2n} \bigl\| \bD_{ \sqrt{\bw_k}} \mathbf Y_{I_i} - \bD_{ \sqrt{\bw_k}}    \mathbf X_{ I_{i} } \boldsymbol \beta \bigl\|^2 \  + \lambda_N \|\boldsymbol \beta\|_1\biggl\},
\end{equation} 
where $ \bD_{ \bw_k} \in \mathbb R^{N \times N}$ is a diagonal matrix with a vector of random weights, $ \bw_k \in \mathbb R^N$ on its diagonal.   Index $k =1,\dots, K$ enumerates the number of random draws of the weight vector $\bw$. Note that  ${\hat {\boldsymbol \beta}}_{i:k} (\lambda_N ) $ is  computed using only observations within one subset of the data.  Yet, for fixed and discrete choice of $\bw_k$, ${\hat {\boldsymbol \beta}}_{i:k} (\lambda_N)$  
 can be rewritten as the solution to   an $n \times p$   problem
\[
{\hat {\boldsymbol \beta}}_{i:k} (\lambda_N)= \arg \min_{\boldsymbol \beta} \left\{ \frac{1}{2n} \left\|  \tilde {\bY} -\tilde {\bX}_{i :k }\boldsymbol \beta \right\|^2 + \lambda_N \|\boldsymbol \beta\|_1 \right\},
\]
 with $\tilde {\mathbf X}_{i:k} ^T\in \mathbb{R}^{p \times n}$ defined as
 \begin{equation} \label{eq:tildeX}
 \tilde {\mathbf X}_i ^T=[\mathbf X_{i_1}^T,\cdots, \mathbf X_{i_1}^T, \cdots,\mathbf X_{i_N}^T,\cdots, \mathbf X_{i_N}^T ], 
 \end{equation}  
with the rows $\mathbf X_{i_j} \in \mathbb{R}^{1\times p}$ repeated $w_{k,i_j}$ times, each for $j \in 1, \cdots, N$ and $k=1,\cdots, K$.
In other words, ${\hat {\boldsymbol \beta}}_{i:k} (\lambda_N)$  minimizes the    penalization problem, computed using a random projection of the original data, where all the features are kept and  the sample space   is projected into  a low-dimensional space of size $N$. The proposed projection, places a constraint that the number of distinct observations $N$  is     fixed, non-random and much smaller than $n$; 
different from the Efron's traditional bootstrap method which  has a random  number of distinct observations. Furthermore, we require the following condition on the random vector $\bw_k$, for all $k=1,\dots, K$.

 
   \begin{condition}\label{cond:w}
  Let $\mathbf w=( w_1,\dots,w_N)$  be a random vector, such that $w_1,\dots,w_N$   are exchangeable random variables and are independent of the    data $(\bY,\bX)$. Moreover, $w_1,\dots,w_N$  are such that    $\sum_{j=1}^N w_j=n$ and 
 $P(w_1>0)=1$. 
 Additionally,
 for  $\mathbf w _{2 }^2 = \int_0^{e^{2n}} P_w\left(     \max_{ 1 \leq j \leq N} w _{j} \geq   \frac{1}{2}\log t  \right)dt$,  let  $ \mathbf w _{2 } $ satisfy $\log   \mathbf w _{2  }\leq p^{n/N } /n$ almost surely.
\end{condition}


Condition \ref{cond:w} is inspired by   similar conditions imposed for the weighted bootstraps \citep{PW93}.
 However, traditional   weights  do not fit the Condition \ref{cond:w} as they typically follow the Multinomial $\mathcal{M}_n(N,(\frac{1}{n},\cdots, \frac{1}{n}))$ distribution of dimension  $n \geq N$. If we reverse the roles of $n$ and $N$, they can easily be adapted   to the Condition \ref{cond:w},
  with vector $\mathbf w_k$, drawn from the Multinomial distribution $\mathcal{M}_N(n,(\frac{1}{N},\cdots, \frac{1}{N}))$.   
Each  random projection of the original data is data of size $n$,  drawn with replacement on  $N \ll n$ fixed, original data points. \cite{BLB12}  use   these resamples  to   average  continuous estimators. In contrast, our focus is the support recovery in large scale and potentially high dimensional problems.  
Other examples of randomized schemes that satisfy Condition 1, include a Balanced P\'{o}lya urn scheme \citep{A07} and a coupling of Poisson  and  Multinomial distribution; the first  includes a P\'{o}lya urn scheme with a strategy guaranteeing that each ball color is  represented at least once whereas the second scheme includes a randomized game of throwing $n$ balls  into $N$ urns, and repeating the throws until all the balls are in urns.

  For the simplicity in presentation we impose the following  finite moment  condition on the noise vector of the linear model \eqref{eq:model}. Nevertheless,   we believe all the results of the manuscript extend to sub-Gaussian errors.
  
 \begin{condition}\label{cond:e}
 The noise vector $\boldsymbol{\varepsilon}$, \eqref{eq:model}, is such that 
  $E_{\varepsilon}  |\varepsilon_{  i }|^r  \leq r! \sigma^2  c^{r-2} /2 $ for every $r \geq 2$ (and all $i$) and some constants $c<\infty$ and $\sigma^2 <\infty$. 
 \end{condition}

Let the indices $ \{1,\dots,d\}$  be  split into $b$ blocks of  $m$-pairs; in particular,  $b$ and $m$ satisfy   $d=bm$.  The number of subsamples, $d$, is allowed to  grow  with  $n$; it depends on $n$ through $b$. Constant $m$ allows for additional flexibility in estimation. 
Next, we introduce the maximal-contrast selection,   a variant of stability selection, where the subsamples
are drawn as $N \ll n$, $m$ complementary  pairs from  $\{1, \dots, n\}$.  Thus the   procedure outputs,  $b$ of such $m$-pairs index sets $\{I_{mq+1-l}; q=1,\dots, b, l=1,\dots, m\}$, where each $I_{mq+1-l}$ is  subset of  $\{1, \dots, n\}$ of size    $N$ and   $I_{mq+1-1} \cap \dots \cap  I_{mq+1-m} =\emptyset$. A  special case  of the above sets, with  $m=1$  and $m=2$  are  the   sets  of \cite{MB11,D14} and the    sets  of \cite{SS12}, respectively. However, both are based on the subsets that  are of the size proportional to $n$. As we allow $d$ to grow with $n$,  our procedure includes subsamples of a much smaller order. For such cases, simply applying existing methods above  leads to ``select all" or ``select none" vote (see Theorem \ref{prop:1} for a detailed proof).
We show in Theorem \ref{thm:21} below that subsamples of a much smaller order  cause  blow-up in the variance of  the  estimated non-zero sets of   $ {\hat {\boldsymbol \beta}}_{i:k} (\lambda_N)$; that is,  the variance   of the sets 
$$
\hat {S}_i ( \lambda_N, k) = \left\{ 1 \leq j \leq p :  {\hat { \beta}}_{i:k,j} (\lambda_N ) \neq 0 \right\},
$$
for each fixed vector   $\mathbf w_k$, $k=1,\dots, K$, is large.
 Hence,  naive estimators above fail  and we aim to improve them. 
%
%
%
%
%
%
%
 After $K$ random draws of $\bw_{\mathbf k}$,  for each subsample $I_i$, we obtain $K$   sets $\hat {S}_i ( \lambda_N, 1), \dots,$ $ \hat {S}_i ( \lambda_N, K)$. 
Initially, we compute  the selection of variable $j$  in a union of  those   $K$ sets. Next 
we compute a minimax majority vote across $m$-pairs of these unions, each computed   on the non-overlapping  subsamples.
The minimax majority vote  estimator     is defined as
\begin{equation}\label{eq:pi}
\pi_j^*(\lambda_N,b,m,K)=  \frac{1}{\sqrt{b}(\sqrt{b}+1)} \sum_{q=1}^b     \mathbbm{1}\left\{ j \in \cap_{l=1}^m \cup_{ k=1}^K\hat {S}_{mq+1 - l}(\lambda_N,k)\right\} + \frac{1}{2} \frac{1}{\sqrt{b}+1} .
\end{equation}
Whenever possible we suppress  $b,m,K$ from  $\pi_j^*(\lambda_N,b,m,K)$ and write   $\pi_j^*(\lambda_N)$ for short. 
This estimator arises as  a solution to the   minimax  estimation of a mean   of a Bernoulli trial  (see Chapter 5,  Example 1.7 of  \cite{L98} for more details). For small values of $b \leq 40$ and  all possible values of   $p_j$,   $\pi_j^*(\lambda_N)$   has a larger bias, (bounded with $2^{-1} (1+\sqrt{b})^{-1}$ in absolute value) but  a smaller variance compared to the  maximum likelihood estimator.

%
%
   
%
%

%
 %

Next, we define  the  weighted maximum-contrast  subagging estimator $\widehat {{\boldsymbol\beta}}^a(\lambda_N) $   to be 
 $\widehat {{\boldsymbol\beta}}^a_j(\lambda_N)  =0$ if $j \notin \hat {S}_\tau$,
 $$
\hat {S}_\tau (\lambda_N) = \left \{1 \leq j \leq p :  \pi_j^*(\lambda_N)  \geq \tau \right\},
$$
with appropriately chosen thresholding parameter $\tau$. Otherwise, when $j \in \hat {S}_\tau$
\begin{equation}\label{eq:def}
\widehat {{\boldsymbol\beta}}^a_j(\lambda_N)  = K^{-1}    d^{-1} \sum_{ i =1}^d  \sum_{k=1}^K {\hat {\beta}}_{i:k,j} (\lambda_N) .
\end{equation}
Observe that the aggregated estimator $\widehat {{\boldsymbol\beta}}^a_j(\lambda_N)  $ is a random and  dependent on the number of random draws $K$, the number of blocks $b$  and the distance between blocks $m$.
We emphasize  here the additional flexibility afforded by not pre-specifying the voting
threshold $\tau$ to be 1/2; or the size of the pairs $m$ to be 1.

The threshold value $\tau$ is a tuning parameter whose influence is very small. For practical
values in the range of, say,  $\tau \in (0.25,0.75)$, results tend to be very similar. The choice of $m$ is more intricate.
Large values of $m$ lead to the smaller number   of both false and true positives. However,  in such cases, they simultaneously lead to smaller values of $b$. Such small values are especially suited for the estimator \eqref{eq:def} as they cause greater value of the bias; which in turn leads to an improvement of the selection. For such choices of $m$, we advocate a smaller value of the tuning parameter $\tau$. For smaller values of $m$ and hence large values of $b$, the estimator \eqref{eq:def} resembles majority vote estimator. However, small $m$  produces a  large number of false positives. However,  large values of $\tau$  can reduce this bias in selection. 

For $m=1,K=1, b=d$ and $\tau =1/2$, the proposed estimator   takes on a form of  the ``majority vote" estimator, so that $\hat {S}_\tau (\lambda_N)$ consists of all $j$ that are included in at least half of the sets $\hat {S}_i(\lambda_N,1)$, $i=1,\dots, d$. Furthermore,  for $\tau =1/4 (1 +  \frac{1}{\sqrt{d}+1} ) >1/4$, $\hat {S}_\tau (\lambda_N)$ consists of all $j$ that are included in at least a quarter of the sets $\hat {S}_i(\lambda_N,1)$, $i=1,\dots, d$.
 Moreover,  for the case of $m=2,K=2, b=\lfloor{d/2}\rfloor$ and  $\tau =1/2  $,   $\hat {S}_\tau (\lambda_N)$ consists of all $j$ that are included in at least half of the sets 
 $$\{ \hat {S}_i(\lambda_N,1) \cup \hat {S}_i(\lambda_N,2) \} \cap \{ \hat {S}_{i+\lfloor{d/2}\rfloor}(\lambda_N,1) \cup \hat {S}_{i+\lfloor{d/2}\rfloor}(\lambda_N,2) \} , \qquad i=1,\dots, \lfloor{d/2}\rfloor. $$
%

   We   observe that maximum-contrast selection allows for more structure in the search of the support set. Moreover, $\widehat {{\boldsymbol\beta}}^a_j(\lambda_N) $ takes the form of a randomized   subagging  \citep{BY02}  and a delete$-(n-N)-$jackknife \citep{E79} with additional discovery control, suited for a smaller order  subsamples and a very large size of the deleted set (with $(n-N)/n \to 1$), respectively.

If we apply Lasso  to the full data    the
estimated  set will converge to the true support set only if stringent conditions are imposed. Next section shows that the  median-contrast selection does not rely on such heavy assumptions, leading to substantial gains in not only computational
time but also feature selection performance. 
The intuition for this gain is that  a large proportion of the subsets  of the data does not preserve  the structure of  the full data; this in turn  has a sizable influence on the selected sets.
By taking biased estimator with controlled variance and additional randomization steps, we obtain a more uniform  model that is not largely influenced by these  altered structure of the subsets. As large-scale data typically contain outliers and data contamination,
this is a substantial practical advantage.

  \section{Theoretical Properties}

  Without   loss of generality, from this point on we   assume that the columns of $\mathbf X$ have unit $l_2$ norm.   Let $a \geq 1$ be a constant. 
Because of the  regularization, the following cone set is important:
$$
\mathbb{C}(a, S)=\left\{ \mathbf v \in \mathbb{R}^p:  \mathbf v \neq 0, \| \mathbf v_{S^c}\|_1 \leq a \| \mathbf v_{S}\|_1 \right\}.
$$
The Restricted Eigenvalue $\zeta_N$ of the matrix $[\sqrt \bw_k \bX]$ for a vector $\bw_k \in \mathbb{R}^N$ and a  design matrix of a partitioned data $\bX \in \mathbb{R}^{N \times p}$ is defined in Condition \ref{cond:re}. 

\begin{condition} \label{cond:re}
Let $\mathbf w_k \in \mathbb{R}^N$ be vectors of weights.
For a  matrix $ \tilde {\mathbf X}_{i:k}  \in \mathbb{R}^{n \times p}$ as in \eqref{eq:tildeX}, and $a >1$,
there exists a positive number $\zeta_N>0$ such that for all $|S|\leq 2s$
\[
\zeta_N =\min_{\mathbf v \in \mathbb{C}(a,S) } \frac{\| \tilde {\mathbf X} \mathbf v\|_2}{\sqrt{n} \| \mathbf v_S\|_2}  .
\]
 \end{condition}
 The restricted eigenvalues $\zeta_N$ are variants of the restricted eigenvalue introduced in  \cite{BRT09} and of the compatibility condition in \cite{vGB09}. 
In the display above
 $
\|\tilde{\mathbf X} \mathbf v\|_2^2 =\sum_{l \in I_{i}} w_{k,l}(\mathbf X_l \mathbf v )^2 
 $
 for any  realization $k$ of the weight vector $\mathbf w_k$. We use
subscript $N$ in $\zeta_N$ to denote  the number of distinct rows  of the matrix $  \tilde {\mathbf X}_{i:k}$. 
As long as the tail of the distribution of $X$ decays exponentially, column correlation is not too large and $N \geq  |S| \log p$,   the  condition   above holds  with high probability
  \citep{RWY11}.  

It is well established that the  necessary condition for the exact recovery of penalized methods consist of the Irrepresentable condition
\citep{vGB09}. 
 \begin{condition}[{\bf IR$(n)$}]\label{cond:ir}
The design matrix $\mathbf X \in \mathbb{R}^{n \times p}$ satisfies {\bf IR$(n)$} condition if    the following    holds
$
\left\| \mathbf X_{ {  S}^c} \mathbf X_ {  S} (\mathbf X_{   S}^T\mathbf X_{  S})^{-1} \mbox{sign}( {\boldsymbol\beta} ^*)\right\|_{\infty} < 1.
$
\end{condition}
 
The symbol { \bf IR$(n)$} denotes the   number of rows  in the design matrix $\mathbf X$, with $n$ denoting the sample size.   
 When{ \bf IR$(N)$}  is assumed to hold for every subsampled dataset, then
    each of the penalized  regression estimates $\hat \bbeta_{i:k}(\lambda_N )$  would recover  the set $S$ with high probability -- a commonly  used condition in an existing literature \citep{D14}. A far more interesting scenario is to allow deviations from{ \bf IR$(N)$} condition. 
With  subsamples  of a very small size, $N \ll p$ and $N \ll n$,  imposing   such conditions on each subsample   is not realistic; most of the  subsamples will not preserve the original data structure. 
 Instead, 
we  examine  different properties of the   design  and the sub design matrices without imposing { \bf IR$(N)$} condition.
  
\begin{lemma} \label{lem:7}
Let Condition \ref{cond:re} hold.
Let $A \subseteq \{1,\dots ,p\}$ that is of the size $|A|=r \leq 2s $ and let $I\subseteq \{1,\dots,n\}$ be of the size $|I|=N$.   Assume that   the matrix $\mathbf X_{I } \in \mathbb{R}^{N\times p}$ satisfies Condition \ref{cond:re} with $a=3$.
Then, for every $j \notin A$ ,
\begin{eqnarray} \label{eq:step0300}
\|  \mathbf X_{I ,  A}^T \mathbf X_{I ,j}\|_2 
&\leq& \zeta_N^{-3 },
\\
 \label{eq:step01}
\left\| (\mathbf X_{ I,A}^T \mathbf X_{ I, A} )^{-1}   \mathbf X_{  I,A}^T \mathbf X_{I,j}\right\| _1^2 &\leq & r/\zeta_N^2.
\end{eqnarray}  
\end{lemma}

Observe that, under Condition \ref{cond:re}, a result of Lemma 1 
        is much weaker than the Irrepresentable condition{ \bf IR$(N)$},  
  as   $r \geq \zeta_N^2$ and is possibly dependent on $n$.

The following lemma characterizes  uniform deviation of a randomized weighted sum of negatively correlated random variables  and is critical in obtaining finite sample properties; see Theorem 3 below. To the best of our knowledge there is no similar result in the
literature.
We use $\langle \cdot, \cdot\rangle_n$ to denote the empirical inner product, i.e. $\langle\bu,\bv\rangle_n = \frac{1}{n}  \bu ^T \bv $, for two vectors $\bu, \bv \in \mathbb{R}^p$.
 \begin{lemma} \label{lem:a}
 Let $\mathbf w=(w_1,\cdots,w_N)$  be a vector of weights that satisfies   Condition  \ref{cond:w}.  Let the error $\boldsymbol\varepsilon$ satisfy Condition \ref{cond:e}.
  Then,  all $i=1,\dots, d$ there exists  a  sequence $u_n$ of non-negative real numbers,   such that
 \[
\mathbb P \biggl(    \Bigl| \left \langle \bvarepsilon_{I_i} ,   \bD_{  \mathbf w} \bX_{I_i}  \right \rangle_n \Bigl| > u_n \biggl) \leq  
\exp \left\{ N \log   \mathbf w _{2}  -\frac{n^2 u_n^2 }{ 2\sigma^2   N \|  \bX_{ I_i} \|_{\infty,2}  +2  n  c u_n \|  \bX_{ I_i} \|_{\infty,\infty}} \right\},
 \]
with $ \mathbf w _{2 } $ defined in Condition \ref{cond:w}. In display above, 
 $\|  \bX_{ I_i} \|_{\infty,2}:= \max{} \Bigl \{ { X}_{I_ij}^2    :  I_i \subset \{1,\cdots, n \},  \Bigl.$ $\Bigl. |I_i|=N  , 1 \leq j \leq p\Bigl\}$
  and
  $ \|  \bX_{ I_i} \|_{\infty,\infty}:= \max \Bigl \{|{ X} |_{I_i j}   :  I_i \subset \{1,\cdots, n \}, |I_i|=N, 1\leq j \leq p \Bigl \}$.
 
 \end{lemma}

Lemma \ref{lem:a} represents a uniform, nonasymptotic exponential inequality
for a sum of negatively correlated random variables. Compared with other concentration  inequalities \citep{K8,P14}, it  holds for continuous random variables which are negatively correlated. Moreover, its independence of dimensionality $p$ and dependence on $n$, proves
to be invaluable for  the variable selection properties.

 \subsection{Support sets}

When dealing with ``imperfect learners", an important question is how to aggregate information over all learners. In order to  examine the performance of $\widehat {{\boldsymbol\beta}}^a_j(\lambda_N) $, we first study ${\hat {\bbeta}}_{i:k} (\lambda_N)$ and its ability to recover the support set $S$, for each fixed $k =1,\dots,K$.

\begin{theorem}\label{thm:21}
Let $k=1,\dots, K$ be fixed.
Let $\mathbf w_k$ be a vector of weights satisfying  Condition  \ref{cond:w}.  Let the error $\boldsymbol\varepsilon$ satisfy Condition \ref{cond:e}. Assume that   the matrix $\mathbf X_{I_i}$ satisfies Condition \ref{cond:re} with $a=3$ and  $i=1,\dots, d$.
Then,  for every $k=1,\cdots, K$,
with a choice of  $c_1 \sigma \sqrt{\log p/n} \leq \lambda_N \leq c_2 \sigma \sqrt{\log p/n}$,  $c_1>0,c_2>1$, 
%
there exists a constant $c >1$ such that,
$P(S\subseteq \hat S_i(\lambda_N,k)) \geq 1-p^{1-c}.$ Moreover, constants $c_1,c_2$ and $c$ do not depend on $n,p$ or $s$.
\end{theorem}

The proof of Theorem \ref{thm:21} is in the Appendix.    One
of the technical challenges is to control the inner product between $\bX_{I_i^c}$ and the out-of-sample fit $\bY_{I_i^c} - \hat \bY_{I_i^c}$ with $ \hat \bY_{I_i^c} = \bX_{I_i^c}\hat {\boldsymbol \beta}_{i:k}(\lambda_N )$. Theorem \ref{thm:21}  has immediate consequences.

 \begin{remark}
Theorem \ref{thm:21}
 shows that the sub-Lasso estimator, computed only on a small fraction of the data, requires   regularization  comparable to  that of the complete data -- proportional to $\sqrt{\log p/n}$. 
As $\mathbb{E}[ \bw_k ]= (n/N,\dots, n/N)$, existing work on Lasso estimator \citep{BRT09} implies that a regularization of   $\sqrt{\log p/N}$ suffices for the purpose of variable selection. Instead, we obtain  a regularization of much smaller order to be necessary.     
\end{remark}
In a different context,  \cite{ZDW13} similarly show that  subsampled kernel smoothing with ridge penalty, requires  regularization proportional  to that of the original data, albeit for prediction purposes.



Next, we compare the selection set of one sub-Lasso estimator and the Lasso estimator computed on  the original data  defined as 
\begin{equation}\label{eq:lasso}
\hat {\boldsymbol\beta}(\lambda_n) = \arg \min_{\boldsymbol\beta \in \mathbb{R}^p} \left\{\frac{1}{2n} \| \mathbf Y-\mathbf X \boldsymbol\beta \|^2 + \lambda_n \|\boldsymbol\beta\|_1\right\}.
\end{equation}
Whenever Condition \ref{cond:ir} holds, the Lasso estimator is an oracle estimator; it recovers the correct set with high probability.
We show 
that  whenever the Irrepresentable Condition {\bf IR$(n)$} holds,
 and all $$\lambda _N \in (0,{c}\lambda_n-\frac{c_1}{\sqrt{n}}) \cup ( n c \lambda_n +c^2c_1,  \infty),$$ for some universal constants $c >1$ and  $c_1>0$, the   estimated sets  of ${\hat {\boldsymbol \beta}}_{i:k} (\lambda_N )$ and ${\hat {\boldsymbol \beta}}   (\lambda_n )$ are  nested, with probability converging to $1$. 
 Theorem \ref{thm:21}, in addition to  Condition \ref{cond:ir}, guarantees that  $\hat S \subseteq \hat S_i$, i.e. there exists a constant $c>1$ such that for all $\lambda_N \leq {c}\lambda_n$, 
 \begin{equation}\label{eq:weak}
P(\hat{\boldsymbol\beta} _{i:k,j}(\lambda_N ) \neq  0) 
\geq
P(\hat{\boldsymbol\beta}_{j}(\lambda_n) \neq  0), \ \ \ j \in  S ,  \ \ \ 
\end{equation}
 whereas Theorem \ref{thm:31} and Condition \ref{cond:ir} guarantee $\hat S \supseteq \hat S_i$, i.e. for $\lambda_N \geq nc \lambda_n$,
\begin{equation}\label{eq:weak1}
P(\hat{\boldsymbol\beta} _{i:k,j}(\lambda_N ) \neq  0) 
\leq 
P(\hat{\boldsymbol\beta}_{j}(\lambda_n) \neq  0), \ \ \ j \in   S  . \ \ \ 
\end{equation}
We obtain the following result.

 \begin{theorem}\label{thm:31}
 Let $k=1,\dots, K$ be fixed.
Let $\mathbf w_k$ be a vector of weights that satisfies   Condition  \ref{cond:w}.  Let the error $\boldsymbol\varepsilon$ satisfy Condition \ref{cond:e}.   Assume that matrix $\bf X$ satisfies Condition \ref{cond:ir}
 and that there exists a positive constant $C'$ such that $\lambda_{\max}\left(\frac{1}{n} \mathbf X^T \mathbf X\right) /{\zeta_n} \leq C' $.
  Assume that   the matrix $\mathbf X_{I_i}$ satisfies Condition \ref{cond:re} with $a=3$ and  $i=1,\dots, d$.
Then,  
for  all 
\begin{eqnarray*}  
 {\lambda_N}{} 
&\geq&   (n+1)  \lambda_n  +\frac{\lambda_n C'^{3/2}  s^{3/2}  }{  {\zeta_n^2 \zeta_{N-n}}}
 +
 \frac{\sqrt{2}\lambda_n C'^{3/2}s^{3/2}}{  \zeta_n \zeta_{N-n}^3} 
 + 2 \sigma   \sqrt{ {(n-N) \log p}{ }} +2 \sigma   \sqrt{ { 4 C'   \log p /n}{    }},
\end{eqnarray*} 
    there exists  a constant $c>1$ such that for every $k=1,\cdots, K$,
 $P(S \supseteq \hat S_i(\lambda_N,k)) \geq 1-p^{1-c}$.

\end{theorem}
 
 Hence, the optimal $\lambda_N$, according to the Theorem \ref{thm:31}, is  of the same order as $\sigma \sqrt{\log p}$. In contrast, the optimal $\lambda_N$, according to the Theorem \ref{thm:21},  is of the order of  $\sigma\sqrt{\log p/n}$. Hence, there does not exists a universal choice of $\lambda_N$, for a sub-Lasso estimator to have exact support recovery.

\begin{remark}
This result provides    novel insights into finite sample  equivalents of the asymptotic bias of subagging and ``majority voting" as presented in \cite{BY02}. There the authors    suggests that there is asymptotically  zero probability that the subsampled Lasso (sub-Lasso) and Lasso  estimators  have  the same sparsity patterns (see Theorem 3.3  therein). Theorems \ref{thm:21} and \ref{thm:31} show   show that the  given probability is equal to zero. We show more details in Section 3.4.
\end{remark}
The immediate consequence   is that naive ``multiple vote" estimate does not have a single choice of $\lambda_N$  that achieves variable selection. Details are presented in Theorem \ref{prop:1}.

\subsection{Efficiency and optimality}\label{sec:support}

Next, we state the main theorems on the finite sample variable selection property and optimality of the proposed   minimax voting scheme \eqref{eq:def}.

The bootstrap scheme is said to be efficient if it mimics the ``behavior" of the original data. In this context, ``behavior" can mean many things, like inheriting rates of convergence in the central limit theorem or in the large deviations. 
In this work, we concentrate on two types of efficiency. 
The first considers ``{\it exact sparse recovery}", where $\hat S=S$, for a candidate estimator $\hat S$, with high probability. The second considers ``{\it approximate sparse recovery}", where the focus is on establishing the following two properties simultaneously 
\[
 P \left(  S \subseteq \hat S \right) \geq 1-\delta, \mbox{ for  }\delta \in (0,1) \qquad   E|\hat S \cap S^c|\leq \varepsilon s /p, \mbox{ for }  \varepsilon \in (0,1).
 \]
Typically, $\delta$ and $\varepsilon$ in the above considerations are numbers which are very close to zero.
%
%
%
%

\begin{theorem}\label{thm:final}
Let $\mathbf w_k$ be a vector of weights satisfying  Condition  \ref{cond:w} for all  $k=1,\dots, K$.  Let the error $\boldsymbol\varepsilon$ satisffy Condition \ref{cond:e}. Assume that   the matrix $\mathbf X_{I_i}$ satisfies Condition \ref{cond:re} with $a=3$ and  $i=1,\dots, d$.
Let  the distribution of 
$$\left\{\mathbbm{1} \left(j \in  \cup_{k=1}^K \hat S_i(\lambda_N, k) \right), s+1 \leq j \leq p \right\},$$ be exchangeable for  $\lambda_N$  as in \eqref{eq:lambda11} and all $i=1,\dots,d$. Moreover, let $\| \mathbf X_{I_i}\|_{\infty,\infty} \leq  c_2 \sqrt{ N / \log p  }$, for some constant $c_2>1$ and all $1\leq i \leq d $.
Then, there exists  a  bounded, positive, universal constant $c'>1$ such that, for 
\begin{equation}\label{eq:lambda11}
\frac{1}{c'} \sqrt{\log p/n}\leq \lambda_N \leq c' \sqrt{\log p/n},
\end{equation}
the following holds
\begin{equation}
P(S \subseteq \hat S_\tau(\lambda_N)) \geq 1-p^{1-c'},
\end{equation}
$ \mbox{and}$  
\begin{equation}\label{eq:3.7}
 \sup_{\tau > \frac{1}{2(1+\sqrt{b})}} E\left[ |S^c \cap \hat S_\tau(\lambda_N) |\right]\leq 2  C\frac{\sqrt{b}}{1+\sqrt{b}} \frac{K^m s^m}{N^{m} p^{m-1} \zeta_n^{2m}},
 \end{equation}
  for all  $mb=d$, $dN \leq n$  and a constant $C$ that depends only on $m$.
\end{theorem}

The proof of Theorem \ref{thm:final} is in the Appendix.  
  An attractive feature
of this result is its generality: no {\bf IR$(N)$} assumptions are placed. Yet, it shows that maximal-contrast is able to approximately retrieve the support set $S$. Moreover, it does so for all thresholds $\tau$ that are bigger than $1/4$. 
The upper bound on the number of false positives, see \eqref{eq:3.7}, is a function that decereases with both larger $b$ and larger $m$ as long as $K \leq N p/s$.
We note  that the weights $\bw_k$ are instrumental in obtaining   the theoretical guarantees above. Weights $\bw_k$ are chosen to minimize the out-of-sample prediction error, which in turn allows for a  simultaneous control of false positives and false negatives.

 \begin{remark}
The results of Section 3.1 imply that aggregation of the sub-Lasso estimators is needed; no single sub-Lasso can recover the correct set $S$.  Theorem \ref{thm:final} has immediate implications. As long as  the  sub-Lasso estimators  have the probability of support recovery  $1-\gamma$, with $\gamma \in (0,1/2)$ -- are slightly better than the random guessing -- the proposed    maximal-contrast selection,  \eqref{eq:def}, guarantees that this probability  is very close to 1. 
  \end{remark}
The study of the  proposed  estimator  \eqref{eq:def} is made difficult by  the fact that  we  are aggregating unstable estimators. 
The proof is based on  allowing    deviations in  optimal   convergence rates   for    each of the small subsamples   and finding the smallest such deviation that  allows good variable selection properties of the aggregated estimator.
  The proof is further made challenging as    Condition \ref{cond:w} does not require  $ \sum_{l=1}^N(w_l-1)^2/N$ to converge to a positive constant $c$, independent of $n$, a condition that is usually imposed for weighted bootstrap samples \citep{BGZ97}.  This condition is violated in  our setting, as $c \to 0$.

In the above result, we require an 
 assumption of an exchangeability of indices   $j$ over the  sets $ \cup_{k=1}^K \hat S_i(\lambda_N, k)$   for a   few, small values of $\lambda_N$. 
A similar assumption appeared   in \cite{MB11}. For simplicity of presentation, we do not provide details of the relaxation  of this condition, although we believe  the method of \cite{SS12}   applies.  
 It is not hard to show that a finite sequence of  Bernoulli random variables
$X_1, X_2, . . . , X_n$ is exchangeable if any permutation of the $X_i$'s has the
same distribution as the original sequence. Bayes' postulate, in our context,    implies    
 that the partial sums  have discrete uniform distribution, i.e.
\[
 \mathbb P \left( \sum_{j \in S^c,j=1}^p  \mathbbm{1} \left(j \in  \cup_{k=1}^K \hat S_i(\lambda_N, k) \right)  = h \right) = 1/(h+1),
\]
for all $1 \leq h \leq p-s$. 
Hence, to check the exchangeability assumption we can perform any goodness of fit test, by splitting each subset $ (\mathbf Y_{I_i},\mathbf X_{I_i})$ into two parts, with the first being the training and  the second being the testing set. 
Moreover,
\cite{B13} provide a number of distributional characterizations of exchangeable Bernoulli random variables. In our context, 
\[
 \mathbb P  \left(j \in  \cup_{k=1}^K \hat S_i(\lambda_N, k) \right)   = p_{j,K}(\lambda_N),
\]
if $p_{s+1,K} (\lambda_N)\geq p_{s+2,K }(\lambda_N) \geq \cdots \geq p_{p,K}(\lambda_N) \geq 0$ (where indexing is  up to a permutation) for some $\lambda_N$ \eqref{eq:lambda11}, then the condition of exchangeability is satisfied. In turn, this implies that any column correlation in $\bX$ that decays with the distance between the columns, satisfies our setting. For an approximations to the distribution of a finite, exchangeable
Bernoulli sequence by a mixture of i.i.d. random variables under appropriate
conditions, see \cite{DF80}.

 In the remainder of the section we focus on the optimality of the results obtained in Theorem \ref{thm:final}. Our goal is to provide conditions under which an $\varepsilon$-approximation of the support set $S$ is impossible -- that is, under which there exist two constants $c>0$
 and $c_\varepsilon>0$, such that 
 \begin{itemize}
 \item[(a)] $\inf_{J} \sup_{\boldsymbol\beta  } P_{\boldsymbol\beta}\left( S \not\subseteq J\right) \geq c >0$, and 
 \item[(b)] $\inf_J\sup_{\boldsymbol\beta  } E_{\boldsymbol\beta}|S^c \cap J| \geq c_{\varepsilon}>0$,
 \end{itemize}
 where the infimum is taken over all possible estimators $J$ of $S$ and supremum is taken over all $s$-sparse vectors $\bbeta$. This setup is different from the classical minimax lower bounds obtained for the purposes of sparse estimation \citep{L11} or variable selection \citep{CD12}. Those   are  primarily concerned with $\inf_{J} \sup_{\boldsymbol\beta } P_{\boldsymbol\beta}\left( S \neq J\right)$.    In contrast,  our setting  allows $\varepsilon\%$ false discovery control, in which case $c$  is a  strictly positive number and $c_{\varepsilon} = \varepsilon p /s$.
   To that end,
 let 
 $$\rho_n =  \max \left\{ \frac{ E (\mathbf X_l \mathbf v )^2   }{\sqrt{n} \| \mathbf v\|_2}  :  |S| \leq s, \| \mathbf v\|_0 =s\right\},$$   and  let $ \mathbb{B}_0(s)$ be  the  $l_0$ ball which corresponds to the set of vectors $\bbeta$ with at most
$s$ non-zero elements, that is
\[
 \mathbb{B}_0(s) := \biggl\{ \bbeta \in \mathbb{R}^p: \sum_{j=1}^p 1\{ \bbeta_j \neq 0\} \leq s\biggl\} .
\]

 \begin{theorem}\label{thm:optimal}
 Let Condition \ref{cond:re} hold for the  design matrix $\mathbf X \in \mathbb{R}^{n \times p}$.
 If there exist  two constants $c'>1$ and $1>c_\varepsilon >0$ such that 
\[
 {s^2 \log{ {p}/{s^2}}}{ } >  n\rho_n + {\log 2}  + (1-c') \log p  
\]
\[
  s \log(p/s) +s \log 2 > \max\{ \log c_\varepsilon + n \rho_n , \log c_\varepsilon -\log (1-\sqrt{n \rho_n/2})\}
  \]
 then 
 $$ {\mbox{\rm (a) }} \inf_{ J} \sup_{\boldsymbol\beta \in \mathbb{B}_0(s)} P_{\boldsymbol\beta}\left( S \not\subseteq J\right) \geq p^{1-c'},$$
 $$ { \mbox{\rm (b) } }\inf_J\sup_{\boldsymbol\beta \in  \mathbb{B}_0(s)} E_{\boldsymbol\beta}|S^c \cap J| \geq c_{\varepsilon}>0.$$ 
 \end{theorem}

 \begin{remark}
 A particular example  of   Theorem \ref{thm:optimal}, is the following set of conditions under which $5\%$ false discovery control is impossible. Namely, as long as $n > 2(1+5\%)^2$ and
 $$(s+1) \log p/s  + s\log 2 + \log 20> \max\left\{ n\rho_n,  \log 1/ ( (1-\sqrt{n \rho_n/2}) ) \right\}, $$ then  
 $$\inf_J\sup_{\boldsymbol\beta \in \mathbb{B}_0(s)} E_{\boldsymbol\beta}|S^c \cap J| \geq  5\%\frac{s}{p}.$$
  In light of  Theorem \ref{thm:final} and $mb=d$,  the previous result implies that as long as the number of  subsamples $d$ satisfies the  bound
 \[
 d < 2 C \frac{n}{s} \exp\left\{(1-\frac{1}{m}) \frac{n}{s} \right\} \frac{\zeta_n^2}{K} \left( \frac{p}{s\sqrt{m}}\right)^{1/m}  \]
   the proposed estimator \eqref{eq:def} is efficient for support recovery with $5\%$ false discovery control. In other words, \eqref{eq:def} achieves optimal approximate recovery as    the upper bound in Theorem \ref{thm:final} cannot be further improved.    
   Moreover,  observe that the upper bound on $d$ converges to $\infty$ as $n,p \to \infty$
   \end{remark}

Apart from the minimax optimality,   we show that  the support recovery of the proposed method is more efficient, in comparison to the support recovery of the   Lasso estimator in the classical sense,  of possessing smaller variance in estimation.
\begin{theorem}\label{thm:var}
Let $\pi_j(\lambda_n) =P( \hat \bbeta_j (\lambda_n)\neq 0)$ with $\hat \bbeta $ as the Lasso estimator    \eqref{eq:lasso}.
Then, for small values of $b$  and  for all $\lambda_n > \frac{1}{c} \sqrt{\log p/n}$    and all $j \in S$,
$
\mbox{var}\left(\pi_j^*(\lambda_n)\right) \leq \mbox{var}\left(\pi_j(\lambda_n) \right)  .
$
\end{theorem} 

\subsection{Inefficiency of subagging}\label{sec:subbagging}

Although  properties of the Lasso estimator are well understood (see \cite{B07}, \cite{L08}, \cite{BRT09}), there hasn't been much theoretical support for the variable selection and prediction properties of subagging  of Lasso estimators \citep{B08}, or its close variants \citep{MI13} for the cases of $p \ge n$.
\cite{B08} raised  a question of whether  bagging can be used for  retrieving the support set $S$ and provided a conjecture of its failure for $p \geq n$.
The following theorem   proves this conjecture and shows inefficiency of classical  bagged estimator in such support recovery problems. 

We denote the bootstrap averaging estimator,  bagged Lasso estimator  \citep{B08},   as 
\begin{equation}\label{eq:subagging}
{\hat {\boldsymbol \beta}}^b(\lambda_n^1) = \frac{1}{d} \sum_{i=1}^d \arg \min_{\boldsymbol \beta} \biggl\{  \frac{1}{2N}  \sum_{l \in  I_i } (Y_l-X_l^T\boldsymbol \beta )^2 + \lambda_n^1 \|\boldsymbol \beta\|_1 \biggl\}.
\end{equation}

\begin{theorem}\label{prop:1}
Assume that matrix $\bf X$ satisfies Condition \ref{cond:ir}
 and that there exists a positive constant $C'$ such that $\lambda_{\max}\left(\frac{1}{n} \mathbf X^T \mathbf X\right) /{\zeta_n} \leq C' $.
  Assume that   the matrix $\mathbf X_{I_i}$ satisfies Condition \ref{cond:re} with $a=3$ and  $i=1,\dots, d$. Let  $p >n$. Then, for  every $\lambda_n$, there exist no sequence of $\lambda_n^1$   with $N \leq n$, such that the bagged estimator ${\hat {\boldsymbol \beta}}^b(\lambda_n^1)$ achieves  exact or approximate sparse recovery   of the  Lasso estimator $\hat{\boldsymbol\beta} (\lambda_n)$. 
\end{theorem}

Note that  Theorem \ref{prop:1}  holds without  imposing {\bf IR}($N$) condition on every  subsampled dataset.
 This   result is in line with the expected disadvantages of subagging  \citep{BY02} but is somehow surprising. The result also   extends  \cite{B08}.
 Furthermore, it demonstrates
  that even the bagged Lasso estimator (with $N=n/2$) for diverging $n$ and $p$ ($p > n$)  fails to exactly recover the sparsity set $S$, when at least one of the {\bf IR}($N$) conditions is violated. 
However,  for $p \leq n$ and   $N = n$, the situation reverses; the bagged estimator recovers exactly the true sparsity set $S$ with probability very close to $1$  (proof is presented in the Appendix).   
Although our primary concern is the support recovery, we present a result on the prediction error of the subagged estimator \eqref{eq:subagging}.   Moreover, it seems to be new and of interest in and of its own.
We show that, only    the choices of $N \geq n/4$,  allow subagging to attain  the same predictive properties as guaranteed by the Lasso \eqref{eq:lasso}.

\begin{theorem}\label{prop:bootstrap}
Let $\mathbf X_b$ be a matrix obtained by the bootstrap sampling of the  rows of the design matrix $\mathbf X$.
Assume that  $\mathbf X_b$ satisfies Condition \ref{cond:re},  with $a=3$.  Then the estimator $\hat{\boldsymbol\beta}^b(\lambda_n^1)$, defined on the bootstrapped sample of size $n/k$ with $k \leq 4$ and   with  $\lambda_n^1=k\sigma\sqrt{\log p/n}$, attains the  $l_2$ prediction error of  the  Lasso estimator $\hat{\boldsymbol\beta} (\lambda_n)$. 
\end{theorem}

\section{Numerical Studies}\label{sec:examples}

In this section, we compare statistical finite sample properties of the proposed weighted maximal-contrast subagging with that of the existing methods via extensive simulation study.
For   experiments that include both Gaussian models and  skewed  models, we study the variable selection and convergence properties of of the maximum-contrast selection. 
We consider  traditional subbaging, stability selection  and the Lasso  estimator for comparison.  The  threshold level $\tau$ is taken to be $1,0.8,0.5$ and $0.3$.
 
{\it The choice of $\lambda_N$}.
We note that the traditional  cross-validation  or information criterion methods,  computed  within each of the blocks $I_i$,   fail   to provide the   $\lambda_N$ of the   scale $\sqrt{\log p/n}$ critical for good support recovery.  We propose an alternative,   block cross-validation statistics, which sets $\lambda$ to be the argument minimum of 
\begin{eqnarray*}
  \sum_{k,k{'}=1;i,i{'}=1}^{K,d}    \frac{ \bigl \| \bD_{\sqrt{\mathbf w_k}}   \mathbf Y_{I_i} - \bD_{\sqrt{\mathbf w_k}}\mathbf X_{I_i} {\boldsymbol\beta} _{i' :k'}(\lambda ) \bigl \|^2 }{  (1-\hat{s}'/n)^2} - \frac{\bigl \| \bD_{\sqrt{\mathbf w_k' }}\mathbf Y_{ i{'}_n} -\bD_{\sqrt{\mathbf w_k'}} \mathbf X_{ I_{i^{'}}} {\boldsymbol\beta} _{i:k}(\lambda ) \bigl \|^2 }{  (1-\hat s/n)^2  },
 \end{eqnarray*}
 with ${\boldsymbol\beta} _{i':k'}(\lambda) $ being the estimator \eqref{eq:def} computed using only the data $I_{m'}^{b'}$ and with $k'$-th realization of the weight vector $\mathbf w_k$.
 The statistic above  measures discrepancy in  out-of sample prediction error between different sub-samples. Proof  of Theorem \ref{thm:final} shows guarantees of such a procedure.

{\it{The choice of $b$ and $m$}.}
We want to choose $b$ and  $m$  in order to obtain the best possible performance
bounds as described in Section 3 above. Smaller values reduce  the computational cost.  However, $b$ and $m$ cannot be chosen small  simultaneously as $d=bm$ is fixed. Moreover, the
the performance bounds rely on  the definition \eqref{eq:def}, whose strength increases as $b$  increases.
Hence  we want to choose $b$ large enough that this  estimator has good averaging properties and such that $m$ is not too large.
In Section \ref{sec:b} below we see that the random  maximum-contrast method is quite robust to the
choice of $b$, and that small values of $b$ suffice. We fix $b=3$ in Sections 4.1 and 4.2 and perform sensitivity analysis in Section 4.3.

{\it{The choice of $K$}.}
In order to minimize the second term in the bound in Theorem 3, we should choose $K$ to
be as small as possible. However, the  first term in the bound in Theorem 3, weakens  with  smaller $K$.
Moreover,  the computational cost of the
estimator  scales linearly with $K$.  In practice, however, we found
that   the proposed method was
robust to the choice of $K \leq m$. In all of our simulations, we vary $K=1,3, 10$ and recommend a universal choice $K=3$.

\subsection{Linear Gaussian Model}

In this example, we measure performance of the proposed method through true positive (TP) and  false positive (FP)   error control. In order  to compare performance  we consider a simple linear model
\[
Y_i = \bX_i^T \bbeta^* + \varepsilon _i
\]
with  the varying sample size  $n$. The $\varepsilon_i$s are  generated as independent, standard  Gaussian components.
We fix    the  feature size to be $p=4000$. We perform three different studies within this model. For each given $n$, $\bbeta^*$, $\bX$ and $\bY$ we generate  100 testing datasets independently and report average metric of interest. We study both the variable section properties of the proposed method and the prediction properties.  We set the $\lambda_N$s according to the block cross-validation statistics proposed above.

We compare the proposed method with two variable selection procedures: 
\begin{itemize}
\item The traditional Lasso estimator. We apply the  Lasso estimator directly to the original   data of size $n$. We set the $\lambda_n$ by the  self-tuned cross-validation statistic.
\item 
The traditional subagging with the majority vote. We set each of the $\lambda_N$s according to the traditional cross-validation statistics as advocated by \cite{B08}.
\end{itemize}

 We consider three models:
\begin{itemize}
\item[-] {\it Model 1:} The design matrix  is such that each $\mathbf X_i$  has a  multivariate normal distribution independently, with  a Toeplitz covariance matrix $\bSigma$ such that $\bSigma_{ij}= 0.5^{|i-j|}$. $\bbeta^*$  is a sparse vector in which the first $10$ elements are equal to $0.7$ and the rest are equal to $0$. 
\item[-] {\it Model 2:} The design matrix is such that,  columns $\mathbf x^j$ satisfy 
\begin{align*}
\mathbf x^j& = 0.2* \bG + \mathbf Z \qquad j=1,\dots, 10
\\
\mathbf x^j& = 0.2* \bG + \mathbf W \qquad j=11,\dots, 20
\\
\mathbf x^j& = 0.2* \bG + \mathbf T \qquad j=21,\dots, 30
\\
\mathbf x^j& = \mathbf R \qquad j=30,\dots, p
\end{align*}
where $\bG, \mathbf {Z}, \bW,\bT,\bR$ are independent standard normal vectors. In this model,
we have three equally important groups, and within each group there are ten members.
$\bbeta^*$  is a sparse vector in which the first $30$ elements are equal to $3$ and the rest are equal to $0$. 

\item[-] {\it Model 3:}  Each row of the design matrix, $\mathbf X_i$,  has a  multivariate normal distribution independently, with  the covariance matrix $\bSigma$. $\bSigma$ is a block-diagonal matrix. Its upper left $60\times60$ block is an
equal correlation matrix with $\rho=0.4$; and its lower right $(p-60)\times(p-60)$ block
is an identity matrix.   $\bbeta^*$  is a sparse vector in which the first $30$ elements are equal to $3$ and the rest are equal to $0$.  In this model,  the cross-correlation between the noise and signal columns is non-zero. 

\end{itemize}



 We show results with varying subsample size  
 $$n=10000, \qquad N =n^{\gamma}, \qquad \gamma=\{0.2,0.4,0.45,0.47,0.485,0.5,0.52,0.55,0.6,0.65,0.75,0.8\}.$$ 
 For $\gamma \in(0.2,0.485)$, $N < s \log p $.
 We set  $b$ and $m$ to be $3$ and $3$, respectively and 
  explore  a number of different choices of the parameters $\tau, K$:
  \begin{itemize}
  \item[-] {\it $K=1$} with $\tau = 1,0.8,0.5$. The choice of $K=1,\tau=0.5$ corresponds to a majority vote scheme, where each estimator is randomized and the decision is made according to the most liberal vote. In contrast, $K=1,\tau=1$ corresponds to a conservative vote (the worst case bound), where only elements that have appeared in all of the sets are kept.
  \item[-] {\it $K=3$} with $\tau = 1,0.8,0.5.$ The choice of $K=3,\tau=0.5$ corresponds to a majority vote scheme, where each estimator is randomized three times and the only elements kept are those  that have appeared in the most of  unions, of the three estimated sets.
  \item[-] {\it $K=10$} with $\tau = 1,0.8,0.5.$ The choice of $K=10,\tau=0.5$ corresponds to the most liberal majority voting scheme, as all elements that have appeared in half of the unions, of the ten estimated sets, are kept.
  \end{itemize}


Obtained results are summarized in the Figure   \ref{fig:fig1a}.
In all three models, the maximum-contrast selection outperforms the  traditional majority voting of the subagging estimator in terms of  variable selection. In Models 1-3, the Lasso estimator acts as an oracle estimator. Nevertheless,  the maximum-contrast selection  is still competitive and achieves perfect recovery in Model 1, for  all the subsamples larger than $\sqrt{n}=100$.  For the case of Model 2, the maximum-contrast estimator requires larger subsample sizes in order to gain perfect recovery; size of $n^{0.65}\approx400$ are needed. In Model 2, the group structure favors methods with larger number of random draws $\mathbf w_k$; $K=10$ achieves the best performance. We see that the performance of the subagging is unsatisfactory no matter of the subsample size.  In Model 3,  the design matrix has a small correlation between the noise and signal variables. Nevertheless, this  correlation does not  affect the performance of the maximum-contrast selection estimator. In all three models, the number of falsely selected components of the maximum-contrast selection estimator is negligible. 

%
%

 \begin{figure}[htbp]
 \centering
 \includegraphics[width=7.5cm,height=4cm]{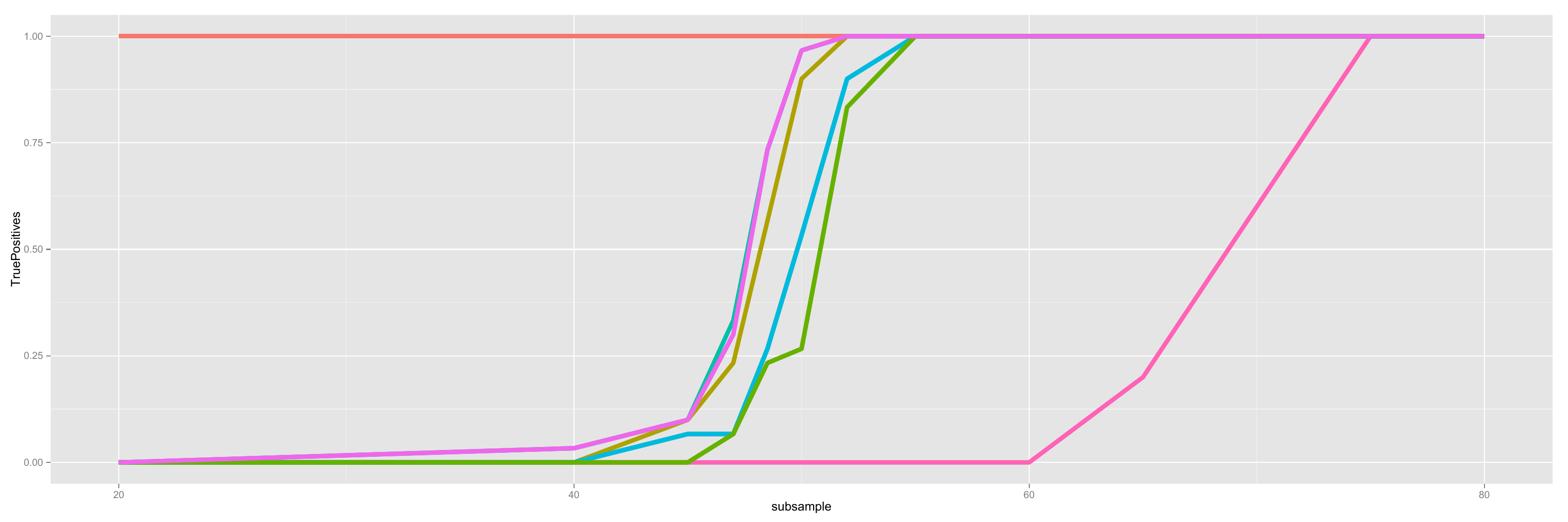}
  \includegraphics[width=7.5cm,height=4cm]{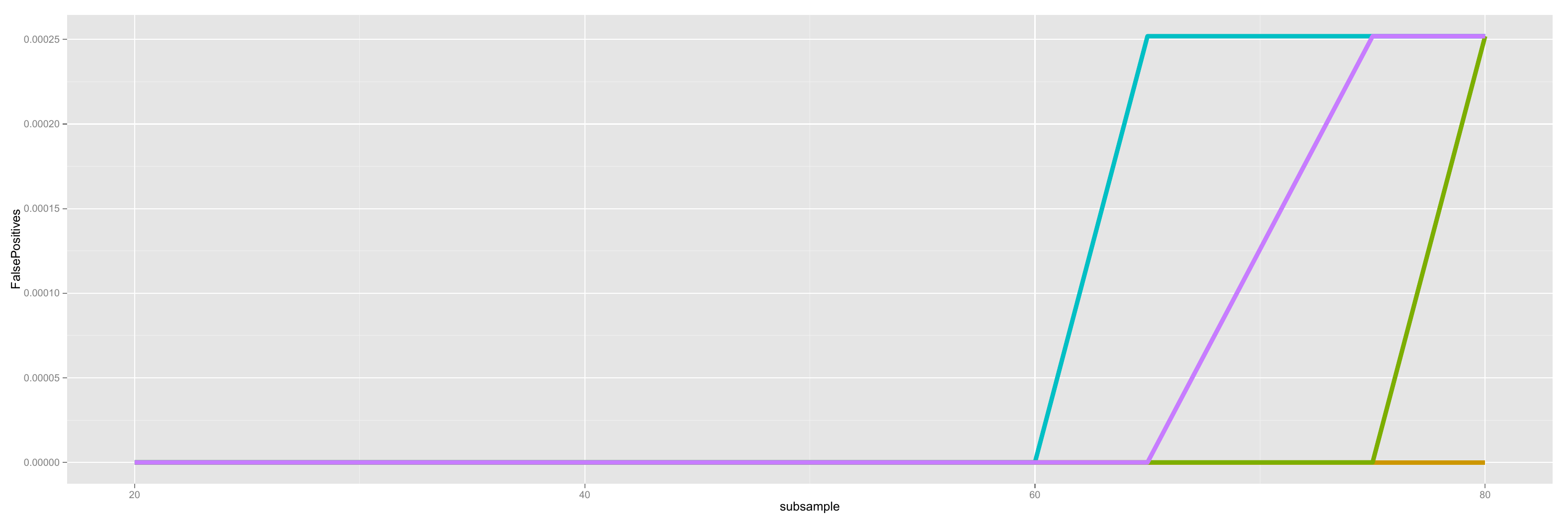}
 \includegraphics[width=7.5cm,height=4cm]{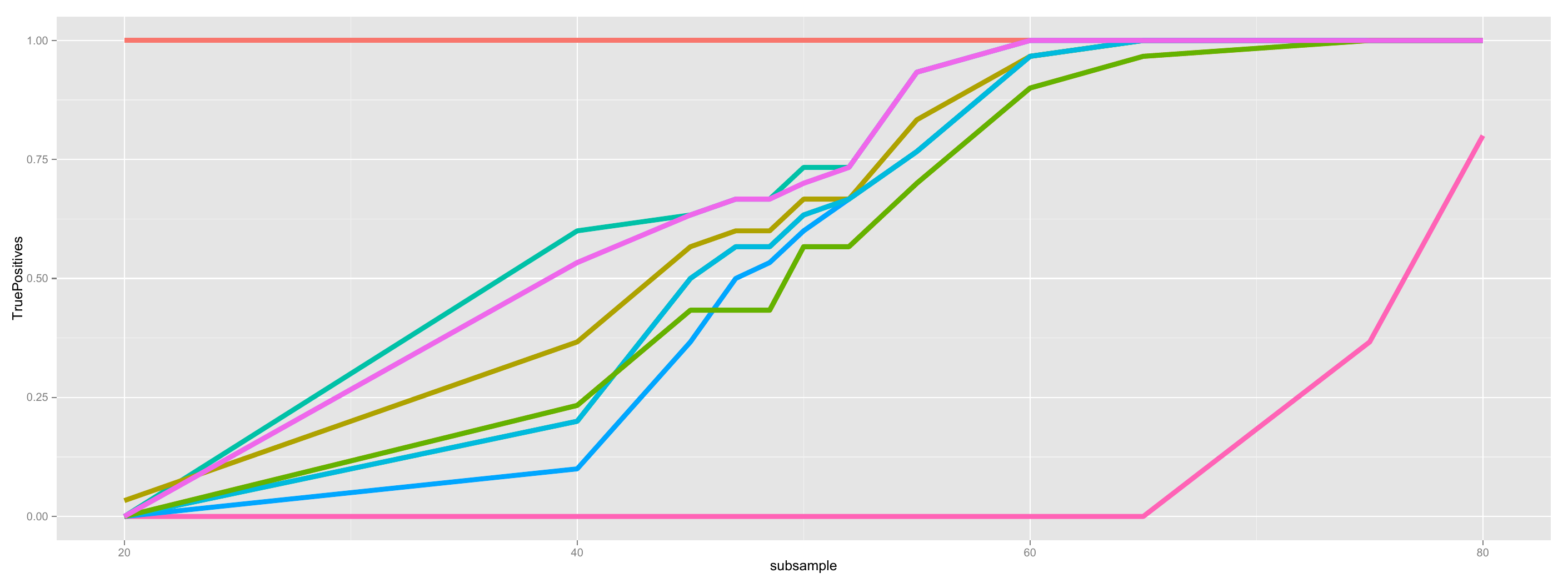}
 \includegraphics[width=7.5cm,height=4cm]{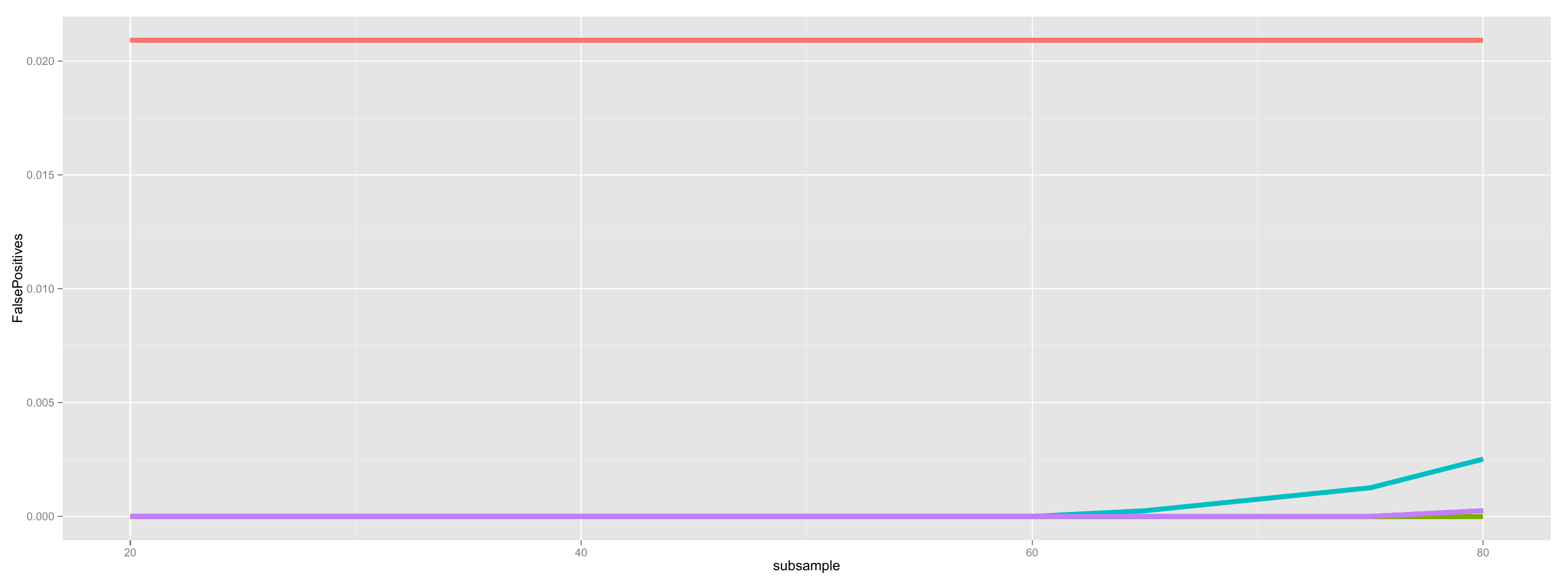}
 \includegraphics[width=7.5cm,height=4cm]{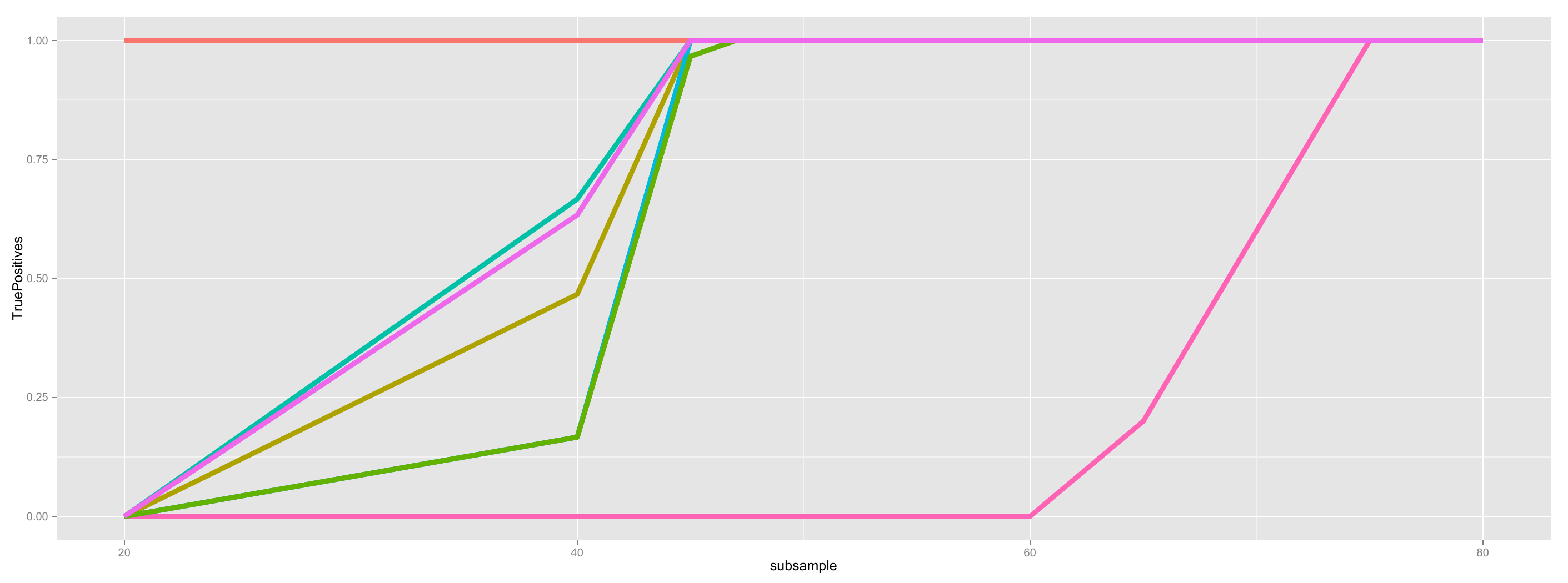}
  \includegraphics[width=7.5cm,height=4cm]{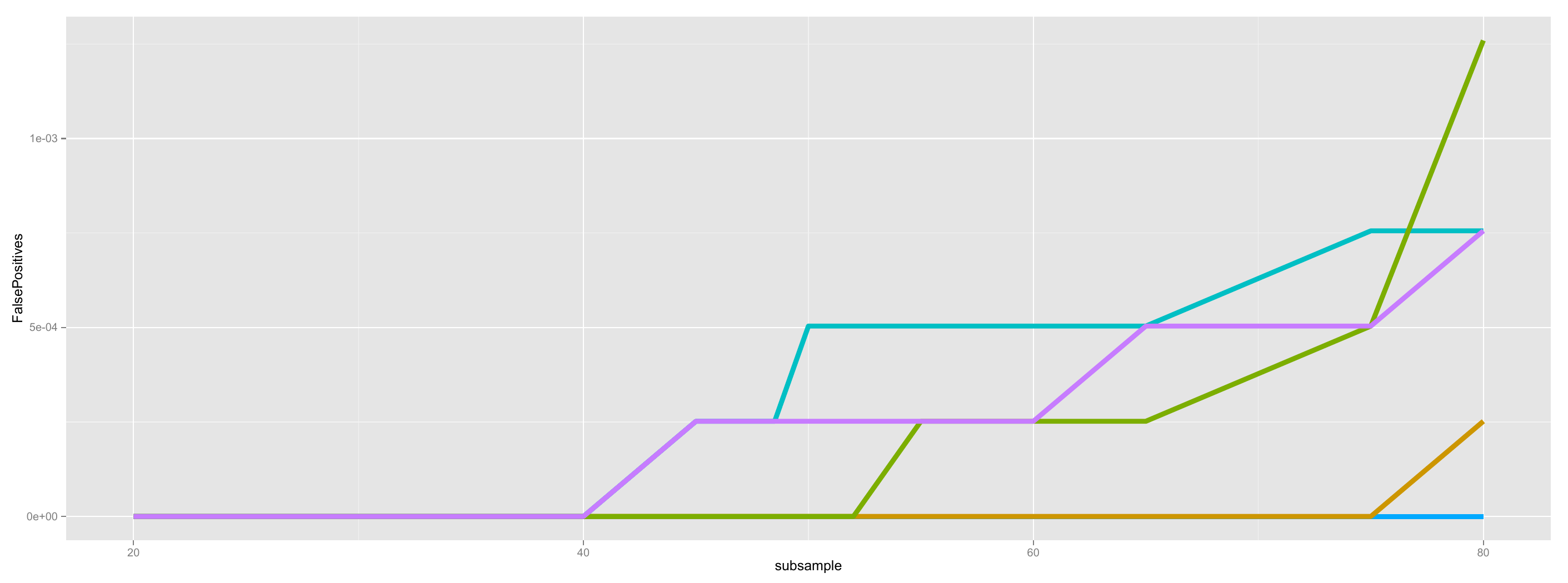}
  \caption{X-axes represents the  subsample size $N$ on a logarithmic scale and the Y-axes represents the average selection probability of the signal variables (the first column) and of the noise variables (the second column).  The first row is according to Model 1, the second according to Model 2 and the third according to Model 3.  Colors represent: red-Lasso, pink-subagging and minimax\eqref{eq:def} with  brown-- $K=1,\tau=1$;
dark-brown--$K=1,\tau=1/2$;
dark-green--$K=1,\tau=4/5$;
  green--$K=10,\tau=1$;
green-blue--$K=10,\tau=1/2$;
 light-blue--$K=10,\tau=4/5$;
 blue--$K=3,\tau=1$ ;
purple--$K=3,\tau=1/2$; 
dark-pink--$K=3,\tau=4/5$. 
   }\label{fig:fig1a}
  \end{figure}

We complement the results presented in the Figure \ref{fig:fig1a}  with the results presented in Tables \ref{tab:1} and \ref{tab:2}. There we fix $N = 251$ and show average of the true  and false positive results for varying choices of the tuning parameter $\lambda_N$.

\begin{sidewaystable}[htbp]
\centering
\begin{tabular}{c||ccc|c|c|c|c|c|c|c|c|c  } 
\hline
\multicolumn{2}{ c} {$\lambda_N$}  &  $0.001$ & $0.005$ & $0.008$ & $0.01$ & $0.03$ &0.05 & 0.1 & 0.3 & 0.9 & 1.5& 1.9\\
\hline
\hline
\multirow{2}{*}{Lasso} & TP &  1 & 1&1&1&1&1&1&1&1 &1&1  \\
 & FP & 0.94 & 0.79&0.66&0.59&0.12&0.005 &0 & 0 &0&0&0   \\ \hline
\multirow{2}{*}{Subagging} & TP &  1 & 1&1&1&1&1&1&1&1 & 0.93&0.83 \\
 & FP &   0&0&0&0&0&0&0&0&0&0 &0\\ \hline
\multirow{2}{*}{WMCB, $K=1$, $\tau=1$} & TP &  1&1&1&1&1&1&1&1&1 &  0.76&0.46\\
 & FP &  0.82&0.01&0.008&0.007&0 &0&0&0&0&0&0\\
\hline
\multirow{2}{*}{WMCB, $K=3$, $\tau=1$} & TP &  1&1&1&1&1&1&1&1&1&0.90& 0.53\\
 & FP &  0.99 & 0.02& 0.01&0.01&0 &0&0&0&0&0&0 \\
\hline
\multirow{2}{*}{WMCB, $K=10$, $\tau=1$} & TP &  1&1&1&1&1&1&1&1&1&0.96& 0.6\\
 & FP &   0.99 & 0.05& 0.02&0.01&0.001 &0&0&0&0&0&0  \\
\hline
\multirow{2}{*}{WMCB, $K=1$, $\tau=0.5$} & TP &  1 & 1&1&1&1&1&1&1&1 &1&0.93\\
 & FP &    0.99 & 0.24& 0.20&0.16&0.09 &0.07&0.04&0.004&0&0&0\\
\hline
\multirow{2}{*}{WMCB, $K=3$, $\tau=0.5$} & TP & 1 & 1&1&1&1&1&1&1&1 &1&1 \\
 & FP &  0.99 & 0.47& 0.37&0.29&0.13 &0.08&0.04&0.003&0&0&0 \\
\hline
\multirow{2}{*}{WMCB, $K=10$, $\tau=0.5$} & TP &1 & 1&1&1&1&1&1&1&1 &1&1  \\
 & FP &   0.99 & 0.6& 0.47&0.36&0.16 &0.11&0.06&0.004&0&0&0 \\
\hline  
\end{tabular}
\caption{Average number of true (TP) and false positives (FP) over $100$ replication of a Gaussian linear model with Toeplitz design with $\rho=0.2$ and $\bbeta^*=(1^T,0^T)^T$. The model has $p=1000$ and $n=10000$. Subagging and Weighted Maximum-Contrast selection use $N =251$ and $b=3$, $m=3$.} \label{tab:1}
\end{sidewaystable}

\begin{sidewaystable}[htbp]
\begin{tabular}{c||ccc|c|c|c|c|c|c|c|c|c  } 
\hline
\multicolumn{2}{ c} {$\lambda_N$}  &  $0.001$ & $0.005$ & $0.008$ & $0.01$ & $0.03$ &0.05 & 0.1 & 0.3 & 0.9 & 1.5& 1.9\\
\hline
\hline
\multirow{2}{*}{Lasso} & TP &  1 & 1&1&1&1&1&1&1&1 &0&0  \\
 & FP & 0.88 & 0.59&0.42&0.29&0.002&0 &0 & 0 &0&0&0   \\ \hline
\multirow{2}{*}{Subagging} & TP &  0.76 & 0.86&0.86&0.86&0.86&0.86&0.83&0.63&0& 0&0 \\
 & FP &   0&0&0&0&0&0&0&0&0&0 &0\\ \hline
\multirow{2}{*}{WMCB, $K=1$, $\tau=1$} & TP &  1&1&0.96&0.96&0.96&0.96&0.93&0.4&0 &  0&0\\
 & FP &  0.01&0.003&0.001&0.001&0 &0&0&0&0&0&0\\
\hline
\multirow{2}{*}{WMCB, $K=3$, $\tau=1$} & TP & 1&1&0.96&0.96&0.96&0.96&0.96&0.46&0 &  0&0\\
 & FP &  0.02 & 0.001& 0&0&0 &0&0&0&0&0&0 \\
\hline
\multirow{2}{*}{WMCB, $K=10$, $\tau=1$} & TP &  1&1&0.97&0.98&0.98&0.99&0.97&0.97&0.97&0.97& 0.5\\
 & FP &    0.04 & 0.003& 0.002&0.001&0.002 &0.001&0&0&0&0&0  \\
\hline
\multirow{2}{*}{WMCB, $K=1$, $\tau=0.5$} & TP &  1 & 1&1&1&1&1&1&0.9&0 &0&0\\
 & FP &    0.25& 0.14& 0.11&0.12&0.06 &0.04&0.003&0.001&0&0&0\\
\hline
\multirow{2}{*}{WMCB, $K=3$, $\tau=0.5$} & TP &1 & 1&1&1&1&1&1&0.96&0 &0&0 \\
 & FP &  0.46 & 0.18& 0.13&0.13&0.06 &0.05&0.01&0&0&0&0 \\
\hline
\multirow{2}{*}{WMCB, $K=10$, $\tau=0.5$} & TP &1 & 1&1&1&1&1&1&1&1 &0&0  \\
 & FP &   0.56 & 0.22& 0.16&0.17&0.08 &0.09&0.07&0.02&0&0&0 \\
\hline  
\end{tabular}
\caption{Average number of true (TP) and false positives (FP) over $100$ replication of a Gaussian linear model with Toeplitz design with $\rho=0.2$ and $\bbeta^*=(0.3^T,0^T)^T$. The model has $p=1000$ and $n=10000$. Subagging and Weighted Maximum-Contrast selection use $N =251$ and $b=3$, $m=3$.}\label{tab:2}
\end{sidewaystable}

\subsection{Skewed Linear Model}

In the next example, we consider two settings that depart  from simple dependency  and normality assumptions.  We consider the same simple linear model as above. Parameter choices are  made by   the same choices as in the Models 1-3 above: $n=10000$, $p=4000$.

 We consider three additional models:
\begin{itemize}
\item[-] {\it Model 4:} The design matrix    has a  multivariate Student distribution, with the covariance matrix $\bSigma$ from the Model 3 above. $\bbeta^*$  is a sparse vector in which the first $30$ elements are equal to $3$ and the rest are equal to $0$. The $\varepsilon_i$s are  generated as independent, standard  Gaussian components.
\item[-] {\it Model 5:} The design matrix is such that, $X_{ij}$'s  are drawn from the Beta distribution with parameters $1+10\frac{j-1}{p-1}$ and $2$ independently for $j=1,\dots,p$. The regression model  is  calibrated to have mean zero, but the distribution of $\bX_i$ is skewed with skewness that varies across dimensions.
$\bbeta^*$  is a sparse vector in which the first $10$ elements are equal to $1,2, 3, \dots, 10$ and the rest are equal to $0$. The $\varepsilon_i$s are  generated as independent, standard  Gaussian components.

\item[-] {\it Model 6:}  The design matrix  is such that each $\mathbf X_i$  has a  multivariate normal distribution independently, with  the  covariance matrix $\bSigma$.
$\bSigma$ is the covariance
matrix of a fractional white noise process, where the difference
parameter $l=0.2$. In other words, $\bSigma$ has a polynomial off-diagonal
decay, $\bSigma_{ij} = \mathcal{O}(|i-j|^{1-2 l})$.
 $\bbeta^*$  is a sparse vector in which the first $10$ elements are equal to $1$ and the rest are equal to $0$. The $\varepsilon_i$s are  generated as independent   components with Student $t$ distribution with  $3$ degrees of freedom.

\end{itemize}

Figures \ref{fig:fig1e} shows results for the regression setting under the three  above models.
 The Lasso estimator is no longer  the oracle estimator for all the models; it is an oracle just for the Model 4 where we observe same patters as in Models 1-3.
Model 5 is particularly difficult, as column correlations depend through dimensionality $p$.  Also,
 the maximum-contrast estimator is better than the Lasso (for example, $0.8$ versus $0.3$, of the average true positive rate in Model 5), which illustrates the advantage
of  maximum-contrast estimator for correlated or skewed settings.
Model 6 is a challenging model, and we observe that the Lasso estimator fails to recover the correct set of variables. Nevertheless,  the maximum-contrast estimator achieves a perfect recovery for all subsample larger than $\sqrt{n}=100$, which illustrates the additional advantage
of  maximum-contrast estimator for correlated designs.
 Additionally, note that weight vector $\bw_k$ is improving the estimation and convergence rate of the introduced method, as the subagging estimators underperform in all of the  Models 4-6.

 \begin{figure}[htbp]
 \centering
 \includegraphics[width=7.5cm,height=4cm]{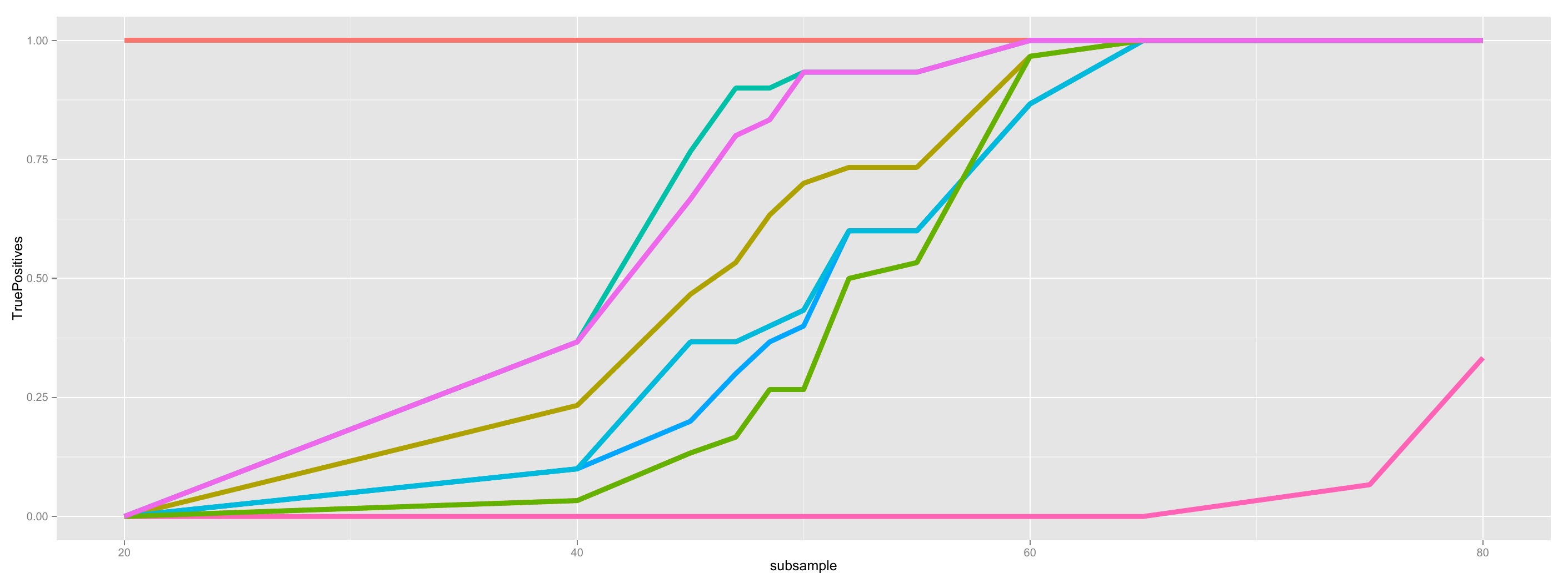}
  \includegraphics[width=7.5cm,height=4cm]{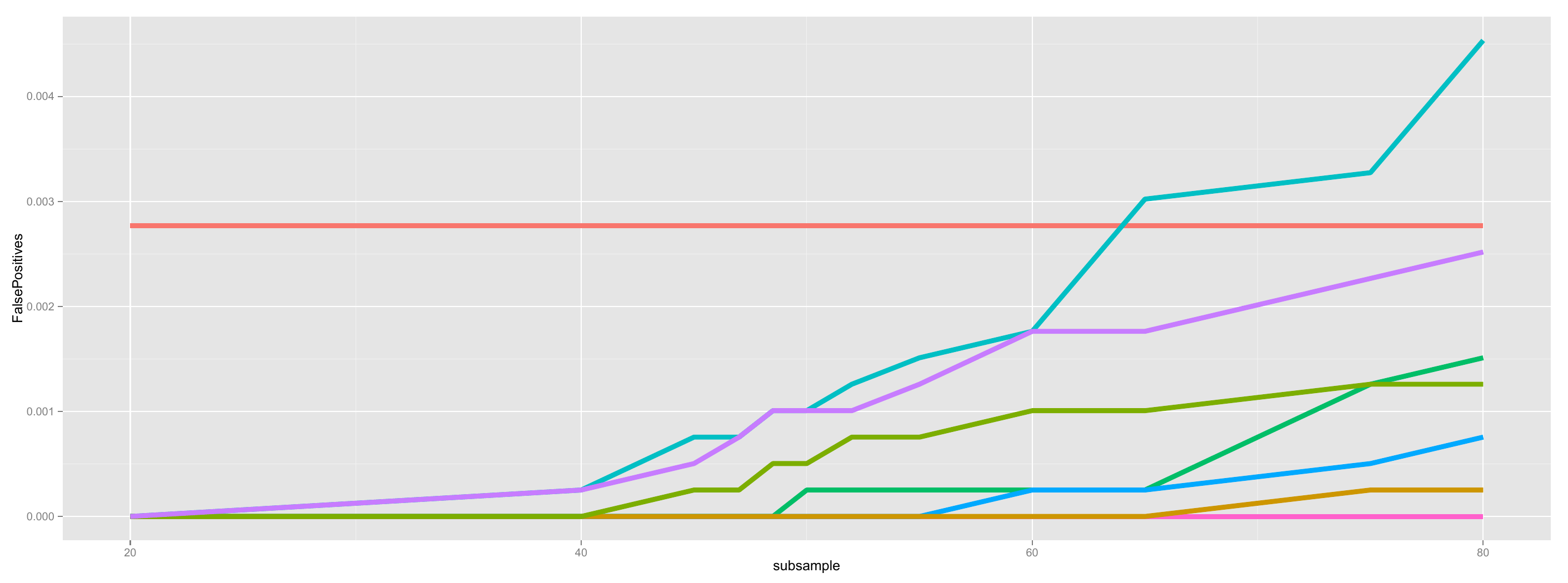}
 \includegraphics[width=7.5cm,height=4cm]{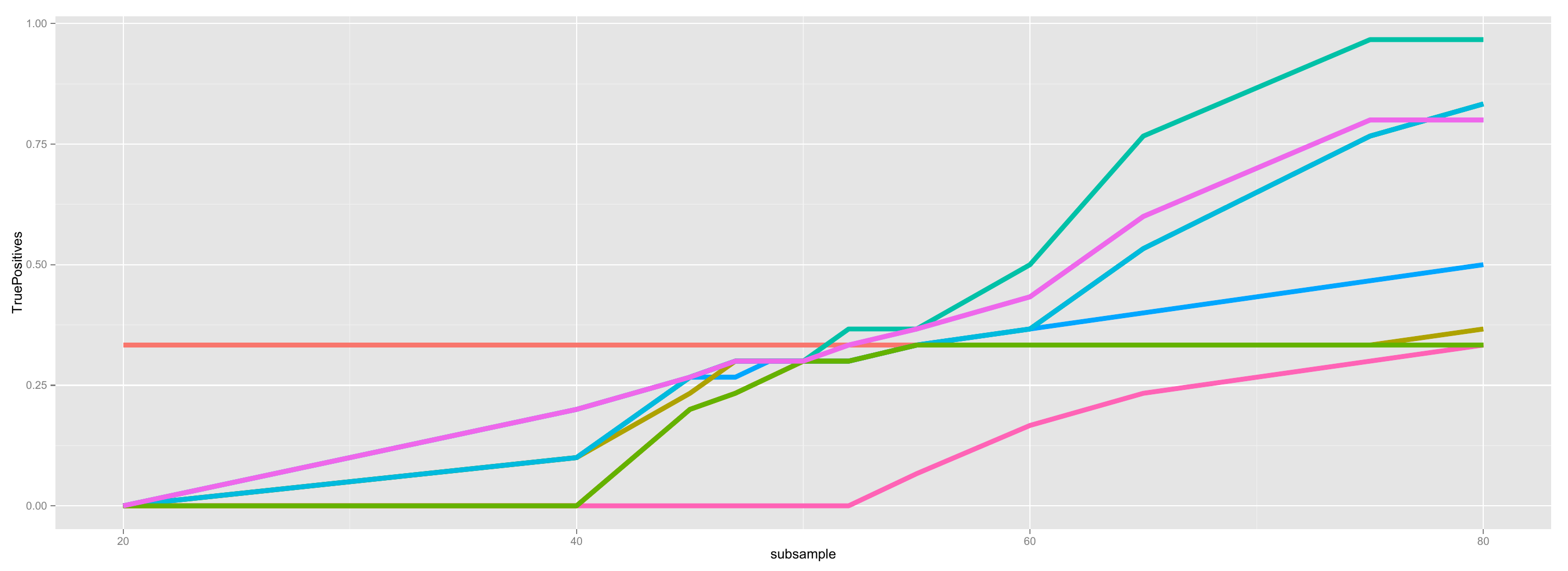}
 \includegraphics[width=7.5cm,height=4cm]{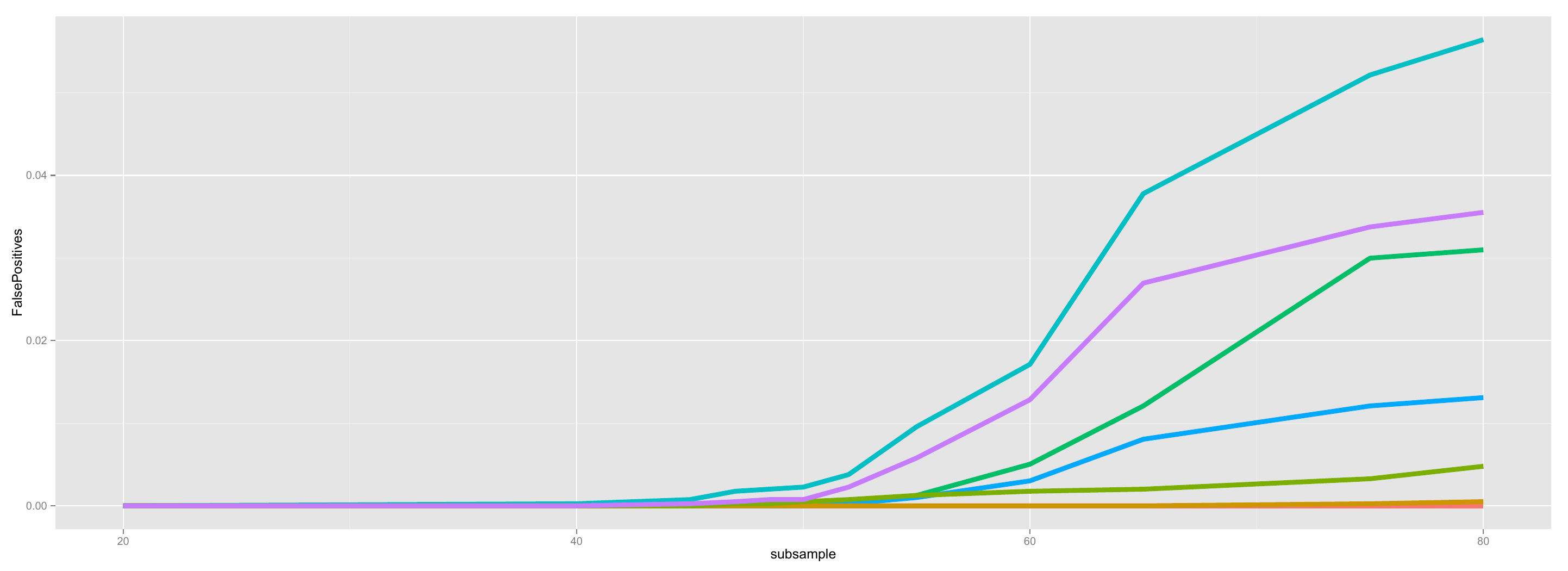}
 \includegraphics[width=7.5cm,height=4cm]{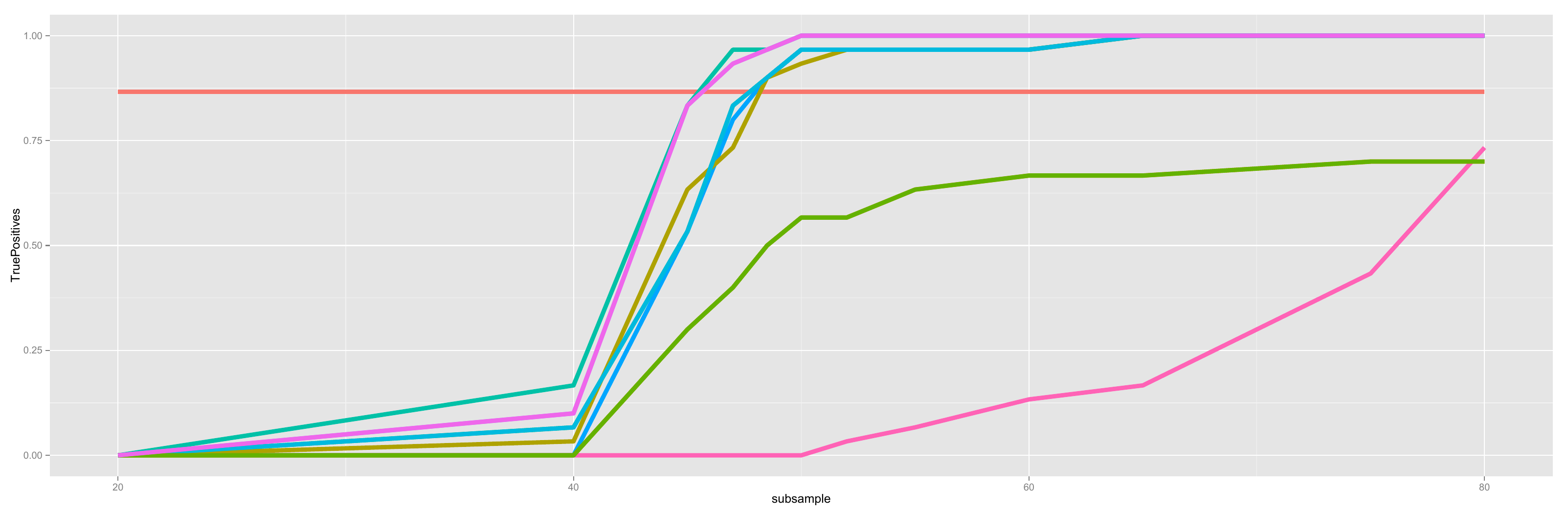}
  \includegraphics[width=7.5cm,height=4cm]{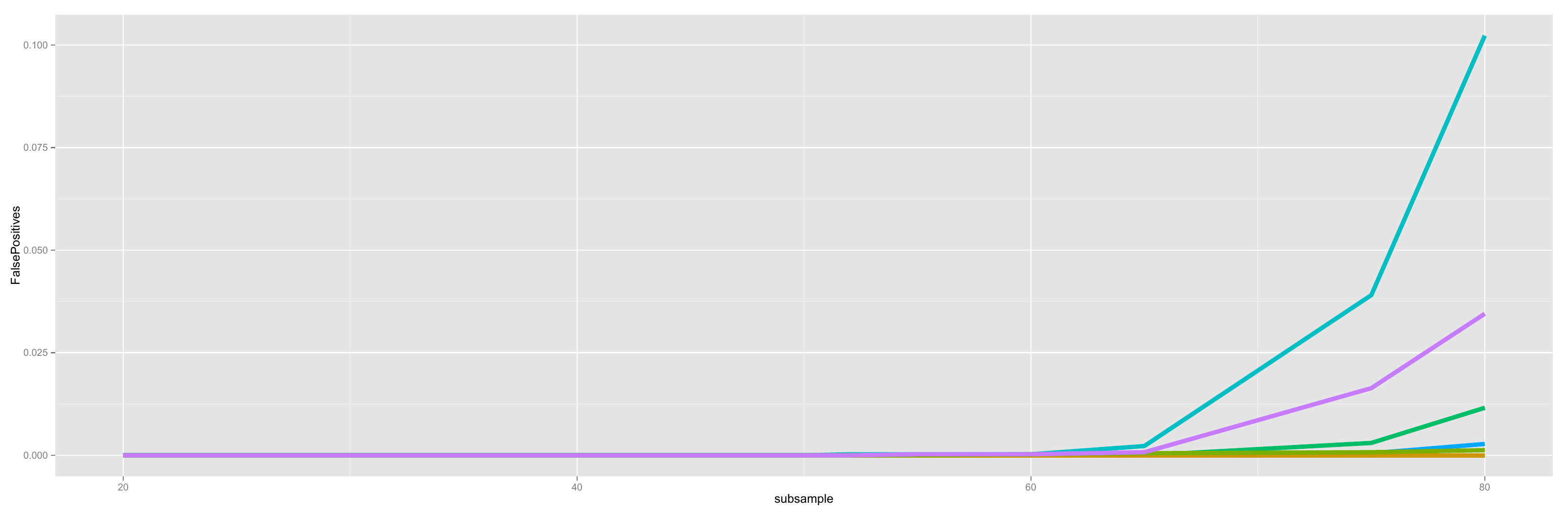}
  \caption{X-axes represents the  subsample size $N$ on a logarithmic scale and the Y-axes represents the average selection probability of the signal variables (the first column) and of the noise variables (the second column).  The first row is according to Model 4, the second according to Model 5 and the third according to Model 6.  Colors represent: red-Lasso, pink-subagging and minimax\eqref{eq:def} with brown-- $K=1,\tau=1$;
dark-brown--$K=1,\tau=1/2$;
dark-green--$K=1,\tau=4/5$;
  green--$K=10,\tau=1$;
green-blue--$K=10,\tau=1/2$;
 light-blue--$K=10,\tau=4/5$;
 blue--$K=3,\tau=1$ ;
purple--$K=3,\tau=1/2$; 
dark-pink--$K=3,\tau=4/5$. }\label{fig:fig1e}
  \end{figure}

\subsection{Mean Squared Error}

Under the same set of Models 1-6 , we investigate the convergence   of the proposed method  with respect to its mean squared error, as the subsample size  $N$ gets larger and larger. In this case, we keep  the sample size $n$ to be $10000$.
Results are summarized in the Figure  \ref{fig:fig1d}.
The maximum-contrast estimator   exhibits faster convergence rate in comparison to the traditional subagging by a large margin across all Models 1-6. Moreover, we see that different choices of $\tau$ do  not  alter   the performance  of the estimator by much.

 \begin{figure}[htbp]
 \centering
 \includegraphics[width=7.5cm,height=4cm]{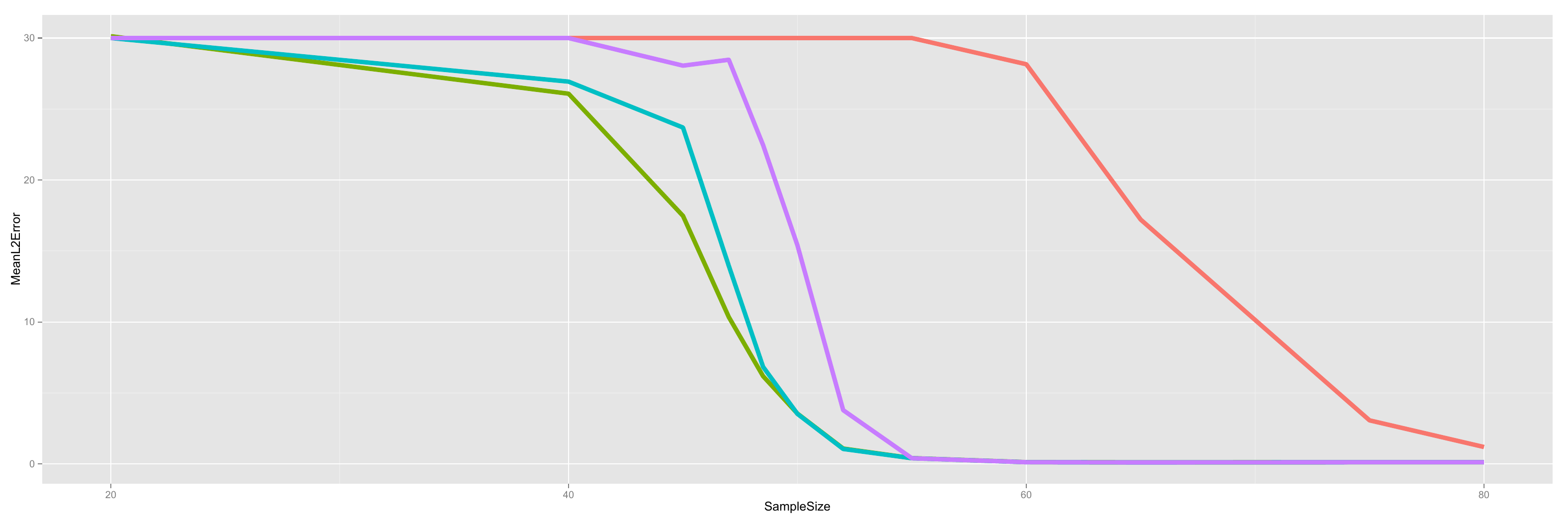}
  \includegraphics[width=7.5cm,height=4cm]{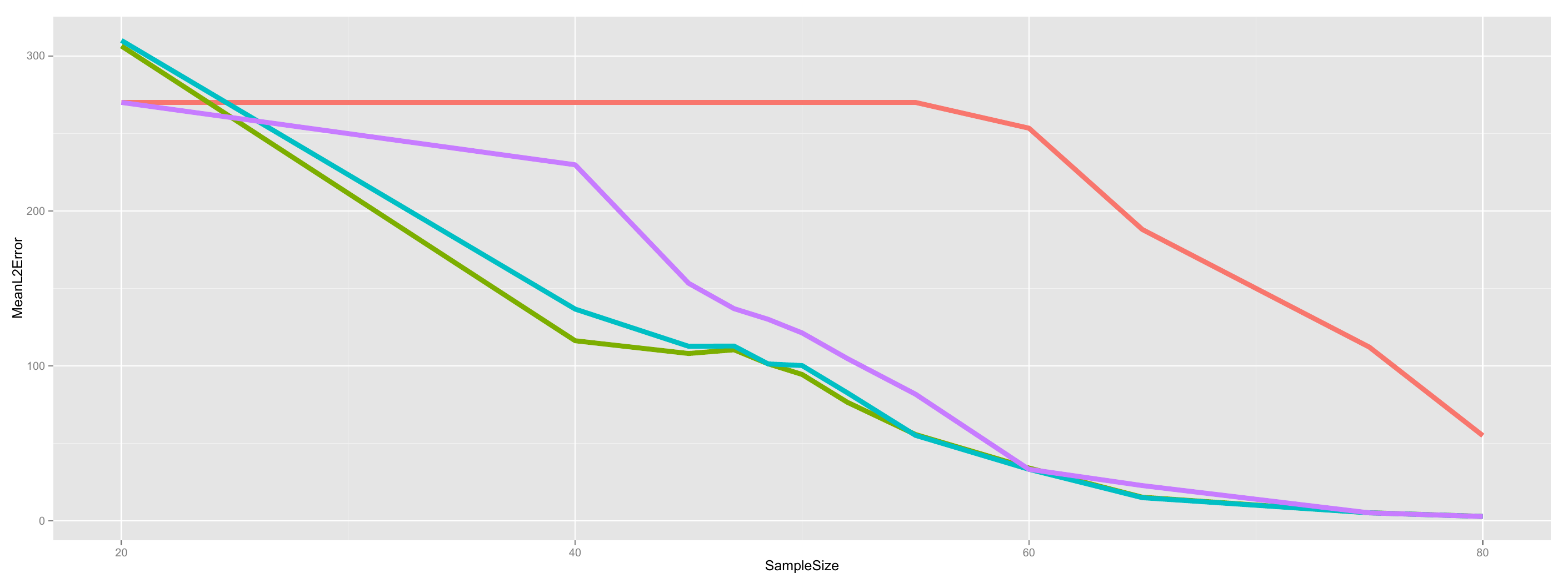}
 \includegraphics[width=7.5cm,height=4cm]{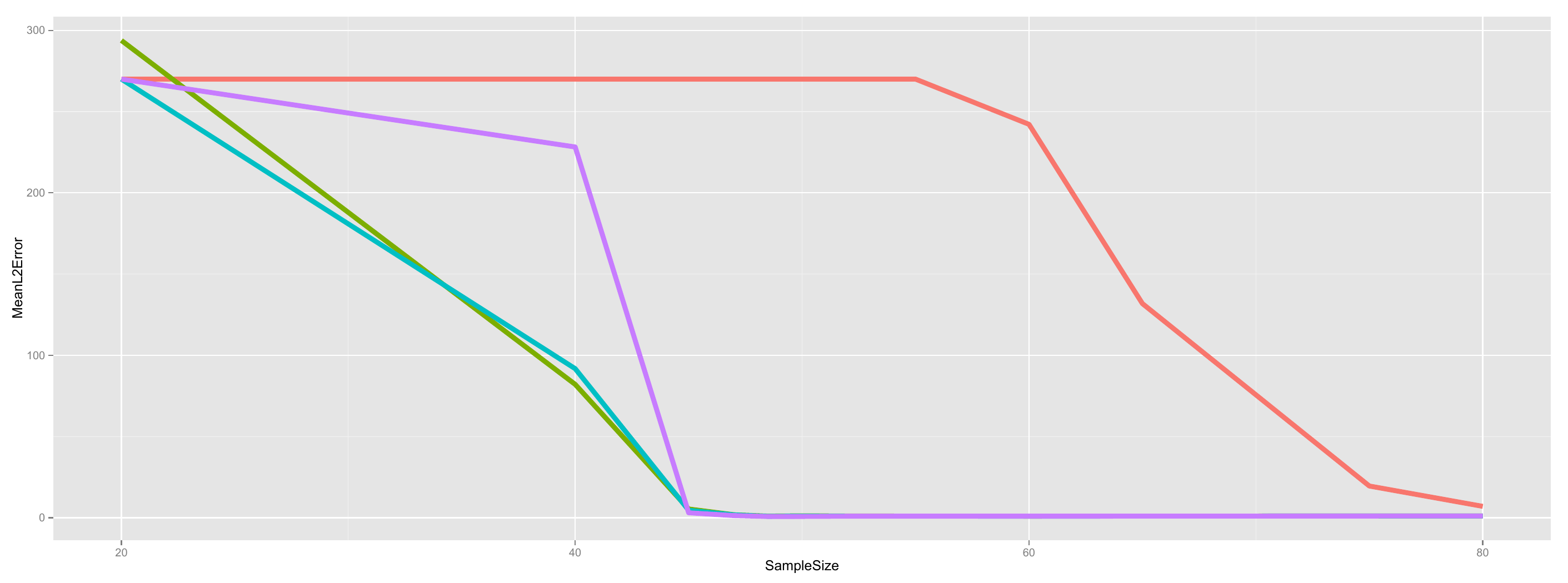}
 \includegraphics[width=7.5cm,height=4cm]{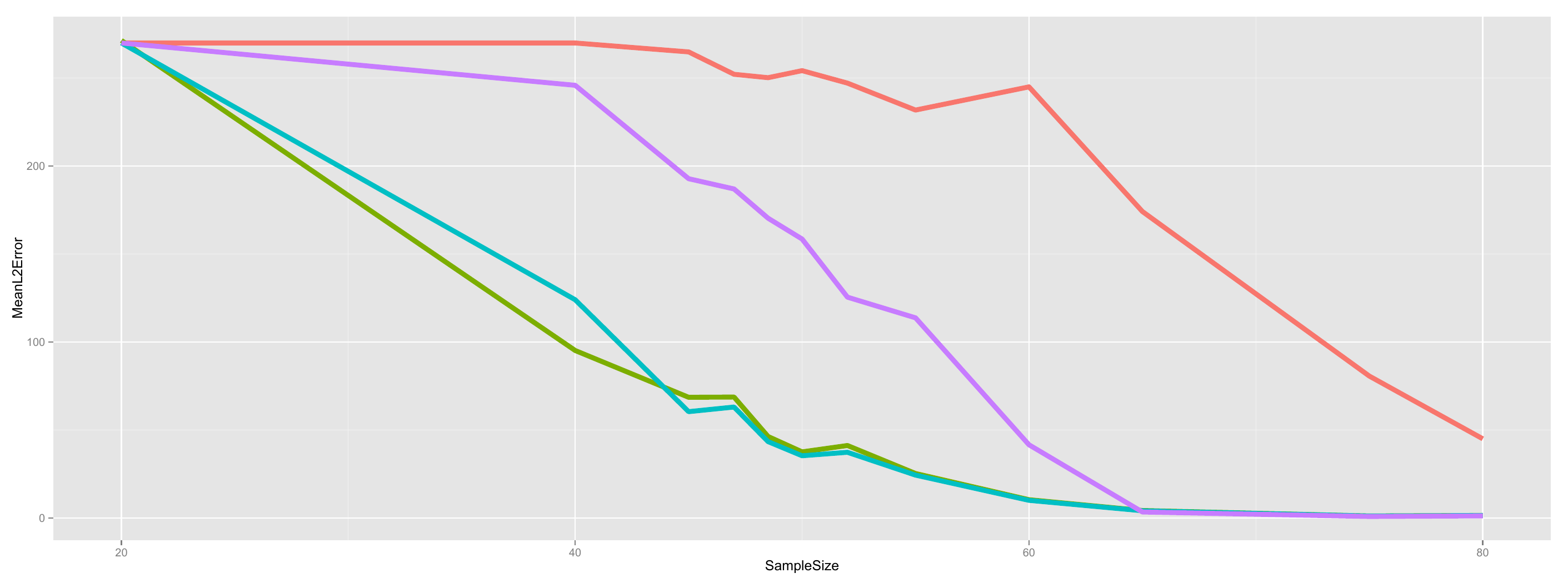}
 \includegraphics[width=7.5cm,height=4cm]{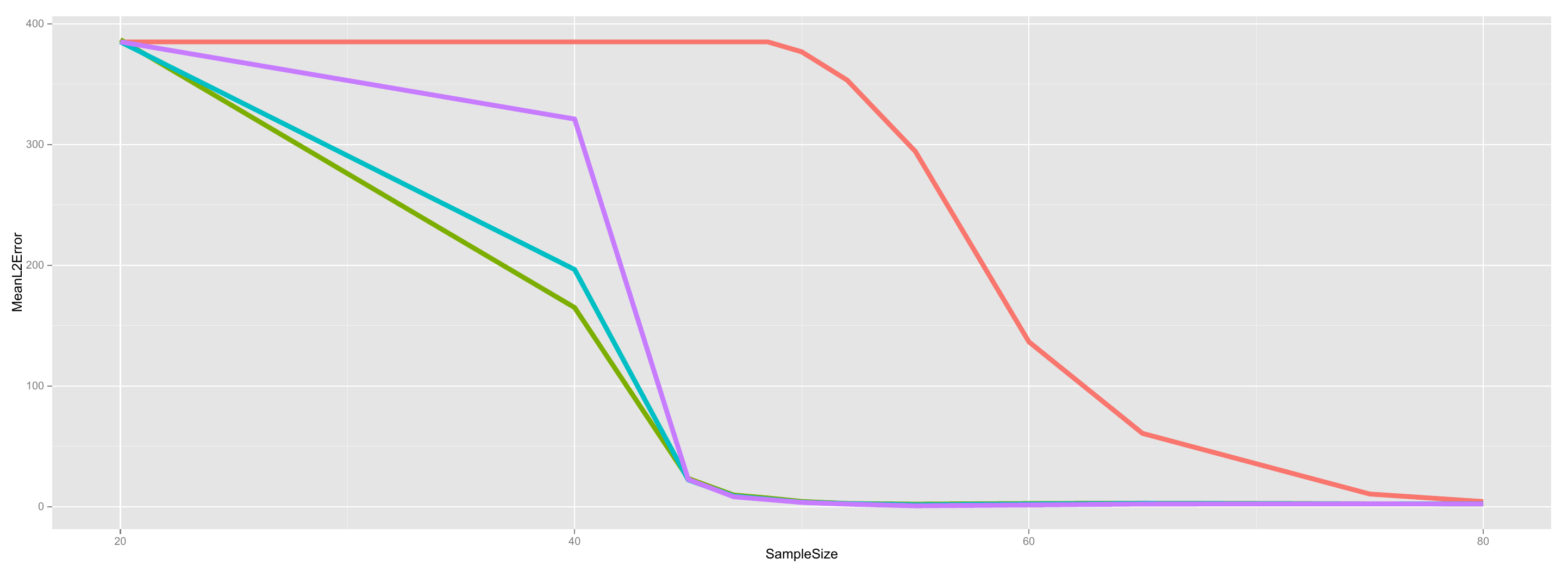}
  \includegraphics[width=7.5cm,height=4cm]{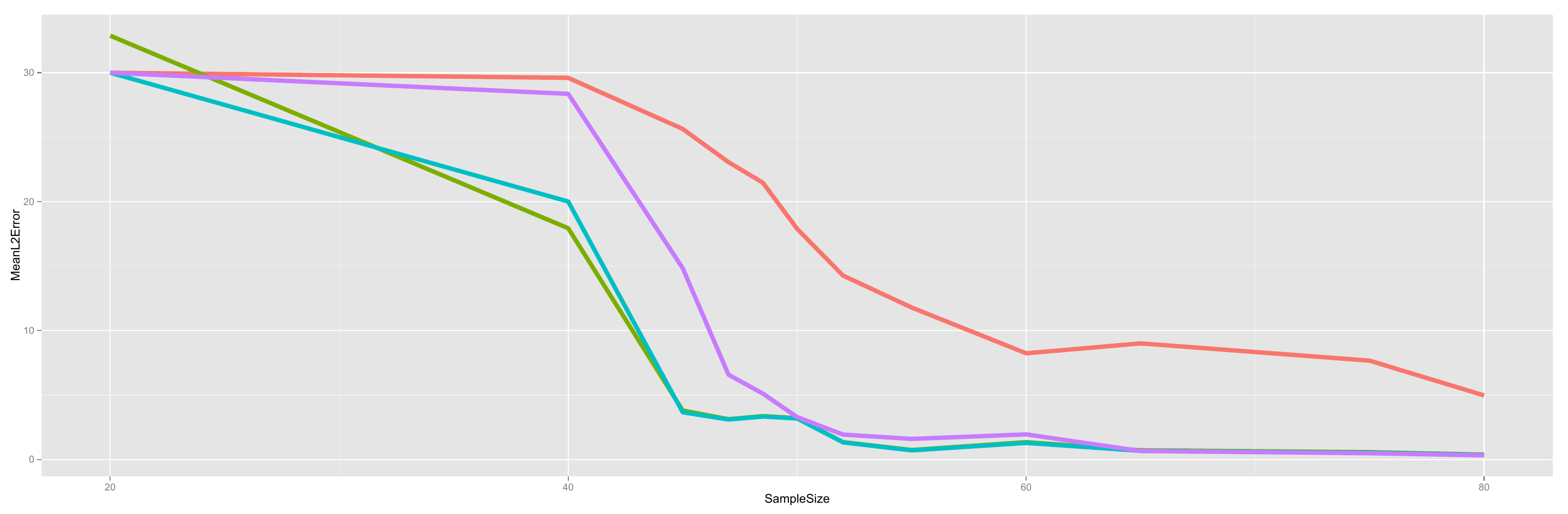}
  \caption{X-axes represents the  sample size   and the Y-axes represents the average $l_2$ error.    Simulation settings: the first row is according to Model 1 and 2, the second according to Model 3 and 4,  and the third according to Model 5 and 6 . Colors represent:   red-subagging and \eqref{eq:def} with   dark green-$K=3,\tau=0.3$,  blue-$K=3,\tau=0.5$ and purple-$K=3,\tau=0.8$}\label{fig:fig1d}
  \end{figure}

\subsection{Sensitivity with respect to the number of blocks $b$ } \label{sec:b}
We  also  test  the sensitivity of the proposed method with respect to  the choice of the number of blocks $b$. For that purpose, we generate synthetic data from the simple linear model.
 Model specifications are equivalent to the ones of Model 1. As a measure of performance, we contrast  mean selection frequency of  the first 30 important variables with a different number of  block i.e. boostrap replications. 
Figure \ref{fig:fig2} summarizes our findings and  reports average selection frequency and its 95\% confidence interval for 4 different bootstrap replications: $b= 1, 2, 5$ and $10$.  As expected, the  larger the $b$, the smaller is the variability in estimating. Interestingly, $b=2$ was sufficient to guarantee perfect recovery for all $\tau> 1/2$. Even for $b=1$ this seems to be true in the example considered, but the variability is significantly larger then for $b\geq2$, hence making general conclusion seem inappropriate.

\begin{figure} 
\centering
\includegraphics[width=13.5cm, height=5.5cm]{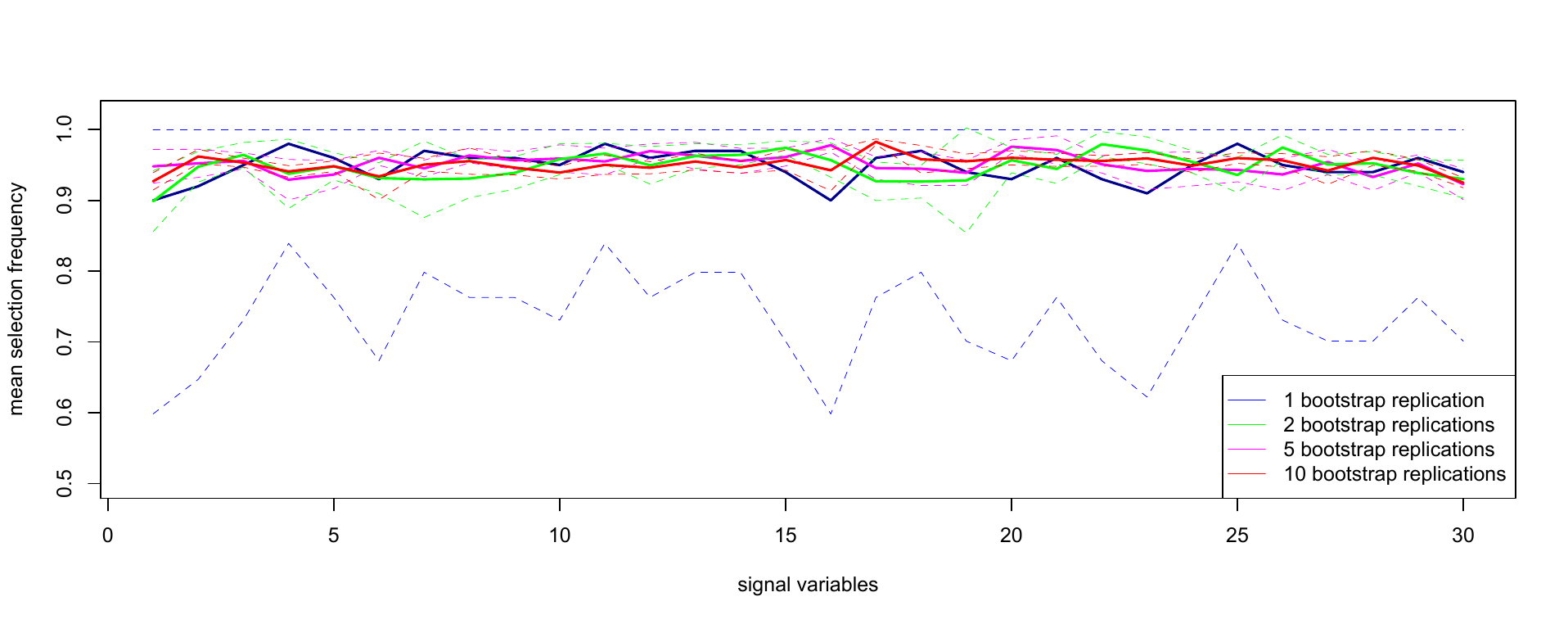}
\caption{ Variability of selection probabilities $\pi_j^*$ among signal variables with a change in the number of data perturbations ($B=1$, $B=2$,$B=5$ blue, green and pink respectively). X-axes enumerates the first 30 signal variables, whereas the Y-axes show the median selection probability. Dashed lines represent corresponding 95\% confidence intervals of it.}\label{fig:fig2}
\end{figure}

\newpage
\section{Discussion}
In this paper, we presented results demonstrating that our decomposition-based method for approximate variable selection achieves minimax optimal convergence rates, whenever the number of data partitions is not too large. We allow the number of partitions  to grow polynomially with the subsample size. The error guarantees of the maximum-contrast estimator  depend on the effective number of blocks of sample splits and the effective dimensionality of the support set (recall bound \eqref{eq:3.7} of Theorem \ref{thm:final}). For any number of blocks of sample splits $b \leq  \exp\{ (1-1/m)n/s + \log n/s\}/2$  and $K \leq \max\{N\rho_n^2\sqrt{b} /s, Np/s\} $,   our method achieves approximate support recovery i.e. 
$$P(S \subseteq \hat S_\tau) \geq 1-p^{1-c_N}, \qquad E |S^c \cap \hat S_\tau| \sim \frac{1}{\sqrt{b}} \frac{s}{p},  $$
for all $\lambda_N \sqrt{n/\log p}\in (1/c_N,c_N)$ and $c_N \geq 1$.  Theorem \ref{thm:optimal} confirms these to be minimax optimal for approximate recovery. In addition, we achieve substantial computational benefits coming from the subsampling schemes, in that computational costs scale linearly with $N$ rather than $n$.

The maximal-contrast estimator also has deep connections with the literature of stability selection.   Stability selection \citep{MB11} and paired stability selection
 \citep{SS12}, and more recent median selection \citep{D14} can be equivalently formulated as voting estimators.  
First, we clarify
that even when $K=1,m=1,b=d$, maximum-contrast is not the same procedure as Stability selection, although they share the
same population objective. The difference is that maximum-contrast utilizes a minimax estimator  of the  population objective.  Second, maximum-contrast  is
designed for   the settings with a growing number of   subsamples of a very small size; none of  the aforementioned methods can be directly implemented. These methods need  a fix number of large subsamples. 
Third, our theoretical analysis differs   from that in the existing work;   we do not require irrepresentable condition and we show optimality   of the proposed method.

     \section{Preliminary Lemmas}
     
     Let us introduce notation used throughout the proofs. We use $\langle \cdot, \cdot\rangle_n$ to denote the empirical inner product, i.e. $\langle\bu,\bv\rangle_n = \frac{1}{n}  \bu ^T \bv $. Whenever possible, we will suppress $\lambda_n$ and $\lambda_N$ in the notation of $\hat \bbeta(\lambda_n)$ and $\hat \bbeta_{i:k}(\lambda_N)$, and use $\hat\bbeta$ and $\hat\bbeta_{i:k}$.
     Let $P$ and $\mathbb P^*$ stand for the probability measures generated  by the error  vector $(\varepsilon_1,\cdots,\varepsilon_n)$ and generated jointly by the weights  $w_1,\cdots,w_N$ and  the errors $\varepsilon_1,\cdots,\varepsilon_n$.

 In order to study  statistical properties of the proposed  estimator, it is useful to present the optimality
conditions for   solutions of the problems \eqref{eq:lasso} and \eqref{eq:smallLasso}.
 
 $\hat \bbeta  $ is a solution to \eqref{eq:lasso}, if  and only if 
 \begin{align} 
\left\langle \bX_j,   \bY - \bX  \hat\bbeta\right\rangle_n &=  \lambda_n \mbox{sign}(\hat\beta_j), & \mbox{if}  & \qquad \hat \beta_j \neq 0 \label{eq:kktlassoa}\\ \label{eq:kktlassob}
 \left| \left\langle  \bX_j ,  \bY - \bX  \hat\bbeta\right\rangle_n  \right| &\leq  \lambda_n  , &  \mbox{if}  & \qquad \hat \beta_j = 0.
  \end{align}

 $\hat \bbeta_{i:k}  $ is a solution to \eqref{eq:smallLasso}, if  and only if 
  \begin{align} \label{eq:kktsublassoa}
\left\langle\bD_{\sqrt {\mathbf w_k}} \bX_{I_i,j},   \bD_{\sqrt {\mathbf w_k}}   \mathbf Y_{I_i} - \bD_{\sqrt {\mathbf w_k}}   \mathbf X_{ I_{i} } \hat \bbeta_{i:k}  \right\rangle_n&=& \lambda_N \mbox{sign}(\hat \beta_{i:k,j} ), &\quad \mbox{if}  &\hat \beta_{i:k,j} \neq 0 \\
\left| 
\left\langle
 \bD_{\sqrt {\mathbf w_k}} \bX_{I_i,j}, \bD_{\sqrt {\mathbf w_k}} \mathbf Y_{I_i} - \bD_{\sqrt {\mathbf w_k} }\mathbf X_{ I_{i} } \hat \bbeta_{i:k}  \right\rangle_n
 \right |
 &\leq& \lambda_N, \qquad \quad &\quad \mbox{if}  &\hat \beta_{i:k,j}  = 0.
  \label{eq:kktsublassob}
  \end{align}
In the display above the set of indices $I_i$ corresponds to those   which were used for the computation of $\hat \bbeta_{i:k} $.

Below we define the new primal dual witness technique to examine when the  solution  to one optimization problem is also a solution to the other. 
     
     \begin{lemma}\label{lem:kkt}
  Suppose $\hat\bbeta_{i:k}$ is a solution to the subLasso problem, i.e. it satisfies \eqref{eq:kktsublassoa} and \eqref{eq:kktsublassob}. Then, if 
\[
\left| \left\langle \bX_j, \bY_{I_i} - \bX_{I_i} \hat\bbeta_{i:k}\right\rangle_n\right| \leq \lambda_n - \lambda_N, \qquad \mbox{for all } j: \hat\beta_{i:k,j}=0,
\]
then $\hat\beta_{i:k,j}$ satisfies \eqref{eq:kktlassob}, that is $\hat S_i^c \subseteq \hat S ^c$.
%
\end{lemma}

 \begin{lemma}\label{lem:2}
 Suppose $\hat\bbeta$ is a solution to the  Lasso problem, i.e. it satisfies \eqref{eq:kktlassoa} and \eqref{eq:kktlassob}. Then, if 
\[
\left| \left\langle \bX_j, \bY_{I_i} - \bX_{I_i} \hat\bbeta \right\rangle_n\right| \leq  \lambda_N/n  - \lambda_n, \qquad \mbox{for all } j: \hat\beta_{j}=0,
\]
then $\hat\beta_{j}$ satisfies \eqref{eq:kktsublassob}, that is $\hat S^c \subseteq \hat S_i ^c$.
 \end{lemma}

For the  purpose of examining condition number of various design matrices, we first establish a bound on the spectral norm of the difference between  the inverses of two positive semidefinite matrices.  To establish this result we use Theorem III. 2.8 of \cite{B97}.

\begin{lemma}\label{lem:matrixineq}
Let $\bD,\bG \in \mathbb{R}^{n\times n}$ be  two semi-positive definite matrices. Let $\| \cdot \|$ be a matrix  norm induced by the vector $l_{\infty}$ norm. Then,
\begin{equation*}
\left\| \bD^{-1}-\bG^{-1} \right\|^2 \leq n  \left(\frac{1}{\lambda_{\min}(\bD)}  + \frac{1}{\lambda_{\min}(\bG)} \right).
\end{equation*}
\end{lemma}

Next we show that the solution of the Sub-Lasso problem \eqref{eq:smallLasso} has good predictive properties. 
Since proof follows the strategy of \cite{BRT09}, the proof is presented in the supplement for completeness.  

\begin{lemma}\label{thm:1}
Let  $\| \mathbf x^j_{I_i}\|_\infty \leq  c_2 \sqrt{ N / \log p}$, for some constant $c_2>1$ and all $1\leq j\leq p$.
Then,  on the event $\mathcal{A}_q(\lambda_N)= \bigcap_{j=1}^p \left\{2  \frac{1}{n}\sum_{l \in I_i} w_{k,l} |\varepsilon_i X_{lj}| \leq \lambda_N - q\right\}$, for $q < \lambda_N$,

 (i) There exists a positive number $ \mbox{e}_n$ such that 
$
\min\bigl\{ { {n}^{-1/2} \mathbb{E}_{\mathbf w} \{\| \tilde \bX \mathbf v\|_2 \}}/{\| \mathbf v_S\|_2}: |S|\leq s, \mathbf v \in \mathbb{R}^p,  \bigl.$
$ \bigl. \mathbf v \neq 0, \mathbf v \in \mathbb{C}(3,S) \bigl\} >  \mbox{e}_n
$.
Then, for all  $m$ and $b$  
$$
 \left\| \mathbf X_{I_i}  (\boldsymbol {\boldsymbol\beta}^* - \hat{\boldsymbol\beta} _{i:k}(\lambda_N)) \right\|_2^2  \leq \frac{ (4\lambda_N-q)^2  s n  }{ e_n^2 }  a_N^2  \mbox{ for } c_1  \sigma \sqrt{\frac{\log p}{n}} \leq \lambda_N \leq c_2 \sigma\sqrt{\frac{\log p}{N}},
$$
for some positive constant $a_N$ such that 
$P(\min_{1\leq l \leq N} w_l \geq \mathbb E _{\mathbf w}\|\sqrt{\mathbf w} \|_\infty a_N^{-1} ) \to 1$.

(ii) If Condition \ref{cond:re} holds, then  
$$
 \left\| \tilde \bX_{I_i}  (\boldsymbol {\boldsymbol\beta}^* - \hat{\boldsymbol\beta} _{i:k}(\lambda_N)) \right\|_2^2  \leq \frac{ (4\lambda_N-q)^2  s n  }{ \zeta_N^2 }      , \mbox{ for } c_1  \sigma \sqrt{\frac{\log p}{n}} \leq \lambda_N.
$$
 \end{lemma}


\begin{lemma}\label{lemma:3}
If   Condition \ref{cond:re} holds almost surely for $\tilde{\mathbf X}$,  and $\lambda_N \geq c_1 \sigma \sqrt{\log p/n}$, then   $|{\hat{S}}_i(\lambda_N, k)|\leq  \hat s:=C s \lambda_{\max}(\tilde{\mathbf X}^T \tilde{\mathbf X})/ (n \zeta_N^2)$ with probability $1-p^{1-c}$, $c>1$. 
\end{lemma}

\section{Proofs of the Main Results}
   In this section, we provide detailed proofs of the main theoretical results of the paper. 
One difficulty  with    each of the sub-Lasso problems  is  that  there is  no automatic mechanism   to provide   the regularization parameter $\lambda_N$.   Note that under Conditions \ref{cond:ir} and \ref{cond:re}, approximating  the true support set   becomes  equivalent to  approximating the   support set  of the Lasso estimator \eqref{eq:lasso}.
We  explore this connection and
find   values of the tuning parameter  $\lambda_N$, which allow  the   sparsity pattern of  the sub-Lasso to
approximate the one of the Lasso estimator. 

\subsection{Proof of Theorem \ref{thm:21}}

With  $k$ fixed,   and a little abuse of notation,  we use   $\hat\bbeta_i$ to denote $\hat \bbeta_{i:k}$ throughout this proof alone.
Utilizing  Lemma \ref{lem:kkt},  it  suffices    to show    that the event 
$$\Omega_n= \left\{ \left| \left \langle  \bX_{I_i^c,j}, \bY_{I_i^c} - \bX_{I_i^c} \hat{\bbeta}_i \right\rangle_n  \right|      \leq \lambda_n-\lambda_N \right\}, \qquad \mbox{for all } j: \hat\beta_{i,j}=0, $$ has large probability. 
In the display above we used   notation
\[
\frac{1}{n}\sum_{l\notin I_i}  X_{lj}(Y_l-\bX_l\hat{\bbeta}_i):= \left\langle   \mathbf X_{ I_i^c,j},\mathbf Y_{  I_i^c} - \mathbf X_{ I_i^c} \hat{\bbeta}_i \right\rangle_n.
\]
 Let $\hat S _i$ denote the set of the non-zero coefficients of $\hat{\bbeta}_i$. Let us denote with $\hat A=\hat S _i  \cup S$. Let $P_{\hat A}$
 be a projection operator into the space spanned by all variables in the set $\hat A$, that is 
 \[
 P_{\hat A} =\mathbf X_{ I_i^c,\hat A} (\mathbf X_{ I_i^c,\hat A}^T \mathbf X_{ I_i^c,\hat A} )^{-1}\mathbf X_{ I_i^c,\hat A}^T. 
 \]
 Then, we can split the inner product $\left\langle   \mathbf X_{ I_i^c,j},\mathbf Y_{  I_i^c} - \mathbf X_{ I_i^c} \hat{\bbeta}_i \right\rangle_n$ into two terms
 \[
 \left\langle \mathbf X_{ I_i^c,j},  (\mathbb I -  P_{\hat A} )(\mathbf Y_{  I_i^c} - \mathbf X_{ I_i^c} \hat{\bbeta}_i)  \right\rangle_n
 +
  \left\langle \mathbf X_{ I_i^c,j},   P_{\hat A}  (\mathbf Y_{  I_i^c} - \mathbf X_{ I_i^c} \hat{\bbeta}_i)  \right\rangle_n.
 \]
 Controlling the size of the set $\Omega_n$ is equivalent to upper bounding    previous two expressions separately. The second one is more challenging and is presented first.

\subsubsection{
Step I: Controlling $\left\langle \mathbf X_{ I_i^c,j},   P_{\hat A}  (\mathbf Y_{  I_i^c} - \mathbf X_{ I_i^c} \hat{\bbeta}_i)  \right\rangle_n$}

The KKT 
 equations \eqref{eq:kktsublassoa} and \eqref{eq:kktsublassob} provide the upper bound 
  \[
  \left\langle \mathbf X_{ I_i^c,j},   P_{\hat A}  (\mathbf Y_{  I_i^c} - \mathbf X_{ I_i^c} \hat{\bbeta}_i)  \right\rangle_n \leq 
\frac{\lambda_N}{n} \mathbf X_{ I_i^c,j}^T \mathbf X_{ I_i^c,\hat A} (\mathbf X_{ I_i^c,\hat A}^T \mathbf X_{ I_i^c,\hat A} )^{-1} (\mbox{sign} \hat{\bbeta}_i ( \lambda_N)).
\]
By  Lemma \ref{lemma:3},  with high probability
\begin{eqnarray} \label{eq:step49}
|\hat A | \leq s + C s \lambda_{\max}(\tilde{\mathbf X}^T \tilde{\mathbf X})/ (n \zeta_N):=r,
\end{eqnarray}
for some nonnegative constant $C$. Hence, $\lambda_{\min}(\mathbf X_{\hat A}^T \mathbf X_{\hat A}) \geq 
\inf_{|A|=r}\lambda_{\min}(\mathbf X_{  A}^T \mathbf X_{  A})$, where the last term is strictly positive by the Condition \ref{cond:re} with a  constant $a=1$ and a vector $\bv=(1,\dots, 1, 0 , \dots ,0)$, $\mbox{supp}(\bv) =  A$. In turn, we see that the matrix $\mathbf X_{  A}^T \mathbf X_{  A}$ is invertible with high probability.

We split $\frac{\lambda_N}{n} \mathbf X_{ I_i^c,j}^T \mathbf X_{ I_i^c,\hat A} (\mathbf X_{ I_i^c,\hat A}^T \mathbf X_{ I_i^c,\hat A} )^{-1} (\mbox{sign} \hat{\bbeta}_i ( \lambda_N))$  into the sum   $T_1 + T_2$ with
\[
T_1 =\frac{\lambda_N  |\hat A| }{n}  (\mathbf X_{\hat A}^T \mathbf X_{\hat A} )^{-1}   \mathbf X_{ I_i^c,\hat A}^T \mathbf X_{ I_i^c,j} \in \mathbb{R}^{|\hat A|}
\]
and 
\[
T_2 = \frac{\lambda_N  |\hat A|}{n}     \left[ (\mathbf X_{ I_i^c,\hat A}^T \mathbf X_{ I_i^c,\hat A} )^{-1}   - (\mathbf X_{\hat A}^T \mathbf X_{\hat A} )^{-1} \right]  \mathbf X_{ I_i^c,\hat A}^T \mathbf X_{ I_i^c,j} \in \mathbb{R}^{|\hat A|}.
\]

By the H\"older's inequality, it suffices to bound $\| T_1\|_\infty$ and $\| T_2\|_\infty$.
We treat the two terms  independently.
First,  by the  triangular inequality
\begin{eqnarray} \label{eq:step50}
\| T_1\|_\infty
&\leq  & 
 \frac{\lambda_N|\hat A |}{n} \left\| (\mathbf X_{\hat A}^T \mathbf X_{\hat A} )^{-1}  \left[ \mathbf X_{ I_i^c,\hat A}^T \mathbf X_{ I_i^c,j}  -\mathbf X_{  \hat A}^T \mathbf X_{ j} \right]\right\|_{\infty} + 
  \frac{\lambda_N|\hat A |}{n}\left\| (\mathbf X_{\hat A}^T \mathbf X_{\hat A} )^{-1}   \mathbf X_{  \hat A}^T \mathbf X_{ j}  \right\|_{\infty} \\
  \nonumber &:=&T_{11} + T_{12}.
\end{eqnarray}
We proceed to bound $T_{11}$ and $T_{12}$ next.
Let us first discuss the term $T_{11}$.
By consistency of the vector norm $\| \|$ and  its corresponding operator norm $\| \|$, Proposition IV.2.4 of  \cite{B97}, guarantees that  
$$\| \bM \bx\| \leq \| \bM\| \| \bx\|,$$
 for a matrix $\bM$ and a vector $\bx$.  Therefore,   for $\bM = (\mathbf X_{\hat A}^T \mathbf X_{\hat A} )^{-1}$ and $\bx =\mathbf X_{ I_i^c,\hat A}^T \mathbf X_{ I_i^c,j}  -\mathbf X_{  \hat A}^T \mathbf X_{ j} $,  
\[
\| T_{11}\|_\infty \leq  \frac{\lambda_N r}{n} \left\| (\mathbf X_{\hat A}^T \mathbf X_{\hat A} )^{-1} \right\| 
 \left\|   \mathbf X_{ I_i^c,\hat A}^T \mathbf X_{ I_i^c,j}  -\mathbf X_{  \hat A}^T \mathbf X_{ j}  \right\|_{\infty} ,
\]
with the induced operator normed $\| \cdot \|$ defined as 
\[
\left\| (\mathbf X_{\hat A}^T \mathbf X_{\hat A} )^{-1} \right\|  = \sup_{\bv \neq 0} \frac{ \| (\mathbf X_{\hat A}^T \mathbf X_{\hat A} )^{-1} \bv \|_\infty }{ \| \bv\|_\infty}.
\]
Let $\mathcal{E}_n =\{ \hat A = \hat S_i \cup S, |\hat A| \leq r: r \geq 0\}$.

On the event $\mathcal{E}_n $, we have 
\[
\left\| (\mathbf X_{\hat A}^T \mathbf X_{\hat A} )^{-1} \right\| 
\leq 
\sup_{A: |A| \leq r} \sup_{\bv \neq 0} \frac{ \| (\mathbf X_{  A}^T \mathbf X_{  A} )^{-1} \bv \|_\infty }{ \| \bv\|_\infty}
=
\max_{\| \bv\|_\infty=1, \bv \neq 0} \sup_{A: |A| \leq r}   { \| (\mathbf X_{  A}^T \mathbf X_{  A} )^{-1} \bv \|_\infty }.
\]
 For a matrix $\bM \in \mathbb{R}^{r \times r}$,  and its operator and Frobenious norm, a simple inequality holds $\| \bM\| \leq \sqrt{r} \| \bM \|_F = \sqrt{r} \sqrt{\lambda_{\max} (\bM^T \bM)}$
\citep{B97}. Using such inequality, on the event $\mathcal{E}_n $ we have 
\[
\left\| (\mathbf X_{\hat A}^T \mathbf X_{\hat A} )^{-1} \right\|  \leq \sqrt{r}  \sup_{A: |A| \leq r}  \sqrt{\lambda_{\max} \left( (\mathbf X_{  A}^T \mathbf X_{ A} )^{-2}\right)}.
\]
Furthermore, as $\lambda_{\max} (\bM^{-1}) = \lambda_{\min}^{-1}(\bM)$ \citep{B97},  on $\mathcal{E}_n $
\[
\left\| (\mathbf X_{\hat A}^T \mathbf X_{\hat A} )^{-1} \right\|  \leq \sqrt{r}  /  \sup_{|A| \leq r}  \lambda_{\min}(\mathbf X_{  A}^T \mathbf X_{ A}) \leq \sqrt{r} / \zeta_N^2,
\]
where in the last step Condition \ref{cond:re} guarantees $ \lambda_{\min}(\mathbf X_{  A}^T \mathbf X_{ A}) \geq \zeta_N^2$  on the set $\mathcal{E}_n$. At the moment,   the bound on $\| T_{11}\|_\infty$, conditional on the event $\mathcal{E}_n$, is as follows 
\begin{equation}
\| T_{11}\|_\infty \leq  \frac{\lambda_N r^2 }{n \zeta_N^2}  
 \left\|   \mathbf X_{ I_i^c,\hat A}^T \mathbf X_{ I_i^c,j}  -\mathbf X_{  \hat A}^T \mathbf X_{ j}  \right\|_{\infty} .
\end{equation}
Now   we observe that simple  inequality provides
\begin{equation}\label{eq:step501}
\| \mathbf X_{  \hat A}^T \mathbf X_{ j}\|_\infty  =  \max_{q \in \hat A} \sum_{i =1}^n | X_{iq}X_{ij}|  \geq \max_{q \in \hat A} \sum_{i \in I_i} | X_{iq}X_{ij}|=\| \mathbf X_{ I_i^c,\hat A}^T \mathbf X_{ I_i^c,j}\|_\infty.
\end{equation}
 Combined with the triangular inequality, conditional on the event $\mathcal{E}_n$, guarantees that $T_{11}$ is bounded as
\begin{equation}
\| T_{11}\|_\infty \leq  \frac{2 \lambda_N r^2 }{n \zeta_N^2}  
 \left\|  \mathbf X_{  \hat A}^T \mathbf X_{ j}  \right\|_{\infty} .
 \end{equation}
 Moreover, the norm inequality    $ \| \bx\|_\infty \leq \| \bx\|_2$ holds for any vector $\bx $.
In combination with  Lemmas \ref{lem:7} and \ref{lemma:3},  
\begin{equation}\label{eq:NewT1}
\| T_{11}\|_\infty = \mathcal{O}_P \left(  \frac{2 \lambda_N r^2 }{n \zeta_N^{5 }} \right).
 \end{equation}

For the term $T_{12}$, we first observe that $\|\bx  \|_{\infty} \leq \| \bx \|_1$. Second, if we  use  equation \eqref{eq:step01} of Lemma \ref{lem:7}, and the result of Lemma \ref{lemma:3}
\begin{equation}\label{eq:NewT2}
\| T_{12}\|_\infty = \mathcal{O}_P\left(  \frac{  \lambda_N r^{3/2}  }{n \zeta_N  } \right).
 \end{equation}

We now discuss the term $T_2$. By the the consistency of the operator and the vector norms  we have 
\begin{equation}\label{eq:step100000}
\| T_2 \|_{\infty} \leq \frac{\lambda_N r}{n} \left \|  (\mathbf X_{ I_i^c,\hat A}^T \mathbf X_{ I_i^c,\hat A} )^{-1}   - (\mathbf X_{\hat A}^T \mathbf X_{\hat A} )^{-1} \right\| \|  \mathbf X_{ I_i^c,\hat A}^T \mathbf X_{ I_i^c,j}\|_{\infty}
\end{equation}
where  the operator norm  above is defined as 
\[
\left\|  (\mathbf X_{ I_i^c,\hat A}^T \mathbf X_{ I_i^c,\hat A} )^{-1}   - (\mathbf X_{\hat A}^T \mathbf X_{\hat A} )^{-1} \right\|  
= \sup_{\bv \neq 0} \frac{ \left \|  \left( (\mathbf X_{ I_i^c,\hat A}^T \mathbf X_{ I_i^c,\hat A} )^{-1}   - (\mathbf X_{\hat A}^T \mathbf X_{\hat A} )^{-1}  \right) \bv \right\|_\infty }{ \| \bv\|_\infty}.
\]

 We treat each term in  \eqref{eq:step100000}  separately.

Using $\bG=\mathbf X_{\hat A}^T \mathbf X_{\hat A}$ and $\bD=\mathbf X_{ I_i^c,\hat A}^T \mathbf X_{ I_i^c,\hat A} $ 
in Lemma \ref{lem:matrixineq} we have 
\begin{eqnarray} \label{eq:step02}
\left \|  (\mathbf X_{ I_i^c,\hat A}^T \mathbf X_{ I_i^c,\hat A} )^{-1}   - (\mathbf X_{\hat A}^T \mathbf X_{\hat A} )^{-1} \right\| \leq  \sqrt{r} \sqrt{ \lambda_{\min}^{-1}(\mathbf X_{ I_i^c,\hat A}^T \mathbf X_{ I_i^c,\hat A} ) + \lambda_{\min}^{-1}(\mathbf X_{ \hat A}^T \mathbf X_{  \hat A} )}.
\end{eqnarray}  
Moreover, on the event $\mathcal{E}_n$
\[
 \lambda_{\min}^{}(\mathbf X_{ I_i^c,\hat A}^T \mathbf X_{ I_i^c,\hat A} ) 
 \geq 
 \lambda_{\min}^{}(\mathbf X_{ \hat A}^T \mathbf X_{  \hat A} )
  \geq  \inf_{|A|=r} \lambda_{\min}^{}(\mathbf X_{  A}^T \mathbf X_{ A} )
   \geq \zeta_n^2,
 \]
 where the last inequality follows from Condition \ref{cond:re}.



Term $\|  \mathbf X_{ I_i^c,\hat A}^T \mathbf X_{ I_i^c,j}\|_{\infty}$ is bounded similarly as  with the term $T_{1}$.
 Therefore, we conclude
 \begin{eqnarray} \label{eq:step012}
\| T_2\|_\infty
= \mathcal{O}_P \left( 
 \frac{ \sqrt{2} \lambda_N r^{3/2} }{n \zeta_n  \zeta_N^{3 }}  
 \right).
\end{eqnarray} 

 \subsubsection{
Step II: Controlling  $\left\langle \mathbf X_{ I_i^c,j},  (\mathbb I -  P_{\hat A} )(\mathbf Y_{  I_i^c} - \mathbf X_{ I_i^c} \hat{\bbeta}_i)  \right\rangle_n $}

By definition of $\hat A$ we have $S \subset \hat A$ and  by Lemma \ref{lemma:3} that $| \hat A| \leq r$ with high probability. Hence, with $r $ as in \eqref{eq:step49},
the term  of interest is   upper bounded with 
\[
\sup_{A: |A|\leq r} \sup_{j \notin A} \left \langle \mathbf X_{ I_i^c,j} , (\mathbb I -  P_{  A} )  \boldsymbol{ \varepsilon} _{ I_i^c}\right \rangle_n.
\]
Remember $\| \bX_j\|_2=1$.
Note that for all $j \notin A$, $ | \bX_{j}^T P_{A } \boldsymbol\varepsilon| \leq \| P_{  A} \boldsymbol \varepsilon \|_2$. Hence, 
the expression in  the above display is then bounded with
\begin{equation}\label{eq:step10}
\max_{j \in \{1,\cdots,p\}} \left \langle  \mathbf X_{ I_i^c,j} ,   \boldsymbol{ \varepsilon} _{ I_i^c} \right \rangle_n + \sup_{A: |A|\leq r}  \frac{1}{n}\|P_{  A} \boldsymbol \varepsilon_{ I_i^c}\|_2.
\end{equation}
Observe that $\| \mathbf X_{ I_i^c,j}\|_2 \leq 1$, $|I_i^c| =N-n$ and that $\langle \cdot, \cdot \rangle$ denotes the empirical inner product.
Since $\{\varepsilon_i\}_{i=1}^n$ are i.i.d.  with bounded moments,  we have by the weighted Bernstein inequality, that there exists a constant $c>0$ such that for a sequence of positive numbers $u_n$
\begin{align}\nonumber
& P \left(\frac{n}{n-N}\max_{j \in \{1,\cdots,p\}}  \left| \left \langle \mathbf X_{ I_i^c,j}, \boldsymbol{ \varepsilon} _{ I_i^c}  \right \rangle_n  \right|\geq    u_n \right) \\  
&=P \left( \frac{1}{n-N} \sum_{i \in I_i^c}\left|  X_{ i,j}  { \varepsilon} _{i}   \right|\geq    u_n \right)
 \leq \exp \left\{- c  \frac{ (N-n)^2 u_n^2}{ 8 \sigma^2 } \right\}.
 \label{eq:step11}
\end{align}
For a choice of $u_n = 2 \sigma \sqrt{\log p /(N-n)}$ the above probability will converge to zero.

We now bound the second term in \eqref{eq:step10}. Lemma \ref{lemma:3} implies that, conditional on the event  $\mathcal{E}_n$
$$\sqrt{ (n-N)} \|P_{\hat A} \boldsymbol{ \varepsilon}_{ I_i^c} \| _2/\sigma \leq \sup_{|A|=r} \sqrt{ (n-N)} \|P_{  A} \boldsymbol{ \varepsilon}_{ I_i^c} \| _2/\sigma.$$ 
 Furthermore, $\sup_{|A|=r} \sqrt{ (n-N)} \|P_{  A} \boldsymbol{ \varepsilon}_{ I_i^c} \| _2/\sigma$
  has  a $\chi_{r}$ distribution. Hence, tail bounds of the Chi-squared distribution (Lemma 1 of \cite{ML00}) lead to
\begin{equation}\label{eq:step12}
P \left((n-N) \sup_{A: |A|\leq r} \|P_{ A} \boldsymbol \varepsilon_{ I_i^c}\|_2^2 \geq    \sigma^2 r  { { ( 1+ 4   \log p)}{ }}\right) \leq \exp\{- 3/2 r  \log p\}.
\end{equation}
 
Plugging in \eqref{eq:step11} and \eqref{eq:step12} into \eqref{eq:step10}, we obtain
\begin{eqnarray} \label{eq:step022}
\left|  \left\langle \mathbf X_{ I_i^c,j},  (\mathbb I -  P_{\hat A} )(\mathbf Y_{  I_i^c} - \mathbf X_{ I_i^c} \hat{\bbeta}_i)  \right\rangle_n \right| \leq 2 \sigma \sqrt{\frac{(n-N) \log p}{n^2}} +2 \sigma \sqrt{\frac{ 4 r  \log p}{n^2 (n- N) }}
\end{eqnarray} 
with probability $1-2p^{-c}$.

Combining        \eqref{eq:NewT1}, \eqref{eq:NewT2}, \eqref{eq:step012} and \eqref{eq:step022} we obtain
\[
\left| \left\langle   \mathbf X_{ I_i^c,j},\mathbf Y_{  I_i^c} - \mathbf X_{ I_i^c} \hat{\bbeta}_i \right\rangle_n \right| \leq  \frac{2 \lambda_N r^2 }{n \zeta_N^{5 }} + \frac{2 \lambda_N r^{3/2}  }{n \zeta_N  } + \frac{ \sqrt{2} \lambda_N r^{3/2} }{n \zeta_n  \zeta_N^{3}}  
  +2 \sigma \sqrt{\frac{(n-N) \log p}{n^2}} +2 \sigma \sqrt{\frac{ 4 r  \log p}{n^2 (n- N) }},
\]
with probability of $1-p^{1-c}$, $c \geq 1$. 

Moreover, according to Lemma \ref{lem:kkt}
the RHS of the expression above needs to be smaller than $\lambda_n-\lambda_N$, for the event $\Omega_n$ to hold. This leads to the relation of 
\begin{eqnarray*}  
\lambda_n 
&\geq& \lambda_N +
  \frac{2 \lambda_N r^2 }{n \zeta_N^{5 }} + \frac{2 \lambda_N r^{3/2}  }{n \zeta_N  } + \frac{ \sqrt{2} \lambda_N r^{3/2} }{n \zeta_n  \zeta_N^{3}} 
  +2 \sigma \sqrt{\frac{(n-N) \log p}{n^2}} +2 \sigma \sqrt{\frac{ 4 r  \log p}{n^2 (n- N) }},
%
%
\end{eqnarray*} 
which completes the proof. \qed

\subsection{Proof of Theorem \ref{thm:31}}

We  show that the condition of Lemma \ref{lem:2} holds with high probability for the Lasso estimator $\hat{\boldsymbol\beta}  = \hat{\boldsymbol\beta} (\lambda_n)$. To that end, we  define an event 
$$\Omega_n= \left\{ \left| \left \langle  \bX_{I_i^c,j}, \bY_{I_i^c} - \bX_{I_i^c} \hat{\bbeta}  \right\rangle_n  \right|      \leq \lambda_N/n-\lambda_n\right\}, \qquad \mbox{for all } j: \hat\beta_{j}=0 $$ and show that it has a large probability. 
In the above, we utilized the following notation
\[
\frac{1}{n}\sum_{l\notin I_i}   X_{lj}(Y_l-\bX_l\hat{\boldsymbol\beta} ) := \left\langle \mathbf X_{ I_i^c,j},  \mathbf Y_{  I_i^c} - \mathbf X_{ I_i^c} \hat{\bbeta}   \right\rangle_n.
\]
Let   $\hat S$  be  the  set of non-zero coefficients of the Lasso estimator $\hat{\boldsymbol\beta} $.   Let $P_{\hat S}$
 be the
 projection operator into the space spanned by all variables in the set $\hat S$.
By repeating similar decomposition analysis  developed in  Theorem \ref{thm:21}, $\left\langle \mathbf X_{ I_i^c,j},  \mathbf Y_{  I_i^c} - \mathbf X_{ I_i^c} \hat{\bbeta}   \right\rangle_n$ is bounded with
  \[
 \left\langle \mathbf X_{ I_i^c,j},  (\mathbb I -  P_{\hat S} )(\mathbf Y_{  I_i^c} - \mathbf X_{ I_i^c} \hat{\bbeta} )  \right\rangle_n
 +
  \left\langle \mathbf X_{ I_i^c,j},   P_{\hat S}  (\mathbf Y_{  I_i^c} - \mathbf X_{ I_i^c} \hat{\bbeta} )  \right\rangle_n,
 \]
the  proof setup  of  Step I and II of Theorem \ref{thm:21} extends easily.

 Controlling $ \left\langle \mathbf X_{ I_i^c,j},  (\mathbb I -  P_{\hat S} )(\mathbf Y_{  I_i^c} - \mathbf X_{ I_i^c} \hat{\bbeta} )  \right\rangle_n $ follows by adapting the proof of Theorem \ref{thm:21} to a different projection matrix.
This term is upper bounded by utilizing KKT conditions with
\[
\frac{\lambda_n}{n} \mathbf X_{ I_i^c,j}^T \mathbf X_{ I_i^c,\hat S} (\mathbf X_{ I_i^c,\hat S}^T \mathbf X_{ I_i^c,\hat S} )^{-1} (\mbox{sign} \hat{\bbeta} ( \lambda_n)) 
\]
Expression above is bounded by 
\begin{eqnarray}  
&&\frac{\lambda_n}{n}   \mathbf X_{ I_i^c,j}^T \mathbf X_{ I_i^c,\hat S} (\mathbf X_{  \hat S}^T \mathbf X_{  \hat S} )^{-1} (\mbox{sign} \hat{\bbeta} ( \lambda_n))  \nonumber \\
&+& 
\frac{\lambda_n |\hat S| }{n}   \left\|  \left[ (\mathbf X_{ I_i^c,\hat S}^T \mathbf X_{ I_i^c,\hat S } )^{-1}   - (\mathbf X_{\hat S}^T \mathbf X_{\hat S} )^{-1} \right]  \mathbf X_{I_i^c,\hat S}^T \mathbf X_{I_i^c,j}\right\|_{\infty} \nonumber \\
&:=& U_1 + U_2.
\end{eqnarray}
We proceed to bound $U_1$ and $U_2$ independently.
 By Condition  \ref{cond:ir}, the first term, $U_1$, can be bounded with 
\[
\frac{\lambda_n}{n}  +\frac{\lambda_n | \hat S | }{n}   \left\| (\mathbf X_{  \hat S}^T \mathbf X_{  \hat S} )^{-1} \left[ \mathbf X_{ I_i^c,j}^T \mathbf X_{ I_i^c,\hat S}   - \mathbf X_{  j}^T \mathbf X_{ \hat S}  \right] \right\|_\infty.
\]
 By the H\"older's inequality,    the expression above is bounded from above by
 \[
\frac{\lambda_n}{n}  +\frac{\lambda_n | \hat S | }{n}  \sup_{|A| \leq r} \lambda_{\min}^{-1}\left( \mathbf X_{  A}^T \mathbf X_{   A} \right) \sup_{|A| \leq r} \sup_{j \notin A} 
\left\|    \mathbf X_{ I_i^c, A}^T\mathbf X_{ I_i^c,j}    -  \mathbf X_{ A}^T \mathbf X_{  j}     \right\|_\infty,
\]
where $r$ denotes the size of the set $\hat S$. Its   size  follows from   Lemma B.1 of   \cite{BRT09},  i.e.  
\begin{equation}\label{eq:step050100}
|\hat S| \leq C{s \lambda_{\max}\left(\frac{1}{n} \mathbf X^T \mathbf X\right)}/{\zeta_n} :=r
\end{equation}
with probability approaching 1.
Next, by  Condition \ref{cond:re} and Lemma \ref{lemma:3} we have, 
$ \sup_{|A| \leq r} \lambda_{\min}^{-1}\left( \mathbf X_{  A}^T \mathbf X_{   A} \right) \leq \frac{1}{\zeta_n^2}$ and $\left\|    \mathbf X_{ I_i^c, A}^T\mathbf X_{ I_i^c,j}    -  \mathbf X_{ A}^T \mathbf X_{  j}     \right\|_\infty \leq 
\left\|    \mathbf X_{ I_i^c, A}^T\mathbf X_{ I_i^c,j}    -  \mathbf X_{ A}^T \mathbf X_{  j}     \right\|_2 
\leq \sqrt{r}/\zeta_{N-n}$, respectively.
Combining all of the above,
\begin{equation}\label{eq:step0501}
U_1 = \mathcal{O}_P \left(  \frac{\lambda_n}{n} + \frac{\lambda_n r^{3/2}  }{n {\zeta_n^2 \zeta_{N-n}}} \right).
\end{equation}
 
 Regarding   $U_2$ we note that 
  \eqref{eq:step02} still holds with $\hat S$ replacing $\hat A$. Moreover,  Lemma \ref{lem:7} still applies. Hence,
   we can conclude 
\begin{eqnarray} \label{eq:step012a}
U_2
= \mathcal{O}_P \left(
 \frac{\sqrt{2}\lambda_n r^{3/2}}{n \zeta_n \zeta_{N-n}^3}
  \right).
\end{eqnarray}

  Controlling $ \left\langle \mathbf X_{ I_i^c,j},   P_{\hat S}  (\mathbf Y_{  I_i^c} - \mathbf X_{ I_i^c} \hat{\bbeta} )  \right\rangle_n$ is done as   in Theorem \ref{thm:21}.
  The same steps still apply by noticing that $S \subseteq  \hat S $ (as Conditions \ref{cond:ir} and \ref{cond:re} hold) and that Lemma \ref{lem:7} holds.  Hence, we obtain 
%
%

\begin{eqnarray} \label{eq:step0220}
\left|  \left\langle \mathbf X_{ I_i^c,j},  (\mathbb I -  P_{\hat A} )(\mathbf Y_{  I_i^c} - \mathbf X_{ I_i^c} \hat{\bbeta} )  \right\rangle_n \right| \leq 2 \sigma \sqrt{\frac{(n-N) \log p}{n^2}} +2 \sigma \sqrt{\frac{ 4 r  \log p}{n^2 (n- N) }}
\end{eqnarray} 

Adding results of \eqref{eq:step0501} and \eqref{eq:step012a} with the one above, we obtain
\begin{eqnarray*}  
\left|  \left\langle \mathbf X_{ I_i^c,j},   \mathbf Y_{  I_i^c} - \mathbf X_{ I_i^c} \hat{\bbeta}    \right\rangle_n  \right|
 &\leq&   \frac{\lambda_n}{n} + \frac{\lambda_n r^{3/2}  }{n {\zeta_n^2 \zeta_{N-n}}}
 +
 \frac{\sqrt{2}\lambda_n r^{3/2}}{n \zeta_n \zeta_{N-n}^3}
 \\
&+&2 \sigma \sqrt{\frac{(n-N) \log p}{n^2}} +2 \sigma \sqrt{\frac{ 4 r  \log p}{n^2 (n- N) }}
\end{eqnarray*}  
with probability $1-p^{1-c}$, $c>1$. 
Note that for all $\lambda_n \geq \sigma \sqrt{\log p/n}$, there exists a  constant $C'>1$,  such that $r \leq  C{s \lambda_{\max}\left(\frac{1}{n} \mathbf X^T \mathbf X\right)}/{\zeta_n} \leq C' s $.
According to the definition of $\Omega_n$  it suffices to have the RHS above bounded with $\frac{\lambda_N}{n } - \lambda_n $. In turn, this implies 
\begin{eqnarray*}  
\frac{\lambda_N}{n } 
&\geq&   \lambda_n +\frac{\lambda_n}{n} + \frac{\lambda_n C'^{3/2}  s^{3/2}  }{n {\zeta_n^2 \zeta_{N-n}}}
 +
 \frac{\sqrt{2}\lambda_n C'^{3/2}s^{3/2}}{n \zeta_n \zeta_{N-n}^3} 
+2 \sigma \sqrt{\frac{(n-N) \log p}{n^2}} +2 \sigma \sqrt{\frac{ 4 C' s  \log p}{n^2 (n- N) }}.
\end{eqnarray*} 
\qed

\subsection{Proof of Theorem \ref{thm:final}}
The main ingredient of the proof is based on the intermediary results  stated in Theorems \ref{thm:21} and \ref{thm:31}.
The first part of the statement follows by utilizing  Theorem \ref{thm:21} in order to conclude that $S \subseteq \hat S_\tau$, with probability close to 1.
Unfortunately, as conditions  of Theorems \ref{thm:31}  contradict those of Theorem \ref{thm:21}, we cannot easily use their results  to conclude the second part of the statement. Therefore, this paper
develops  and presents a new method for finding the optimal value of the tuning parameter $\lambda_N$. It is based on finding the optimal bias-variance tradeoff, where bias is replaced with  variable selection  error and  variance with prediction error.
 It allows  good, but not the best, prediction properties while obtaining desired variable selection properties. We  split the proof into two parts. The first   bounds the number of false positives, whereas the second     finds the optimal choice of $\lambda_N$.

\subsubsection{Bounding False Positives}


Let $\hat S_i (\lambda_N) = \cup_{k=1}^K \hat S_i(\lambda_N,k)$. Assume that the weighted maximal-contrast subbaging procedure is not worse than a random guess (see Theorem 1 of \cite{MB11}), i.e. for 
\begin{equation}\label{eq:rg}
|S^c| { \mathbb{E}\left [ |S \cap \hat S_i(\lambda_N)| \right]}{} \geq {|S|}{}{\mathbb{E}}\left [|S^c \cap \hat S_i(\lambda_N)|\right]
\end{equation}
then the expected number of falsely selected variables is bounded by
$$
{\mathbb{E}}[|S^c \cap \hat S_\tau|]\leq   2 \frac{\sqrt{b}}{1+\sqrt{b}}  \frac{ K^M  \left(\max_{1 \leq k \leq K} {\mathbb{E}}|\hat S_1 (\lambda_N,  k)| \right)^M}{ p^{M-1}},
$$
for all choices of $\tau \geq  \frac{1}{2(1+\sqrt{b})}$.

{\it Proof}


 Define a binary random variable 
$
H_K^\lambda=1\{ j \subseteq \cap_{l=1}^m \hat S_{mq+1-l}(\lambda_N) \}
$
for all variables $j \subset \{1,\dots,p\}$ with $   \hat S_i (\lambda_N) = \cup _{k=1}^K \hat S_i(\lambda_N, k) $ and $q=1,\cdots,b$. Remember that  $mb=d$.   Then, the   selection probability is expressed as a function of simultaneous selection probability
\begin{align*}
{\pi}^*_j(\lambda_N)&:= \frac{\sqrt{b}}{1+\sqrt{b}}\mathbb P^*(j \subseteq \cap_{l=1}^m \hat S_{mq+1-l}  (\lambda_N)) + \frac{1}{2(1+\sqrt{b})}\\
&=\frac{\sqrt{b}}{1+\sqrt{b}}E(H_j^\lambda | Z) +\frac{1}{2(1+\sqrt{b})}
\end{align*}
where  the probability $\mathbb P^*$ denotes probability with respect to the random sample splitting. 
Here $Z=(X_1,\dots,X_n, Y_1,\dots,Y_n)$ denotes the whole original sample of size $n$. 
Then, for $p_j^*(\lambda_N)=\mathbb P^*(j \subseteq \cap_{l=1}^m \hat S_{mq+1-l}(\lambda_N))$, we have
\[
P(\pi_j^*(\lambda_N) \geq {\tau})   = P\biggl( p_j^* (\lambda_N) \geq \frac{1+\sqrt{b}}{\sqrt{b}} \biggl( \tau - \frac{1}{2(1+\sqrt{b})}\biggl)  \biggl)
\]
Then,
\begin{align*}
E(H_j^\lambda|Z)&=E({p}^*_j(\lambda_N))=P(H_K^\lambda=1) \\
&\leq \prod_{l=1}^m P(K \subseteq  \hat S_{mq+1-l})=P(j \subseteq \hat S_1(\lambda_N))^m.
\end{align*}
Here  $\hat S_1 (\lambda_N)  $ denotes  $\cup _{k=1}^K \hat S_1 (\lambda_N, k)$.   
By Markov inequality  for exchangeable Bernoulli random variables,  we know that 
\begin{align*}
P({p}^*_j(\lambda_N) \geq \xi) \leq {E({p}^*_j(\lambda_N))}/{\xi} \leq {P(j \subseteq \hat S_1(\lambda_N))^m}/{\xi}.
\end{align*}
By arguments similar to that of Theorem 1 in \cite{MB11}, we know that 
$$
P(j \subseteq \hat S_1(\lambda_N)) \leq {E|\hat S_1(\lambda_N)|}/{p} \leq  {\max_{1 \leq k \leq K} K E|\hat S_1(\lambda_N,  k)|}/{p}.
$$
%
%

Hence,
for a threshold $\tau \geq  \frac{1}{2(1+\sqrt{b})}$ we have 
\begin{eqnarray*}
P(\pi_j^*(\lambda_N) \geq {\tau})  &\leq&   \frac{ \sqrt{b}{K^m(\max_{1 \leq k \leq K} E|\hat S_1 (\lambda_N, k)| )^m}}{(1+\sqrt{b})(\tau - \frac{1}{2(1+\sqrt{b})}) p^m} 
\\
&\leq& 2 \frac{\sqrt{b}}{1+\sqrt{b}} \frac{ K^m  (\max_{1 \leq k \leq K} E|\hat S_1 (\lambda_N,  k)| )^m}{ p^m}.
\end{eqnarray*}
Together with
$
E[|S^c\cap \hat S_\tau|]=\sum_{j \in S^c} P(\max_{\lambda_N} \pi_j^*(\lambda_N) \geq \tau)$, it leads to the 
 $$E[|S^c \cap \hat S_\tau|] \leq   2 \frac{\sqrt{b}}{1+\sqrt{b}}  \frac{ K^m  (\max_{1 \leq k \leq K} E|\hat S_1(\lambda_N, k)| )^m}{ p^{m-1}}.
$$

\subsubsection{ Optimal choice of $\lambda_N$ }

Next, we show that proposed aggregated sub-Lasso estimators are better than the random guess  in the sense of \eqref{eq:rg}. As expected, such property is not  achieved for all values of $\lambda_N$.  
By analyzing equation \eqref{eq:rg} and utilizing results of Theorem \ref{thm:21}, we infer that the sequence of $\lambda_N$ that achieves control of false positives and allows results of Theorem \ref{thm:21} to hold, is the optimum of the following  optimization problem
\begin{eqnarray}\label{eq:so}
&\min & q  \\ \nonumber
& s.t.& E|{\hat{S}}_i(\lambda_N)| \leq \frac{1}{2}p^{1-1/m}
\\\nonumber
& & P(\mathcal A_q (\lambda_N)) \geq 1- \delta
\\\nonumber
&& \lambda_N >0, q \geq 0, \delta >0
\end{eqnarray}
where  the events 
$$
\mathcal A_q (\lambda_N) =\bigcap_{j=1}^p \left\{2 \frac{1}{n} \sum_{l \in I_i} w_{k,l} |\varepsilon_i X_{lj}| \leq \lambda_N + q \right\}.
$$ Although the problem \eqref{eq:so}
 is stated in terms of $q$, the paper demonstrates that the optimal value of $q$ leads to the optimal value of $\lambda_N$.
 
 We provide  a few comments on the  optimization problem  \eqref{eq:so}.  
While allowing deviations of  the {\bf{IR}}$(N)$  conditions, the first constraint is sufficient to guarantee that sub-Lasso estimators are better than random guessing (i.e. that \eqref{eq:rg} is satisfied). The second constraint restricts our attention to a sequence of random coverage sets $\mathcal{A}_q(\lambda_N)$. They control  variable selection properties, whereas  the first constraint intrinsically controls predictive properties. Hence, they cannot be satisfied simultaneously on sets  $\mathcal{A}_0(\lambda_N)$.  For $q=0$ the best prediction is still achievable,  but variable selection   is not. Hence, we need to allow for possible deviation of the smallest  sets $\mathcal{A}_0(\lambda_N)$  by allowing small perturbations of size  $q$.   Our goal is to find the smallest possible perturbation $q$, which allows high probability bounds on the selection of the false negatives and     simultaneously  controls the size of the selected sets in Sub-Lasso estimators.

We further represent conditions of the stochastic problem \eqref{eq:so} in a concise way.
Note that from KKT conditions of each of the sub-Lasso problems and definition of the sets $\mathcal{A}_q$, we can see that for all $j$ such that $\hat{ \beta} _{i:k,j}(\lambda_N) \neq 0$  
\[
\frac{1}{n} \sum_{l \in I_i} w_{k,l} | \left[ \tilde  Y_l- \tilde \bX_l   \hat{\boldsymbol\beta} _{i:k}(\lambda_N) \right] X_{lj}  |    =   {\lambda_N} \mbox{sign}(\hat{ \beta} _{i:k,j}(\lambda_N)).
\]
Moreover, triangle inequality upper bounds the LHS with 
\[
\frac{1}{n} \sum_{l \in I_i} w_{k,l} | \left[ \tilde  \bX_l \bbeta^*- \tilde \bX_l   \hat{\boldsymbol\beta} _{i:k}(\lambda_N) \right] X_{lj}  | + \frac{1}{n} \sum_{l \in I_i} |w_{k,l} X_{lj} \varepsilon_l|.\]
On the set $\mathcal{A}_q(\lambda_N)$ we have that the last term is bounded with $(\lambda_N +q)/2$. This leads to 
\begin{equation}\label{eqn:brt09}
\frac{1}{n} \sum_{l \in I_i} w_{k,l} \left\| \left[ \tilde \bX_l  \bbeta^*- \tilde \bX_l   \hat{\boldsymbol\beta} _{i:k}(\lambda_N) \right] X_{lj}  \right\|_\infty \geq {\lambda_N}   - \frac{\lambda_N+q}{2}.
\end{equation}
Then, on the event $\mathcal{A}_q$ from  the from the KKT conditions of $\hat{\boldsymbol \beta} _{i:k }(\lambda_N) \neq 0$ for all $j$, such that $\hat{ \beta} _{i:k,j}(\lambda_N) \neq 0$,
\begin{eqnarray*}
&=& \frac{1}{n^2} \sum_{j=1}^p  \left\| {\tilde \bX_{k,j}}^T  \left(\tilde \bX_k \bbeta^*-\tilde \bX_k \hat{\boldsymbol \beta} _{i:k }(\lambda_N) \right)\right\|_2^2
\\
&\stackrel{(i)}{\geq}& \frac{1}{n^2} \sum_{j \in \hat S_i(\lambda_N,k)  }  \left\| {\tilde \bX_{k,j}}^T  \left(\tilde \bX_k \bbeta^*-\tilde \bX_k \hat{\boldsymbol \beta} _{i:k }(\lambda_N) \right)\right\|_2^2
\\
&\stackrel{(ii)}{=}& |{\hat{S}}_i(\lambda_N,k)| (\lambda_N-q)^2/4,
\end{eqnarray*}
where $(i)$ follows from the non-negativity of the summands and $(ii)$   from \eqref{eqn:brt09} and   inequality of the vector norms $\| \bx\|_\infty \leq \| \bx \|_2$, for a  vector $\bx$.
All of the above leads to 
\begin{eqnarray}\label{eq:temp17}
|{\hat{S}}_i(\lambda_N,k)| 
&\stackrel{(i)}{\leq}&\frac{4 \lambda_{\max}( { \tilde{\mathbf X}_{k}}^T \tilde{ \mathbf X}_{k} )}{ (\lambda_N -q)^2 n^2} \left\| \tilde {\mathbf X}_k \left[{\boldsymbol\beta}^* -  \hat{\boldsymbol\beta} _{i:k}(\lambda_N)\right] \right\|_2^2,
\end{eqnarray}
where inequality $(i)$ follows  from the above manipulations and   inequality of the norms $\|\bM^T \bx \|_2^2 \leq \lambda_{\max}(\bM^T \bM) \| \bx\|_2^2$ , with a matrix $\bM$ and a vector $\bx$. 
Moreover,  from Lemma \ref{thm:1} (ii)  we have that
$$
\frac{1}{n} \| \tilde {\mathbf X}_k ({\boldsymbol\beta}^* - \hat{\boldsymbol\beta} _{i:k}(\lambda_N)) \|_2^2 \leq \frac{(16 \lambda_N^2 +q^2
) s}{\zeta_N^2}.
$$   The tower property of expectations together, with \eqref{eq:temp17},
 \begin{equation}\label{eq:temp17a}
{\mathbb E}|{\hat{S}}_i(\lambda_N)| \leq  \frac{ (64\lambda_N^2  + 4q^2 )s K }{n (\lambda_N-q)^2 \zeta_N^2}  \max_{1\leq k \leq K} \mathbb{E}_{\bw} \left(\lambda_{\max}({ \tilde{\mathbf X}_{k}}^T \tilde{ \mathbf X}_{k} ) \right).
\end{equation}
In the above expressions $\lambda_N \geq c_4 \sqrt{\log p/n}$, for some $c_4>0$.
We are left to evaluate the size of sets $\mathcal{A}_q$. This step of the proof is based on a Bernstein's inequality for the exchangeable weighted bootstrap sequences contained in an intermediary result Lemma \ref{lem:a}. 
From Lemma  \ref{lem:a} we have

 \[
\mathbb P \biggl(    \Bigl| \left \langle \bvarepsilon_{I_i} ,   \bD_{  \mathbf w} \bX_{I_i}  \right \rangle_n \Bigl| > u_n \biggl) \leq  
\exp \left\{ N \log   \mathbf w _{2}  -\frac{n^2 u_n^2 }{ 2\sigma^2   N \|  \bX_{ I_i} \|_{\infty,2}  +2  n  c u_n \|  \bX_{ I_i} \|_{\infty,\infty}} \right\},
 \]
with $ \mathbf w _{2 } $ defined in Condition \ref{cond:w}. In display above, 
 $\|  \bX_{ I_i} \|_{\infty,2}:= \max{} \Bigl \{ { X}_{I_ij}^2    :  I_i \subset \{1,\cdots, n \},  \Bigl.$ $\Bigl. |I_i|=N  , 1 \leq j \leq p\Bigl\}$
  and
  $ \|  \bX_{ I_i} \|_{\infty,\infty}:= \max \Bigl \{|{ X} |_{I_i j}   :  I_i \subset \{1,\cdots, n \}, |I_i|=N, 1\leq j \leq p \Bigl \}$.
%
%
%
%
%
Hence, for $u_n =t \sigma \sqrt{2}$,  and $t$ such that for    two constants $c_5$ and $c_6$ 
\begin{eqnarray}\label{eq:temp17b}
c_6 \sqrt{\log p/ N } \geq  t \geq   c_5 \sqrt{\log p/ n } 
\end{eqnarray} 
and $  \|  \bX_{ I_i} \|_{\infty,\infty} \leq  c_6 \sqrt{ N /  \log p \    }$ 
  then there exists a positive constant $c_0>1$
\[
P\left( \frac{1}{n}\bigl| \sum_{l \in I_i} w_{k,l} X_{lj} \varepsilon_l \bigl| >{t} \sigma \sqrt{2}  \right) \leq 2  \exp \left\{ -c_0 \frac{n t^2 }{ \|  \bX_{ I_i} \|_{\infty,2} } \right\}.
\]
In particular, any choice of $ c_5 <1<c_6$  would work; even those  as large as  $c_5 = N/n$, $c_6 =n$ satisfy previous constraints.
For a choice of $t$ of $\lambda_N - q$ we then  have 
\begin{eqnarray}\label{eq:temp18}
P(\mathcal{A}_q(\lambda_N))\geq1-2 \exp\left\{ -c n (\lambda_N - q)^2/\|  \bX_{ I_i} \|_{\infty,2}\right\}  \to 1,
\end{eqnarray} 
as long as 
$
\log  \mathbf w _{2 }^N < n/N \log p
$, 
which in turn is guaranteed by the Condition \ref{cond:w}.\\

Now with \eqref{eq:temp17a}, \eqref{eq:temp17b} and \eqref{eq:temp18} we can represent the solution to the stochastic optimization problem \eqref{eq:so} as a solution to the following program  
\begin{eqnarray}\label{eq:so1}
&\min &  q \geq 0\\ \nonumber
& s.t.& \frac{(\lambda_N -q)^2}{\lambda_N^2 +q^2/4} \geq {\frac{16 s  p^{1/m} }{N \zeta_N^2 p }}\max_{1\leq k \leq K} \mathbb{E}_{\bw} \left(\lambda_{\max}({ \tilde{\mathbf X}_{k}}^T \tilde{ \mathbf X}_{k} ) \right),\\\nonumber
&& P(\mathcal{A}_q(\lambda_N))\geq1-2 \exp\left\{ -c n (\lambda_N - q)^2/\|  \bX_{ I_i} \|_{\infty,2}\right\}  \\\nonumber
&& \max\{c_4,c_5\} \sigma \sqrt{\frac{\log p}{n}} < \lambda_N \leq c_6 \sigma \sqrt{\frac{\log p}{N}}.
\end{eqnarray}
for   constants $ \max\{c_4,c_5\}  < c_6   $.
The RHS of the last constraint inequality is a consequence of Lemma \ref{thm:1} (which is used numerous times in the steps of the proof).
The first constraint of the above problem can be reformulated as 
\[
\lambda_N^2 >  q^2+{32 q^2 {\frac{ s  p^{1/m}   }{N \zeta_N^2 p }}} \Lambda_K \left({1-\frac{16 s  p^{1/m} }{N \zeta_N^2 p }} \Lambda_K \right)^{-1} ,
\]
with $\Lambda_K= \max_{1\leq k \leq K} \mathbb{E}_{\bw} \left(\lambda_{\max}({ \tilde{\mathbf X}_{k}}^T \tilde{ \mathbf X}_{k} ) \right)$.
Then we can see that the optimal values of $q$ are  of the order of $  c_7\sigma \sqrt{\frac{\log p}{n}}$ for a constant $c_7>0 $ that satisfies  
$$
c_7 \leq \sqrt{\frac{n-N}{n}} \frac{\sqrt{\zeta_N} + 2 \sqrt{  \Lambda_K}}{2 \sqrt{\zeta_N} + \sqrt{\zeta_N} / \lambda_{\min}(\bX_{ I_i^c}^T \bX_{ I_i^c})} .
$$

Notice that the  optimal value  of $q$ allows sets  $\mathcal{A}_q(\lambda_N)$ to have large coverage probability.
$c_4,c_5,c_6$ and $c_7$ are constants; they satisfy $0<     \max\{c_4,c_5\} +c_7 < c_6$. They are not close  ,  as there is a great deal of latitude as to which number once can choose. 
 For example, constant $c_6$ can be chosen to be $\max\{c_4,c_5\} +1$ as constant $c_7 \leq1$.
All of the above 
 results  in the choice of the optimal value of the tuning parameter $\lambda_N$ as follows
$ c_8 \sigma \sqrt{\frac{\log p}{n}} \leq  \lambda_N \leq  c_9 \sigma \sqrt{\frac{\log p}{n}}$, 
where 
$c_8=\max\{c_1,c_4,c_5\} \leq 1$
and 
$c_9 =\min\{c_2,(c_6-c_7)\} >1$.

 \qed


 \subsection{Proof of Theorem \ref{thm:optimal}}
 
\subsubsection{Part(a)}
To prove the part (a), we use the standard technique of Fanno's lemma  in order to reduce the minimax bound to one problem of testing $M+1$ hypothesis. We split $p$ covariates into $M \geq s$ disjoint subsets $J_1,\dots, J_M$, each of size $p/M$.  Let  $ {J}_l$ be a collection of disjoint sets each of sparsity $s$, which we denote with $\{\mathcal{J}_l\}$. Hence, each subset $J_l$ is a collection of ${p/M\choose s}$, $s$-sparse sets. We proceed by defining probability measures $\mu_l$ on the $\mathbb{B}_0(s)$ ball   to be Dirac measures at $\bbeta_l$
 where $\bbeta_l$ is chosen as follows. We define $\bbeta_l$ for $l \neq 0$ as  linear combination of vectors in $\mathbb{Z}^p$
 \[
 \bbeta_l = \sum_{ \mathbf z \in \mathbb{Z}^p} \theta_{\mathbf z,l}  \mathbf z,
 \]
 with $\theta_{\mathbf z,l} =1$ if $\mathbf z=(z_1,\dots,z_p)=(1\{1\in \mathcal{J}_l\},\dots,1\{p\in \mathcal{J}_l\})$ and zero otherwise.
Obviously, all $\bbeta_l \in \mathbb{B}_0(s)$ and have $\mathcal{J}_l$ as sparsity pattern. 
 These measures $\mu_l$ are chosen in such a way that for each $l \geq 1$ there exists a set $\mathcal{J}_l$ of cardinality $s$ such that 
 $\mu_l\{S \subseteq \mathcal{J}_l\}=1$
 and all the sets $\mathcal{J}_l$ are  distinct.  The measure $
 \mu_0$  is the Dirac  measure at  $0$. 
 Consider these $\mu_l$ as priors on $\mathbb{B}_0(s)$ ball and define the corresponding posteriors $\mathbb{P}_0,\dots,\mathbb{P}_M$ by
$
\mathbb{P}_l(A) =\int_{\boldsymbol \beta \in \mathbb{B}_0(s)} P_{\boldsymbol\beta}(A) d \mu_l(   \bbeta).
 $

With this choice of $\bbeta_l$ we can easily check  
\begin{equation}\label{eq:max1}\mathcal{K}(\mathbb P_{\boldsymbol\beta_l},\mathbb P_{0}) \leq n \rho_n,\end{equation}
where $\mathcal{K}$ denotes the Kullback-Leibler divergence between two probability measures and for   $\rho_n   < \infty$ and such that 
\begin{equation}\label{eq:max}
\max\left\{ \sum_{l =1}^n \frac{ E (\mathbf X_l \mathbf v )^2  }{\sqrt{n} \| \mathbf v_S\|_2}: |S|\leq s, \mathbf v \in \mathbb{R}^p, \mathbf v \neq 0, \mathbf v \in \mathbb{C}(3, S) \right\} <\rho_n.
\end{equation}
Next, observe that
\[
\inf_{ J} \sup_{\boldsymbol\beta \in \mathbb{B}_0(s)} P_{\boldsymbol\beta}\left( S \not\subseteq J\right) \geq \inf_{ J}  \sup_{l=1,\cdots, M} \sup_{\bbeta \in \Theta_l}P_{\bbeta} \left( S \not\subseteq J \right)
\]
for $\Theta_l = \{\bbeta: \mbox{supp}(\bbeta) \subset \mathcal{J}_l\}$. By   the Scheffe's theorem and the first Pinsker's inequality (see Lemma 2.1 and 2.6 of \cite{T09}),   the RHS above can be lower bounded with
\[
  1-  \frac{1}{M}\sum_{l=1}^M  \| \mathbb{P}_{ \bbeta_l}- \mathbb{P}_0 \|_{TV}  \geq 1-  \frac{1}{M}\sum_{l=1}^M (1-\frac{1}{2}\exp\{-\mathcal{K}(\mathbb P_{\boldsymbol\beta_l},\mathbb P_{0})\}),
\]
where $\|\|_{TV}$ denotes total variation distance between two probability measures. 
Notice that we can choose the sets within a collection $\mathcal{J}_l$ into $ {p/M \choose s}$ ways. Together with \eqref{eq:max1}   we have
\[
\inf_{ J} \sup_{\boldsymbol\beta \in \mathbb{B}_0(s)} P_{\boldsymbol\beta}\left( S \not\subseteq J\right) \geq  \frac{1}{2} \exp\{ \log {p/M \choose s} -   n \rho_n \} \geq   \frac{1}{2} \exp\{ s^2 \log (p/s^2) - n \rho_n\}.
\]
It suffixes to notice that the RHS is bigger than $p^{1-c'}$ under conditions of the theorem.

 \subsubsection{Part(b)}
To prove part (b), we  use  the  Assouad's lemma with appropriately chosen  hypothesis    to reduce the minimax bound to $Q$ problems of testing only $2$ hypothesis.
Consider the set of all binary sequences of length $p$ that have exactly $s$ non zero elements, 
$$\Omega=\{\omega=(\omega_1,\dots,\omega_p), \omega_i \in\{0,1\} : \| \omega\|_0=s\}  .$$
 Note that the cardinality of this set is $|\Omega|=2^s{p \choose s}$.
 Let $\rho(\omega,\omega')$ be the Hamming distance between $\omega$ and $\omega'$, that is $\rho(\omega,\omega') = \sum_{q=1}^Q 1\{\omega_q \neq \omega'_q\}$. 
 First,  the focus is on accessing the cardinality of the set $\{\omega' \in \Omega: \rho(\omega,\omega')\leq1\}$. Observe that one can choose a subset of size $1 $ where $\omega$ and $\omega'$ agree and then choose the other $s-1$ coordinates arbitrarily. Hence, the cardinality is less than   $3{p \choose 1}$.
Now consider the set $A \subset \Omega$ such that $|A| \leq {p \choose s}/{p \choose 1} \leq  {  (p-s)}$.
The set of elements $\omega \in \Omega$ that are within Hamming distance $1$ of some element of $A$ has cardinality of at most
$
|A|  {3}{p \choose 1 } < |\Omega|.
$
Therefore, for any such set with cardinality $|A|$, there exists an $\omega \in \Omega$ such that $\rho(\omega,\omega')>1 $ .
%
The expected number of false positives of an estimator $J$ is given by
\[
E_{\omega}|J\setminus S_\omega| = \sum_{q=1}^Q E_\omega d_q(J,\omega_q)
\]
 with 
$
 d_q(J,\omega_q)=  \rho(1\{q \in J\},w_q)
 $ and $\rho$ as Hamming distance. Define the statistic $\omega'_q = \arg\min_{t=0,1} d_q(J,t)$. Then,
  by the  definition of $\omega ' _q$ we have  $d_q(\omega_q',\omega_q)  = 
  |\omega_q - \omega^{'}_q|  
   \leq d_q(J,\omega_q' )+d_q(J,\omega_q)\leq 2 d_q(J,\omega_q)$.
   Moreover,
\[
E_{\omega}|J\setminus S_\omega| \geq \frac{1}{2}  \sum_{q=1}^p E_\omega  |\omega'_q-\omega_q|  = \frac{1}{2} E_\omega\rho(\omega',\omega).
\]
Therefore, from Assouad's Lemma (see Theorem 2.12 of \cite{T09})  we have 
 \[
\inf_J\sup_{\boldsymbol\beta \in  \mathbb{B}_0(s)} E_{\boldsymbol\beta}|S^c \cap J|  \geq   \frac{1}{2 } \inf_{\omega'} \sup_{\omega \in \Omega} E_{\omega} \rho(\omega',\omega) 
\]
\[
\geq  \frac{ 1}{4}  2^s {p \choose s} \max\left\{ \exp\{-\alpha\}, (1-\sqrt{\alpha/2}) \right\}
\]
as long as 
$
\mathcal{K}(\mathbb P_{\omega'},\mathbb P_{\omega}) \leq \alpha <\infty.
$
 Straight forward computation shows that $\mathcal{K}(P_{\omega'},P_{\omega})  \leq n \rho_n$.

\qed

\subsection{Proof of Theorem \ref{thm:var}}
Proof follows simple computations using result of Theorem 1 and equations  \eqref{eq:weak} and \eqref{eq:weak1}.

The risk of the estimator $\pi^*(\lambda_n)$ is equal to $E (\pi^*(\lambda_n) - p_j)^2 = \frac{1}{4} \frac{1}{(1+\sqrt{b})^2}.$
Moreover, 
\begin{eqnarray*}
\mbox{var}\left(\pi_j^*(\lambda_n)\right)  &\leq&  \frac{1}{4} \frac{1}{(1+\sqrt{b})^2} \\
&\stackrel{(i)}{<}&  \frac{p_j (1-p_j)}{b}\\
&\stackrel{(ii)}{\leq}& \frac{ P(\hat \beta_j (\lambda_n) \neq 0) \left(1-P(\hat \beta_j (\lambda_n) \neq 0)\right)}{b}\\
&=&\mbox{var}\left(\pi_j (\lambda_n)\right) 
\end{eqnarray*}
where $(i)$ holds for all $p_j \in (1/2 - c_n,1/2+c_n)$ and for small values of $b$, $c_n \sim 1/2$ and $(ii)$  follows from Theorem 1 and equations  \eqref{eq:weak} and \eqref{eq:weak1}.
\qed

\subsection{Proof of Theorem \ref{prop:1}}
  Lemmas  \ref{lem:kkt} and \ref{lem:2} are stated for general sub-Lasso estimator and can  easily be adapted to case of   bagged estimator. With their help and results of Theorems \ref{thm:21} and \ref{thm:31}, we  are ready to finalize the proof of Theorem \ref{prop:1}.
 Equivalent of Lemma \ref{lem:kkt} requires 
 $
 |\sum_{l\notin I_i} (Y_l- \mathbf X_l^T \hat {\boldsymbol\beta}^b(\lambda_n^1))X_{lj}|\leq n\lambda_n- n\lambda_n^1
$
to hold, whereas  equivalent of Lemma \ref{lem:2}  requires 
 $
 |\sum_{l\notin I_i}  (Y_l- \mathbf X_l  \hat{\boldsymbol\beta}(\lambda_n) X_{lj}|   \leq {\lambda_n^1}/n- n\lambda_n
 $
 to hold.
 The proof follows easily as a consequence of results obtained in Theorems \ref{thm:final}, \ref{thm:21} and \ref{thm:31}. 
 
Approximating sparse recovery is not possible as a consequence of the proof of  Theorem \ref{thm:final}. For the bagged estimator \eqref{eq:subagging}, there exists no feasible $q$ that is different from zero, which  solves the equivalent of \eqref{eq:so1}.  The equivalent of \eqref{eq:so1} would require, on one side $q > \sqrt{\log p / N}$ and on the other $q < \sqrt{\log p /n}$. For $N \ll n$ this is not possible as $\sqrt{\log p / N} > \sqrt{\log p / n}$. For a special case of $N = n/k$,  only the choice of $k=1$ and fixed, not divergent  $s$ and $p$, allows both conditions to be satisfied.
 
 Second, on the subject of the exact sparse recovery, as a consequence of previous equivalent of Lemma \ref{lem:2} and  Theorem \ref{thm:21},  
 \begin{eqnarray}  \label{eq:bag1}
\lambda_n 
 \geq  \lambda_n^1 +
  \frac{2 c_9 \lambda_n^1 s ^2 }{n \zeta_N^{5 }} + \frac{2 c_{10} \lambda_n^1 s^{3/2}  }{n \zeta_N  } + \frac{ \sqrt{2} c_{11}  \lambda_n^1 s^{3/2} }{n \zeta_n  \zeta_N^{3}} 
  +2 \sigma \sqrt{\frac{(n-N) \log p}{n^2}} +2 \sigma \sqrt{\frac{ 4 r  \log p}{n^2 (n- N) }},
\end{eqnarray} 
for some universal, positive and bounded constants $c_9,c_{10} ,c_{11} $.
 As a consequence of equivalent of Lemma \ref{lem:kkt} and  Theorem \ref{thm:31} (where a factor of $1/n$ is lost due to the fact that all weights are equal to $1$)  
 \begin{eqnarray}  \label{eq:bag2}
 {\lambda_n^1}{  } 
&\geq&    \lambda_n +\frac{\lambda_n}{n} + \frac{c_{12} \lambda_n  s^{3/2}  }{n {\zeta_n^2 \zeta_{N-n}}}
 +
 \frac{\sqrt{2}c_{13}  \lambda_n  s^{3/2}}{n \zeta_n \zeta_{N-n}^3}  +2 \sigma \sqrt{\frac{(n-N) \log p}{n^2}} +2 \sigma \sqrt{\frac{ 4 C' s  \log p}{n^2 (n- N) }},
\end{eqnarray}
for some universal, positive and bounded constants $ c_{12} ,c_{13} $.

 If $N \leq n$ then, from the above contradictory conditions   one can see that for all fixed $\lambda_n$ and all  $j \in \hat S(\lambda_n)$, for all $\lambda_n^1$, $P (j \in \hat S_i (\lambda_n^1)) =0$.
 Moreover,
   we  employ the Massart's Dvoretzky-Kiefer-Wolfowitz  inequality  to bound   the distance between an empirically determined distribution function and  the population distribution function. Hence,  
\begin{eqnarray*}\nonumber
&&P \left( \frac{1}{d} \sum_{i=1}^d \mathbbm{1}( j \in \hat S_i(\lambda_n^1)) \leq \frac{1}{2}P( j \in \hat S_1(\lambda_n^1)) \right)
\\\nonumber
&&\leq  P \left( \sup_{\lambda \in (0,\lambda_n^1]} \sqrt{n} \left| \frac{1}{d} \sum_{i=1}^d \mathbbm{1}( j \in \hat S_i(\lambda_n^1))  - P( j \in \hat S_1(\lambda_n^1))  \right| \geq \sqrt{d} /2\  P( j \in \hat S_1(\lambda_n^1))  \right)
\\
&&\leq 2 e^{-  d P( j \in \hat S_1(\lambda_n^1)) ^2 /2 }.
\label{eq:dvoretsky}
\end{eqnarray*}
 As we have shown that for all $\lambda_n^1$, $P (j \in \hat S_i (\lambda_n^1)) =0$, it follows that
 subagged estimator does not have the  same sparsity set    as  the Lasso estimator, i.e.
\[
P \left( \exists \lambda_n \geq 0, \exists  \lambda_n^1 \geq 0: \hat S(\lambda_n)= \hat S^b(\lambda_n^1)\right) =0.
\]
For a special case of $N = n/k$ we see that the only choice of $k=1$ and fixed, not divergent  $s$ and $p$  allows equations \eqref{eq:bag1} and \eqref{eq:bag2} to be satisfied up to a constant. That is, there exist two constants $0<c<\infty$ and $0<c_1<\infty$ such that for the choice of  $\lambda_n=c \lambda_n^1 \geq c_1 n^{-1/2}$ (i.e. result of \cite{B08} only holds for fixed  $p$)
\qed

\subsection{Proof of Theorem \ref{prop:bootstrap}}
  
  If  Condition \ref{cond:re} holds on the bootstrapped data matrix, then the result of this Theorem follows by repeating the steps of the proof of Lemma \ref{thm:1} with simplification of no weighting scheme $\mathbf w_k$, to obtain
  \[
  \sqrt{\sum_{j \in S} |{\hat {\boldsymbol \beta}}_{i:k,j} - \boldsymbol {\boldsymbol\beta}^*_j |^2}  \leq \frac{\left\|   X_{I_i} (\hat{\boldsymbol\beta} _{i:k}(\lambda_N) - {\boldsymbol\beta}^*) \right\|_2}{\zeta_N \sqrt{n}}.
  \]
  Following the steps parallel to those in  \cite{BRT09}, one can obtain  the predictive bounds of the order of $s \lambda_N/\zeta_N^2$, for $N=n/k$ and $\lambda_N \geq 2 \sigma \sqrt{2 k \log p/n}$.  From the classical results on Lasso prediction bounds, we know that optimal $\lambda_n \geq 2 \sigma \sqrt{2 \log p/n}$.
  The statement of the theorem follows, if we are able to bound the following expression
  \[
 \delta_n= \left| \frac{ \lambda_N}{\zeta_N^2} - \frac{\lambda_n}{\zeta_n^2}\right|  
  \]
We write $|\zeta_N^2 - \zeta_n^2|=\zeta_n^2-\zeta_N^2 =\epsilon_n$ for some $\epsilon_n \geq 0$. Then for $\lambda_N\leq 2 \sigma \sqrt{2 k \log p/n}$ we have
  \[
  \delta_n = \left| \frac{ \lambda_N}{\zeta_n^2- \epsilon_n} - \frac{\lambda_n}{\zeta_n^2}\right| < \frac{2 \sigma \sqrt{\frac{\log p}{n}}}{\zeta_n^2} \left( \sqrt{k} \frac{1}{1-\frac{\epsilon_n}{\zeta_n^2}}-1\right)  
    \]
now we claim that if  $k \leq 4$ then $ \sqrt{k} \frac{1}{1-\frac{\epsilon_n}{\zeta_n^2}}-1 \leq  C$ for some bounded constant $C>1$. This claim is equivalent to claiming that $\frac{1+ \epsilon_n/\zeta_n^2}{1-\epsilon_n/\zeta_n^2} \leq C$, that is $ \epsilon_n/\zeta_n^2 \leq \frac{C-1}{C+1}$.
However, from Condition \ref{cond:re} applied on the full data matrix $\bX$, we know that  $0 \leq \epsilon_n < \eta \zeta_n^2$ for some constant $\eta <1$. Hence, constant $C>1$ that satisfies above properties is $(\eta +1)/(1-\eta)$. Therefore, one can conclude that 
$ \delta_n \leq \frac{2 C \sigma }{\zeta_n^2} \sqrt{\frac{\log p}{n}}$.
  \qed

\section{Proofs of   Lemmas}

\subsection{Proof of Lemma \ref{lem:7}}
\begin{proof}

Observe that $ \mathbf X_{I,  A} (\mathbf X_{ I,A}^T \mathbf X_{ I, A} )^{-1}   \mathbf X_{ I, A}^T \mathbf X_{I,j}$ is a projection of $\bX_{I,j}$ onto space spanned by the columns of $\bX _{I,A}$. Moreover, $\| \bX_j\|_2^2 =1$ and 
\[
1 = \sum_{l=1}^n X_{lj}^2 \geq \sum_{l \in I}X_{lj}^2 = \|\bX_{I,j} \|_2^2.
\]
Therefore, by the properties of the  projection matrices
\[
\left\| \mathbf X_{I,  A} (\mathbf X_{ I,A}^T \mathbf X_{ I, A} )^{-1}   \mathbf X_{ I, A}^T \mathbf X_{I,j}\right\|_2^2 \leq 1.
\]
Moreover, observe that  
$$\| \bM \bx\|_2^2 = \bx^T \bM^T \bM \bx \geq \lambda_{\min} (\bM^T \bM) \| \bx\|_2^2.$$ With $\bM = \mathbf X_{I,  A} $ and $\bx = (\mathbf X_{ I,A}^T \mathbf X_{ I, A} )^{-1}   \mathbf X_{ I, A}^T \mathbf X_{I,j}$,  
\[
\left\|  \mathbf X_{I,  A} (\mathbf X_{ I,A}^T \mathbf X_{ I, A} )^{-1}   \mathbf X_{ I, A}^T \mathbf X_{I,j}\right\|_2^2 \geq  \lambda_{\min} \left( \mathbf X_{ I,A}^T \mathbf X_{ I, A}\right)  \left\| (\mathbf X_{ I,A}^T \mathbf X_{  I,A} )^{-1}   \mathbf X_{I,  A}^T \mathbf X_{I,j}\right\| _2^2.
\]
Next, notice that the last two inequalities combined lead to 
\begin{eqnarray*} \label{eq:step001}
\left\| (\mathbf X_{ I,A}^T \mathbf X_{ I, A} )^{-1}   \mathbf X_{  I,A}^T \mathbf X_{I,j}\right\| _2^2  \leq 1/ \lambda_{\min} \left( \mathbf X_{ I,A}^T \mathbf X_{ I, A}\right)  \leq 1/\zeta_N^2 ,
\end{eqnarray*}  
where in the last step we used Condition \ref{cond:re} with the vector $\bv = (1,\dots 1, 0, \dots, 0)^T$ and  the constant $a=1$ and have made a simple observation $\lambda_{\min} \left( \mathbf X_{ I,A}^T \mathbf X_{ I, A}\right) \geq \lambda_{\min} \left( \frac{1}{n}\mathbf X_{ I,A}^T \mathbf X_{ I, A}\right)$.
By the inequality of $l_p$ norms, $\| \bx\|_2 \geq \sqrt{r} \|\bx \|_1$ for all vectors $\bx \in \mathbb{R}^r$,
\begin{eqnarray*} \label{eq:step0100}
\left\| (\mathbf X_{ I,A}^T \mathbf X_{ I, A} )^{-1}   \mathbf X_{ I, A}^T \mathbf X_{I,j}\right\| _1^2 \leq r   /\zeta_N^2.
\end{eqnarray*} 

In addition, the following holds 
\[
1/\zeta_N^2 \geq \left\| (\mathbf X_{ I,A}^T \mathbf X_{ I, A} )^{-1}   \mathbf X_{  I,A}^T \mathbf X_{I,j}\right\| _2^2  \geq \lambda_{\min} \left( (\mathbf X_{ I,A}^T \mathbf X_{ I, A})^{-2} \right)\left\|  \mathbf X_{  I,A}^T \mathbf X_{I,j}\right\| _2^2,
\]
where the last step follows from the observation $\| \bM \bx\|_2^2 = \bx^T \bM^T \bM \bx \geq \lambda_{\min} (\bM^T \bM) \| \bx\|_2^2$, with $\bM = (\mathbf X_{ I,A}^T \mathbf X_{ I, A} )^{-1}$ and $\bx = \mathbf X_{  I,A}^T \mathbf X_{I,j}$.
Moreover,  utilizing the bound $\lambda_{\min}  (\bA^{-1}) = \lambda_{\min}^{-1}(\bA )$ for any positive semi-definite  matrix $\bA$,  
\begin{eqnarray} \label{eq:step03000}
\|  \mathbf X_{I ,  A}^T \mathbf X_{I ,j}\|_2 
&\leq&  \lambda_{\min}^{-1} (\mathbf X_{ I,A}^T \mathbf X_{  I,A} )  /  {\zeta_N }  \leq \inf_{|A | = r} \lambda_{\min}^{-1} \left (  \mathbf X_{ I,A}^T \mathbf X_{  I,A}\right)/ {\zeta_N} \leq \zeta _N ^{-3 },
\end{eqnarray}  
where Condition \ref{cond:re} guarantees  $\lambda_{\min} \left (  \mathbf X_{ I,A}^T \mathbf X_{  I,A}\right) \geq \zeta_N$ 
for any $A $, such that $|A| \leq r$.
\end{proof}

\subsection{Proof of    Lemma \ref{lem:a}}
 \begin{proof} 
 By simple Markov's inequality we have
\begin{align} \label{eq:w1}
 \mathbb P^*\left(   |\sum_{l \in I_i} w_l X_{lj} \varepsilon_l| > nu_n  \right) \leq \inf _{q \geq 0} \left\{ \exp\{-n q u_n\} \mathbb E^* \exp\{q |\sum_{l \in I_i} w_l X_{lj} \varepsilon_l|\} \right\}.
\end{align}
Observe that Condition \ref{cond:w} implies    that
random variables $\exp\{ w_l X_{lj} \varepsilon_l\}$ are negatively dependent. Hence the RHS can be upper bounded with
\[
\exp\{-n q u_n\}  \prod_{l\in I_i} \mathbb E^*  \exp\{q | w_l X_{lj} \varepsilon_l|\},
\]
for every $q \geq 0$.
 Let $\mathbf w=(w_1,\dots,w_N)$ be a vector of exchangeable random variables that satisfy Condition \ref{cond:w}.

 Let us define $\mathbf S $ to be  a random permutation   over the set of all combinations of $N$ sized subsets of $1,\dots, n$, by  requiring that $w_{\mathbf S(1)} \geq w_{\mathbf S(2)} \geq \cdots \geq w_{\mathbf S(N)}$ and if $w_{\mathbf S(j)}=w_{\mathbf S(j+1)}$ then $\mathbf S(j) <\mathbf S(j+1)$. This is one possible definition that is unambiguous to the presence of ties. 
 Let $\mathbf R$ denote a random permutation uniformly distributed over the set of all combinations of $N$ sized subsets of $1,\dots, n$. Note that $X_1,\dots, X_n$ are independent of $(\mathbf w, \mathbf R)$. Observe that $\mathbf R$ is independent of $(\mathbf w , \mathbf S)$.
Notice that $\mathbb P  = P_\epsilon \times P_w$. By exchangeability of vector $\mathbf w_k$, for $l \in I_i$ we have
\begin{align*}
  \mathbb E  \exp\{q| w_l X_{lj} \varepsilon_l|\}
   &= \mathbb E_w E_{\varepsilon}  \exp\{q| w_l X_{lj} \varepsilon_l|\}
    = \mathbb E_w E_{\varepsilon}  \exp q |  w_l X_{\mathbf R(l)  j} \varepsilon_{\mathbf R(l)  }|   \\
    &= \mathbb E_w E_{\varepsilon}  \exp q |w_{\mathbf S(l)} X_{\mathbf R \circ \mathbf S(l) j} \varepsilon_{\mathbf R \circ \mathbf S(l)}| .
\end{align*}
Let $\circ$ denote pointwise multiplication.
Observe that $\mathbf R \circ \mathbf S $ is   independent of $\mathbf S$ and has the same distribution as $\bR$.
Therefore, for an $l \in I_i$,
\begin{align}
  \mathbb E  \exp\{|q w_l X_{lj} \varepsilon_l|\} 
  &= 
    \mathbb E_w \Bigl[ E_{\varepsilon} \bigl[\exp  q|w_{\mathbf S(l)} X_{\mathbf R \circ \mathbf S (l)  j} \varepsilon_{\mathbf R \circ \mathbf S(l)}| \bigl] \Bigl]
\nonumber  \\
 &\stackrel{(i)}{\leq }
  \sqrt{  \mathbb E_w \Bigl[\exp  2|w_{\mathbf S(l)}| \Bigl] } \sqrt{ E_{\varepsilon} \Bigl[ \exp 2q |X_{\mathbf R(l)  j} \varepsilon_{\mathbf R \circ \mathbf S(l)}| \Bigl]}
\nonumber \\ 
 &
 \stackrel{(ii)}{\leq} \mathbf w _{2 } \sqrt{ E_{\varepsilon} \Bigl[ \exp 2 q |X_{\mathbf R(l)  j} \varepsilon_{\mathbf R (l)}| \Bigl]}, \label{eq:w2}
\end{align}
where $(i)$ follows from CauchyÐSchwarz inequality and $(ii)$ follows from
\[
\mathbf w _{2}^2= \mathbb E_w \exp  2|w_{\mathbf S(l)}| = \int_0^\infty  P_w \left(  \exp  2|w_{\mathbf S(l)}| \geq t \right)dt 
=
\int_0^\infty  P_w \left(    |w_{\mathbf S(l)}| \geq \frac{1}{2}\log t \right)dt .
\]
  

Next, observe that $P_w(X_i >a) \leq P_w(\sup_{i} X_i >a)$ holds for any $a \in \mathbb{R}$. Hence,
   \begin{align}\label{eq:w3}
   \mathbf w _{2} 
   \leq \sqrt{ \int_0^\infty  \sup_{l \in I_i} P_w \left(     w _{\mathbf S(l)} \geq   \frac{1}{2}\log t  \right)dt}
   \leq  \sqrt{ \int_0^{e^{2n}} P_w\left(    w _{\mathbf S(1)} \geq   \frac{1}{2}\log t  \right)dt}.
   \end{align}
  where in the last step we observed that  $ \max _ {l \in I_i} w_i \leq \sum_{l \in I_i} w_l = n$ by Condition \ref{cond:w}.
  
Furthermore,  the Taylor expansion around $0$ provides
 \[
E_{\varepsilon}  \exp q |X_{\mathbf R(l)  j} \varepsilon_{\mathbf R (l)}| = 1 + E_{\varepsilon}  q |X_{\mathbf R(l)  j} \varepsilon_{\mathbf R (l)}| + q^2|X_{\mathbf R(l)  j} |^2 \sum_{r=2}^\infty \frac{1}{r!} |q X_{\mathbf R(l)  j} |^{r-2}E_{\varepsilon}  |\varepsilon_{\mathbf R (l)}|^r.
 \]

 Since $E_{\varepsilon} [\varepsilon_i]=0$ and $E_{\varepsilon}  |\varepsilon_{\mathbf R (l)}|^r  \leq r! \sigma^2  c^{r-2} /2$
 we have 
 \[ E_{\varepsilon}  \exp q |X_{\mathbf R(l)  j} \varepsilon_{\mathbf R (l)}| \leq 1 + \frac{ |X_{\mathbf R(l)  j} |^2 \sigma^2 q^2}{2} \sum_{r=2}^\infty q^{r-2} |  X_{\mathbf R(l)  j} |^{r-2}  c^{r-2},
 \]
 for some constant $c<\infty$.
 
As $\log e^{\lambda x} \leq e^{\lambda x} -1$, for all $q \leq 1/c$, we have the following estimation of logarithmic moment generating function
 \[
\log   E_{\varepsilon}  \exp q |X_{\mathbf R(l)  j} \varepsilon_{\mathbf R (l)}| 
  \leq   |X_{\mathbf R(l)  j} |^2 \sigma^2 q^2    \left( 1- q c  |  X_{\mathbf R(l)  j} | \right)^{-1}.
 \]
 Observe that  $ |  X_{\mathbf R(l)  j} | \leq  \max_{1 \leq l \leq N} \|  \bX_{\mathbf R(l)} \|_\infty \leq \|  \bX_{ I_i} \|_{\infty,\infty} $. Hence, the logarithmic moment generating function satisfies
 \begin{align}\label{eq:w5}
\log   E_{\varepsilon}  \exp q |X_{\mathbf R(l)  j} \varepsilon_{\mathbf R (l)}| 
  \leq   |X_{\mathbf R(l)  j} |^2 \sigma^2 q^2    \left( 1- q c \|  \bX_{ I_i} \|_{\infty,\infty} \right)^{-1}.
 \end{align}
 
 Utilizing \eqref{eq:w1} - \eqref{eq:w5}
 
 \begin{align*}
   \mathbb P \left(   |\sum_{l \in I_i} w_l X_{lj} \varepsilon_l| > nu_n  \right)& \leq  e^{ N\log \mathbf w _{2}  }
 \inf _{q \geq 0} \Biggl\{ \exp \left\{-n q u_n
 +\frac{ \sigma^2 q^2}{2}  \sum_{l  =1}^N |X_{\mathbf R(l)  j} |^2   \left( 1- q c \|  \bX_{ I_i} \|_{\infty,\infty} \right)^{-1} \right\}\Biggl\}.
\end{align*}

 Since the right hand side   above   depends on $q$, we proceed to find the optimal $q$ that minimizes it.
This is simply done, and the optimal $q$  is 
\[
q = \frac{n u_n} {  \sigma^2 \sum_{l  =1}^N |X_{\mathbf R(l)  j} |^2  + c n u_n  \|  \bX_{ I_i} \|_{\infty,\infty}}.
\]
This optimal $q$ leads to   the bound
  \[
 \mathbb P \left(   |\sum_{l \in I_i} w_l X_{lj} \varepsilon_l| > nu_n  \right) \leq    {\mathbf w _{2} }^{N }  \exp \left \{ -  \frac{ n^2 u_n^2}{2  \sum_{l=1}^N |X_{\mathbf R(l)  j} |^2\sigma^2 +2n c  u_n  \|  \bX_{ I_i} \|_{\infty,\infty}}
 \right\}.
\]

By observing simple relations $ \sum_{l=1}^N |X_{\mathbf R(l)  j} |^2 =  \|X_{\mathbf R(l)  j} \circ  X_{\mathbf R(l)  j} \|_1  \leq N \max_{1 \leq l \leq N} \max_{1 \leq j \leq p} |X_{\mathbf R(l)  j} |^2$, with the last term being upper bounded with 
$N\|  \bX_{ I_i} \|_{\infty,2} = N\max_{1 \leq j \leq p}  |  \bX_{ I_i j} |^2$
we obtain
\begin{eqnarray*}
 \mathbb P\left(   |\sum_{l \in I_i} w_l X_{lj} \varepsilon_l| > nu_n  \right) \leq  
\exp \left\{ N \log   \mathbf w _{2}  -\frac{n^2 u_n^2 }{ 2\sigma^2   N \|  \bX_{ I_i} \|_{\infty,2}  +2  n  c u_n \|  \bX_{ I_i} \|_{\infty,\infty}} \right\}.
 \end{eqnarray*}

\end{proof}

\subsection{Proof of Lemma \ref{lem:kkt}}
\begin{proof}  
We want to show that  for all $j$ for which $\hat{\boldsymbol\beta}_{i:k,j} (\lambda_N)=0$ and \eqref{eq:kktsublassob} hold, equation \eqref{eq:kktlassob} also hold, that is 
\[
\left| \left\langle \bX_j,  \bY - \bX \hat\bbeta_{i:k} \right\rangle_n \right| \leq  \lambda_n.  
\]
As $\bw_k$ is a vector of a strictly positive random variables,    $\min_{1 \leq j \leq N} \sqrt{w_{k,j}} < \sum_{j=1}^N \sqrt{w_{k,j}}$. Hence, the desired inequality above 
 follows easily from the following inequality
\[
\left| \left\langle \bX_j,  \bY - \bX \hat\bbeta_{i:k} \right\rangle_n \right| 
  \leq 
 \left|  \left\langle \bD_{\sqrt \mathbf w_k} \bX_{I_i,j},  \bD_{\sqrt \mathbf w_k} \bY_{I_i} - \bD_{\sqrt \mathbf w_k}  \bX_{I_i} \hat\bbeta_{i:k} \right\rangle_n \right| 
  + 
\left|  \left\langle \bX_{I_i^c,j},  \bY_{I_i^c} - \bX_{I_i^c} \hat\bbeta_{i:k} \right\rangle_n \right| 
\]
where the first term in the rhs is bounded by $\lambda_N$ (by \eqref{eq:kktsublassob}) and the second with $\lambda_n-\lambda_N$  (by  the assumption of  the Lemma).
\end{proof}

\subsection{Proof of Lemma \ref{lem:2}}

 \begin{proof} 
 Let us assume that for those $j$ such that $\hat{\boldsymbol\beta} _{j}=0$,  \eqref{eq:kktlassob} holds. 
 We show that for such $j$'s, equation \eqref{eq:kktsublassob} also holds.  First we observe,
 \[
 \left| \left\langle \bX_j,  \bY - \bX \hat\bbeta  \right\rangle_n \right| 
 \geq 
 \left| 
 \left\langle \bX_{I_i,j},  \bY_{I_i} - \bX_{I_i} \hat\bbeta  \right\rangle_n 
 - 
 \left\langle \bX_{I_i^c,j},  \bY_{I_i^c} - \bX_{I_i^c}   \hat\bbeta \right\rangle_n 
 \right|
 \]
By analyzing two cases individually:

 Case (I): $ \left |\left\langle \bX_{I_i,j},  \bY_{I_i} - \bX_{I_i} \hat\bbeta  \right\rangle_n \right| \leq  \left|\left\langle \bX_{I_i^c,j},  \bY_{I_i^c} - \bX_{I_i^c}   \hat\bbeta \right\rangle_n \right| $,
 and 
 
  Case (II): $\left| \left\langle \bX_{I_i,j},  \bY_{I_i} - \bX_{I_i} \hat\bbeta  \right\rangle_n \right|  \geq  \left| \left\langle \bX_{I_i^c,j},  \bY_{I_i^c} - \bX_{I_i^c}    \hat\bbeta\right\rangle_n \right| $, we have
  \[
 \left |\left\langle \bX_{I_i,j},  \bY_{I_i} - \bX_{I_i}   \hat\bbeta  \right\rangle_n \right| \leq  \left|\left\langle \bX_{I_i^c,j},  \bY_{I_i^c} - \bX_{I_i^c}   \hat\bbeta \right\rangle_n \right|  + \lambda_n
  \] 
 holds in both cases. With it we can then see that 
  \[
  \left |\left\langle  \bD_{\sqrt \mathbf w_k}  \bX_{I_i,j},  \bD_{\sqrt \mathbf w_k} \bY_{I_i} - \bD_{\sqrt \mathbf w_k} \bX_{I_i} \hat\bbeta  \right\rangle_n \right| \leq  n  \left|\left\langle \bX_{I_i^c,j},  \bY_{I_i^c} - \bX_{I_i^c}   \hat\bbeta  \right\rangle_n \right|  + n \lambda_n
  \]
  since  $ \| \bD_{\sqrt \mathbf w_k}^2\|_F = \| \bD_{ \mathbf w_k}\|_F\leq n$ for all $k$ almost surely.
  Hence to show that $\hat{\boldsymbol\beta} $ satisfies KKT for sub-Lasso as well,  we need 
$\left|\left\langle \bX_{I_i^c,j},  \bY_{I_i^c} - \bX_{I_i^c}  \hat\bbeta  \right\rangle_n \right|   \leq \frac{\lambda_N}{n} - \lambda_n .
  $ We observe that the last inequality is in the statement of the lemma.
 \end{proof}

\subsection{Proof of Lemma \ref{lem:matrixineq}}

\begin{proof}
Note that by $\| \bD\| \leq \sqrt{n} \| \bD\|_F$  and the definition of the Frobenius norm we have that for two semi-positive definite matrices $\bD,\bG \in \mathbb{R}^{n\times n}$
\[
\left\| \bD^{-1}-\bG^{-1} \right\| = \sqrt{n}  \sqrt{\max\biggl\{ \left|\lambda_{\min} (\bD^{-1}-\bG^{-1})\right| , \left|\lambda_{\max}(\bD^{-1}-\bG^{-1}) \right|\biggl\}}.
\]
For all $i,j \geq 1$ and $i+j-1 \leq n$,
 Weyls inequalities  and Theorem III.2.8 of  \cite{B97}  we have
\[
\lambda_{i +j-1} (\bD +\bG) \leq \lambda_i(\bD) + \lambda_j(\bG).
\]
Utilizing that $\lambda_{\max}(-\bG^{-1})= - \lambda_{\min}(\bG^{-1})=- \lambda_{\min}^{-1}(\bG)$ and $\lambda_{\max}(\bD^{-1})=\lambda_{\max}^{-1}(\bD)$ 
 we have
\[
 \lambda_{\max}(\bD^{-1}-\bG^{-1})  \leq \left| \frac{1}{\lambda_{\max}(\bD)} - \frac{1}{\lambda_{\min}(\bG)} \right|,
\]
\[ 
 \lambda_{\min}(\bD^{-1}-\bG^{-1}) \leq \left| \frac{1}{\lambda_{\min}(\bD)}  - \frac{1}{\lambda_{\min}(\bG)}\right|.
\]
 Therefore,
\[
\left \| \bD^{-1}-\bG^{-1}\right\| \leq  \sqrt{n} \sqrt{\left| \frac{1}{\lambda_{\min}(\bD)} - \frac{1}{\lambda_{\min}(\bG)} \right|} \leq \sqrt{n} \sqrt{\frac{1}{\lambda_{\min}(\bD)}  + \frac{1}{\lambda_{\min}(\bG)}}.
\]

\end{proof}

\renewcommand\thesection{\Alph{section}}
\setcounter{section}{0}
\section{Supplementary Matterials}

\subsection{Proof of Lemma \ref{thm:1}}

\begin{proof} 
Note that part (ii) of this Lemma is an easy consequence of Lemma B.1 in \cite{BRT09}, hence we omit the details. For part(i) we proceed as follows.
From the definition,  we have  for every realization of random weights $\mathbf w_k$,
$$
\frac{1}{n} \sum_{l \in I_i } w_{k,l}(Y_l-\bX_l\hat{\boldsymbol\beta} _{i:k}(\lambda_N)  )^2  + 2\lambda_N \| \hat{\boldsymbol\beta} _{i:k}(\lambda_N) \| _1 \leq \frac{1}{n} \sum_{l \in I_i} w_{k,l}(Y_l-\bX_l{ {\boldsymbol \beta}} )^2 + 2\lambda_N \|\beta\|_1
$$ holds for any value of $\beta$. For simplicity of the notation we have suppressed  the dependence  $\hat{\boldsymbol\beta} _{i:k}(\lambda_N)$ of $\lambda_N$ and $k$. By using  $Y_i = X_l {\boldsymbol\beta}^* + \varepsilon_i$, and setting $\boldsymbol\beta={\boldsymbol\beta}^*$, previous becomes  equivalent to 
\begin{eqnarray*}
\frac{1}{n} \sum_{l \in I_i} w_{k,l}(\bX_l{\boldsymbol\beta}^*-\bX_l\hat{\boldsymbol\beta} _{i:k}(\lambda_N) )^2  &\leq& 2 \frac{ \|\hat{\boldsymbol\beta} _{i:k}(\lambda_N) - {\boldsymbol\beta}^* \|_1}{n} \max_{1 \leq j \leq p} \left( \sum_{l \in I_i } w_{k,l} |\varepsilon_i X_{lj}| \right) \\
&+&  2\lambda_N \| {\boldsymbol\beta}^*\|_1 + 2\lambda_N \| {\boldsymbol\beta} _i\|_1.
\end{eqnarray*}
Consider  the event $\mathcal{A}_q(\lambda_N)= \bigcap_{j=1}^p \left\{2  \frac{1}{n}\sum_{l \in I_i} w_{k,l} |\varepsilon_i X_{lj}| \leq \lambda_N - q\right\}$, for $q < \lambda_N$. Using the fact that $w_l\geq 1$ for all $l\in I_i$ we have the following
\begin{eqnarray*}
&&\frac{1}{n} \min_{1 \leq l \leq N} w_l \left\| \bX_{I_i}(\boldsymbol {\boldsymbol\beta}^* - \hat{\boldsymbol\beta} _{i:k}(\lambda_N)) \right \|_2^2  + \lambda_N  \|\hat{\boldsymbol\beta} _{i:k}(\lambda_N)- \boldsymbol {\boldsymbol\beta}^* \|_1 \\
&&\leq \frac{1}{n} \sum_{l \in I_i} w_l(\bX_l\boldsymbol {\boldsymbol\beta}^*-\bX_l\hat{\boldsymbol\beta} _{i:k}(\lambda_N) )^2  +  \lambda_N \|\hat{\boldsymbol\beta} _{i:k}(\lambda_N) - \boldsymbol {\boldsymbol\beta}^* \|_1 \\
&&\leq(2 \lambda_N -q) \|\hat{\boldsymbol\beta} _{i:k}(\lambda_N) - \boldsymbol {\boldsymbol\beta}^* \|_1+ 2\lambda_N \|\boldsymbol {\boldsymbol\beta}^*\|_1 - 2\lambda_N \| \hat{\boldsymbol\beta} _{i:k}(\lambda_N)\|_1 \\
&&\leq (4 \lambda_N -q) \sum_{j \in S} |{\hat {\boldsymbol \beta}}_{i,j} - \boldsymbol {\boldsymbol\beta}^*_j |
\end{eqnarray*}
which leads to the first conclusion. From the previous result  
$$
\lambda_N  \|\hat{\boldsymbol\beta} _{i:k}(\lambda_N)- \boldsymbol {\boldsymbol\beta}^* \|_1\leq (4 \lambda_N -q)\sum_{j \in S} |{\hat {\boldsymbol \beta}}_{i,j} - \boldsymbol {\boldsymbol\beta}^*_j |
\leq (4 \lambda_N -q)\sqrt{s} \sqrt{\sum_{j \in S} |{\hat {\boldsymbol \beta}}_{i,j} - \boldsymbol {\boldsymbol\beta}^*_j |^2} 
$$
leading to   $\hat{\boldsymbol\beta} _{i:k}(\lambda_N)- {\boldsymbol\beta}^*  \in \mathbb{C}(3,S)$. Using the Expected RE Condition on the set $\Omega_w$ we have 
$
\sqrt{\sum_{j \in S} |{\hat {\boldsymbol \beta}}_{i,j} - \boldsymbol {\boldsymbol\beta}^*_j |^2}  \leq {\mathbb E _{\mathbf w}\left\|   \tilde {\mathbf X} (\hat{\boldsymbol\beta} _{i:k}(\lambda_N) - {\boldsymbol\beta}^*) \right\|_2}/({e_n \sqrt{n}}).
$
Here $\mathbb E_{\mathbf w}$ denotes expectation taken with respect to the probability measure generated by $\mathbf w_k$.
Using Jensen's inequality for concave functions and independence of the weighting scheme $\mathbf w_k$ of vectors $\mathbf X_l$, we have
\begin{equation}
\begin{split}
\mathbb E _{\mathbf w}\left\|   \tilde{ \mathbf X} (\hat{\boldsymbol\beta} _{i:k}(\lambda_N) - {\boldsymbol\beta}^*) \right\|_2 &\leq \left\| \mathbb E _{\mathbf w} [  \tilde{ \mathbf X} ](\hat{\boldsymbol\beta} _{i:k}(\lambda_N) - {\boldsymbol\beta}^*) \right\|_2\\
&\leq \mathbb{E}_{\mathbf w} \|\sqrt{\mathbf w}\|_\infty  \left\|   { \mathbf X} _{I_i} (\hat{\boldsymbol\beta} _{i:k}(\lambda_N) - {\boldsymbol\beta}^*)\right\|_2.
\end{split}
\end{equation}
Combining previous inequalities we have
\[
\frac{1}{n} \min_{1 \leq l \leq N} w_l  \left\|  { \mathbf X} _{I_i} (\hat{\boldsymbol\beta} _{i:k}(\lambda_N) - {\boldsymbol\beta}^*)\right\|_2^2  \leq (4 \lambda_N -q)\sqrt{s}  \mathbb{E}_{\mathbf w} \|\sqrt{\mathbf w}\|_\infty  \frac{\left\|  { \mathbf X} _{I_i} (\hat{\boldsymbol\beta} _{i:k}(\lambda_N) - \boldsymbol {\boldsymbol\beta}^*) \right\|_2}{ e_n\sqrt{n}}.
\]
leading to
$
 \| { \mathbf X} _{I_i} (\boldsymbol {\boldsymbol\beta}^* - \hat{\boldsymbol\beta} _{i:k}(\lambda_N)) \|_2  \leq { (4 \lambda_N-q)   \sqrt{sn} } \mathbb{E}_{\mathbf w} \|\sqrt{\mathbf w}\|_\infty /{e_n (\min_{1 \leq l\leq N} w_l) }.
$
Hence, if we define $a_N$ as such that event $\{ \min_{1 \leq l \leq N} w_l \geq \mathbb{E}_{\mathbf w} \|\sqrt{\mathbf w}\|_\infty /a_N\}$ has probability close to 1, then 
\begin{equation}\label{eq:norm2}
 \left\|  { \mathbf X} _{I_i}  (\boldsymbol {\boldsymbol\beta}^* - \hat{\boldsymbol\beta} _{i:k}(\lambda_N)) \right\|_2^2  \leq \frac{ (4\lambda_N-q)^2  s n  }{ e_n^2 }  a_N^2.
\end{equation}
The size of the set $\mathcal{A}_q(\lambda_N)$ can be deduced from Lemma \ref{lem:a}  and is hence omitted. 
\end{proof}

\subsection{Proof of Lemma \ref{lemma:3}}

In light of the result of Lemma \ref{thm:1}, 
the proof    follows by repeating  exact  steps  of Theorem 7.2 of   \cite{BRT09}.
By contrast,
   with the difference that the loss function is now weighted least squares loss function; hence  we omit the proof.

\bibliographystyle{plainnat}
\bibliography{Bradic_Arxiv}{}

\begin{thebibliography}{35}
\providecommand{\natexlab}[1]{#1}
\providecommand{\url}[1]{\texttt{#1}}
\expandafter\ifx\csname urlstyle\endcsname\relax
  \providecommand{\doi}[1]{doi: #1}\else
  \providecommand{\doi}{doi: \begingroup \urlstyle{rm}\Url}\fi

\bibitem[Antognini and Giannerini(2007)]{A07}
A.~B. Antognini and S.~Giannerini.
\newblock Generalized pólya urn designs with null balance.
\newblock \emph{Journal of Applied Probability}, 44\penalty0 (3):\penalty0 pp.
  661--669, 2007.
\newblock ISSN 00219002.

\bibitem[Bach(2008)]{B08}
F.~Bach.
\newblock Bolasso: model consistent lasso estimation through the bootstrap.
\newblock \emph{CoRR}, abs/0804.1302, 2008.

\bibitem[Banerjee and Richardson(2013)]{B13}
M.~Banerjee and T.~Richardson.
\newblock Exchangeable bernoulli random variables and bayes? postulate.
\newblock \emph{Electron. J. Statist.}, 7:\penalty0 2193--2208, 2013.

\bibitem[Bhatia(1997)]{B97}
R.~Bhatia.
\newblock \emph{Matrix analysis}, volume 169 of \emph{Graduate Texts in
  Mathematics}.
\newblock Springer-Verlag, New York, 1997.
\newblock ISBN 0-387-94846-5.

\bibitem[Bickel et~al.(1997)Bickel, G{\"o}tze, and van Zwet]{BGZ97}
P.~J. Bickel, F.~G{\"o}tze, and W.~R. van Zwet.
\newblock Resampling fewer than {$n$} observations: gains, losses, and remedies
  for losses.
\newblock \emph{Statist. Sinica}, 7\penalty0 (1):\penalty0 1--31, 1997.
\newblock ISSN 1017-0405.
\newblock Empirical Bayes, sequential analysis and related topics in statistics
  and probability (New Brunswick, NJ, 1995).

\bibitem[Bickel et~al.(2009)Bickel, Ritov, and Tsybakov]{BRT09}
P.~J. Bickel, Y.~Ritov, and A.~B. Tsybakov.
\newblock Simultaneous analysis of lasso and {D}antzig selector.
\newblock \emph{Ann. Statist.}, 37\penalty0 (4):\penalty0 1705--1732, 2009.
\newblock ISSN 0090-5364.

\bibitem[Breiman(1996)]{B96}
L.~Breiman.
\newblock Bagging predictors.
\newblock \emph{Machine Learning}, 24\penalty0 (2):\penalty0 123--140, 1996.
\newblock ISSN 0885-6125.

\bibitem[B{\"u}hlmann and Yu(2002)]{BY02}
P.~B{\"u}hlmann and B.~Yu.
\newblock Analyzing bagging.
\newblock \emph{Ann. Statist.}, 30\penalty0 (4):\penalty0 927--961, 2002.
\newblock ISSN 0090-5364.

\bibitem[Bunea(2008)]{B07}
F.~Bunea.
\newblock \emph{Consistent selection via the Lasso for high dimensional
  approximating regression models}, volume Volume 3 of \emph{Collections},
  pages 122--137.
\newblock Institute of Mathematical Statistics, Beachwood, Ohio, USA, 2008.

\bibitem[Comminges and Dalalyan(2012)]{CD12}
L.~Comminges and A.~S. Dalalyan.
\newblock Tight conditions for consistency of variable selection in the context
  of high dimensionality.
\newblock \emph{Ann. Statist.}, 40\penalty0 (5):\penalty0 2667--2696, 10 2012.

\bibitem[Diaconis and Freedman(1980)]{DF80}
P.~Diaconis and D.~Freedman.
\newblock Finite exchangeable sequences.
\newblock \emph{Ann. Probab.}, 8\penalty0 (4):\penalty0 745--764, 08 1980.

\bibitem[Efron(1979)]{E79}
B.~Efron.
\newblock Bootstrap methods: another look at the jackknife.
\newblock \emph{Ann. Statist.}, 7\penalty0 (1):\penalty0 1--26, 1979.
\newblock ISSN 0090-5364.

\bibitem[{Fan} et~al.(2013){Fan}, {Han}, and {Liu}]{FHL13}
J.~{Fan}, F.~{Han}, and H.~{Liu}.
\newblock {Challenges of Big Data Analysis}.
\newblock \emph{ArXiv e-prints}, August 2013.

\bibitem[Fithian and Hastie(2014)]{FH14}
W.~Fithian and T.~Hastie.
\newblock Local case-control sampling: Efficient subsampling in imbalanced data
  sets.
\newblock \emph{The Annals of Statistics}, 42\penalty0 (5):\penalty0
  1693--1724, 10 2014.

\bibitem[Kleiner et~al.(2014)Kleiner, Talwalkar, Sarkar, and Jordan]{BLB12}
A.~Kleiner, A.~Talwalkar, P.~Sarkar, and M.~I. Jordan.
\newblock A scalable bootstrap for massive data.
\newblock \emph{Journal of the Royal Statistical Society: Series B (Statistical
  Methodology)}, 76\penalty0 (4):\penalty0 795--816, 2014.
\newblock ISSN 1467-9868.

\bibitem[Kontorovich and Ramanan(2008)]{K8}
L.~Kontorovich and K.~Ramanan.
\newblock Concentration inequalities for dependent random variables via the
  martingale method.
\newblock \emph{Ann. Probab.}, 36\penalty0 (6):\penalty0 2126--2158, 11 2008.
\newblock \doi{10.1214/07-AOP384}.
\newblock URL \url{http://dx.doi.org/10.1214/07-AOP384}.

\bibitem[Laurent and Massart(2000)]{ML00}
B.~Laurent and P.~Massart.
\newblock Adaptive estimation of a quadratic functional by model selection.
\newblock \emph{Ann. Statist.}, 28\penalty0 (5):\penalty0 1302--1338, 10 2000.

\bibitem[Lehmann and Casella(1998)]{L98}
E.~L. Lehmann and G.~Casella.
\newblock \emph{Theory of point estimation}.
\newblock Springer Texts in Statistics. Springer-Verlag, New York, second
  edition, 1998.
\newblock ISBN 0-387-98502-6.

\bibitem[Lounici(2008)]{L08}
K.~Lounici.
\newblock Sup-norm convergence rate and sign concentration property of lasso
  and dantzig estimators.
\newblock \emph{Electronic Journal of Statistics}, 2:\penalty0 90--102, 2008.

\bibitem[Lounici et~al.(2011)Lounici, Pontil, van~de Geer, and Tsybakov]{L11}
K.~Lounici, M.~Pontil, S.~van~de Geer, and A.~B. Tsybakov.
\newblock Oracle inequalities and optimal inference under group sparsity.
\newblock \emph{Ann. Statist.}, 39\penalty0 (4):\penalty0 2164--2204, 08 2011.

\bibitem[McDonald et~al.(2010)McDonald, Hall, and Mann]{MHM10}
R.~McDonald, K.~Hall, and G.~Mann.
\newblock Distributed training strategies for the structured perceptron.
\newblock In \emph{Human Language Technologies: The 2010 Annual Conference of
  the North American Chapter of the Association for Computational Linguistics},
  HLT '10, pages 456--464, Stroudsburg, PA, USA, 2010. Association for
  Computational Linguistics.
\newblock ISBN 1-932432-65-5.

\bibitem[Meinshausen and B{\"u}hlmann(2010)]{MB11}
N.~Meinshausen and P.~B{\"u}hlmann.
\newblock Stability selection.
\newblock \emph{J. R. Stat. Soc. Ser. B Stat. Methodol.}, 72\penalty0
  (4):\penalty0 417--473, 2010.
\newblock ISSN 1369-7412.

\bibitem[{Minsker}(2013)]{MI13}
S.~{Minsker}.
\newblock {Geometric Median and Robust Estimation in Banach Spaces}.
\newblock \emph{ArXiv e-prints}, August 2013.

\bibitem[{Minsker} et~al.(2014){Minsker}, {Srivastava}, {Lin}, and
  {Dunson}]{M14}
S.~{Minsker}, S.~{Srivastava}, L.~{Lin}, and D.~B. {Dunson}.
\newblock {Robust and Scalable Bayes via a Median of Subset Posterior
  Measures}.
\newblock \emph{ArXiv e-prints}, March 2014.

\bibitem[Permantle and Peres(2014)]{P14}
R.~Permantle and Y.~Peres.
\newblock Concentration of lipschitz functionals of determinantal and other
  strong rayleigh measures.
\newblock \emph{Combinatorics, Probability and Computing}, 23:\penalty0
  140--160, 1 2014.
\newblock ISSN 1469-2163.
\newblock \doi{10.1017/S0963548313000345}.
\newblock URL \url{http://journals.cambridge.org/article_S0963548313000345}.

\bibitem[Politis et~al.(2001)Politis, Romano, and Wolf]{P01}
D.~N. Politis, J.~P. Romano, and M.~Wolf.
\newblock {On the asymptotic theory of subsampling}.
\newblock \emph{Statist. Sinica}, 11\penalty0 (4):\penalty0 1105--1124, 2001.

\bibitem[Pr{\ae}stgaard and Wellner(1993)]{PW93}
J.~Pr{\ae}stgaard and J.~A. Wellner.
\newblock Exchangeably weighted bootstraps of the general empirical process.
\newblock \emph{Ann. Probab.}, 21\penalty0 (4):\penalty0 2053--2086, 1993.
\newblock ISSN 0091-1798.

\bibitem[Raskutti et~al.(2010)Raskutti, Wainwright, and Yu]{RWY11}
G.~Raskutti, M.~J. Wainwright, and B.~Yu.
\newblock Restricted eigenvalue properties for correlated {G}aussian designs.
\newblock \emph{J. Mach. Learn. Res.}, 11:\penalty0 2241--2259, 2010.
\newblock ISSN 1532-4435.

\bibitem[Samworth(2003)]{S03}
R.~Samworth.
\newblock A note on methods of restoring consistency to the bootstrap.
\newblock \emph{Biometrika}, 90\penalty0 (4):\penalty0 985--990, 2003.

\bibitem[Shah and Samworth(2013)]{SS12}
R.~D. Shah and R.~J. Samworth.
\newblock Variable selection with error control: another look at stability
  selection.
\newblock \emph{J. R. Stat. Soc. Ser. B. Stat. Methodol.}, 75\penalty0
  (1):\penalty0 55--80, 2013.
\newblock ISSN 1369-7412.

\bibitem[Tsybakov(2009)]{T09}
A.~B. Tsybakov.
\newblock \emph{Introduction to nonparametric estimation}.
\newblock Springer Series in Statistics. Springer, New York, 2009.
\newblock ISBN 978-0-387-79051-0.
\newblock Revised and extended from the 2004 French original, Translated by
  Vladimir Zaiats.

\bibitem[van~de Geer and B{\"u}hlmann(2009)]{vGB09}
S.~A. van~de Geer and P.~B{\"u}hlmann.
\newblock On the conditions used to prove oracle results for the {L}asso.
\newblock \emph{Electron. J. Stat.}, 3:\penalty0 1360--1392, 2009.
\newblock ISSN 1935--7524.

\bibitem[Wang et~al.(2014)Wang, Peng, and Dunson]{D14}
X.~Wang, P.~Peng, and D.~B Dunson.
\newblock Median selection subset aggregation for parallel inference.
\newblock In Z.~Ghahramani, M.~Welling, C.~Cortes, N.D. Lawrence, and K.Q.
  Weinberger, editors, \emph{Advances in Neural Information Processing Systems
  27}, pages 2195--2203. Curran Associates, Inc., 2014.
\newblock URL
  \url{http://papers.nips.cc/paper/5328-median-selection-subset-aggregation-for-parallel-inference.pdf}.

\bibitem[{Zhang} et~al.(2012){Zhang}, {Duchi}, and {Wainwright}]{ZDW12}
Y.~{Zhang}, J.~C. {Duchi}, and M.~{Wainwright}.
\newblock {Comunication-Efficient Algorithms for Statistical Optimization}.
\newblock \emph{ArXiv e-prints}, September 2012.

\bibitem[{Zhang} et~al.(2013){Zhang}, {Duchi}, and {Wainwright}]{ZDW13}
Y.~{Zhang}, J.~C. {Duchi}, and M.~J. {Wainwright}.
\newblock {Divide and Conquer Kernel Ridge Regression: A Distributed Algorithm
  with Minimax Optimal Rates}.
\newblock \emph{ArXiv e-prints}, May 2013.

\end{thebibliography}

\end{document}